\documentclass[11pt, reqno]{amsart}
\usepackage{latexsym}
\usepackage{color,dsfont}
\usepackage[dvipsnames]{xcolor}
\usepackage{amssymb,amsfonts,amsmath,mathrsfs}
\usepackage{cases}

\numberwithin{equation}{section}

\DeclareMathOperator{\sgn}{sgn}

\renewcommand{\P}{\mathbb{P}}

\DeclareMathOperator{\E}{\mathbb{E}}

\newtheorem{lem}{Lemma}[section]
\newtheorem{remark}[lem]{Remark}
\newtheorem{prop}[lem]{Proposition}
\newtheorem{thm}[lem]{Theorem}

\newtheorem{cor}[lem]{Corollary}
\newtheorem{conj}[lem]{Conjecture}

\usepackage{etoolbox}

\makeatletter
\patchcmd{\@starttoc}
  {\linespacing\@plus\linespacing}
  {6pt}
  {}{}
\makeatother

\begin{document}

\title{Small values of the Riemann zeta-function near the critical line}
\author{Hung M. Bui, Alexandra Florea, Andr\'e Heycock}
\address{Department of Mathematics, University of Manchester, Manchester M13 9PL, UK}
\email{hung.bui@manchester.ac.uk}
\address{UC Irvine, Mathematics Department, Rowland Hall, Irvine 92697, USA}
\email{floreaa@uci.edu}
\address{Department of Mathematics, University of Manchester, Manchester M13 9PL, UK}
\email{andreheycock@gmail.com}
\date{}

\allowdisplaybreaks

\begin{abstract}
Assuming the Riemann Hypothesis, we study negative moments and
small values of the Riemann zeta-function to the right of the
critical line. For every fixed $0<k<1/2$, we establish the
conjecturally sharp bound
$$ \frac1T\int_T^{2T}
 \Big|\zeta(\tfrac12+\alpha+it)\Big|^{-2k}\,dt
 \ll \alpha^{-k^2},
$$
for $\alpha=o(1)$ and $\alpha\gg 1/\log T$. We also obtain almost sharp upper bounds slightly to the left of $1/\log T$. This improves the upper bounds of the first two authors 
for shifts as small as $1/\log T$,
confirming a conjecture of Gonek in this range.
A separate bound for the moment at $k=1/2$ yields
an improved bound for cancellation in the M\"obius function. Namely, we show that
$$
\sum_{n\leq x}\mu(n)
\ll
\sqrt{x}\exp \bigg(C\sqrt{\frac{\log x}{\log\log x}}\bigg)
$$
for some $C>0$, improving on previous results of Soundararajan, Balazard and de Roton, and of the first two authors. For $t$ chosen uniformly from $[T,2T]$, we also prove that
\[
 \min_{|h|\leq1}
 \Big|\zeta\Big(\tfrac12+\alpha+i(t+h)\Big)\Big|
 =\alpha^{1+o(1)}
\]
in probability, provided that $\alpha \ll (\log T)^{-c}$ for some fixed $c >0$,
 and
$\alpha\log T/\log\log T\to\infty$. This agrees with the
leading order predicted by a minimum analogue of the
Fyodorov--Hiary--Keating Conjecture. Further results include
bounds for minima on intervals of logarithmically growing
length, for negative moments in almost all short intervals, and
an extension of the range of sharp negative-moment bounds
for $k>1/2$.
A main ingredient in the proof of some of these results is a Gaussian large-deviation upper bound
for $\log|\zeta(1/2+\alpha+it)|^{-1}$,  obtained by building on work of Arguin, Bailey,
and Roberts. 

\end{abstract}
\maketitle 

\setcounter{tocdepth}{1}
\tableofcontents

\section{Introduction}

In this paper, we study small values of the Riemann zeta--function and its negative moments. The questions we are interested in are how frequently
small values occur, and how small they can be on a typical
short interval, as the distance to the critical line tends to zero.

Throughout the paper, $t$ is chosen uniformly from $[T,2T]$, and $\P$
denotes the corresponding probability.

\subsection{Negative moments near the critical line}

We first state our sharp bounds for negative moments.  For $k>0$ and $\alpha>0$, let
\[
 I_{-k}(\alpha,T)
 :=\frac1T\int_T^{2T}
 \Big|\zeta\Big(\tfrac12+\alpha+it\Big)\Big|^{-2k}\,dt.
\]
Gonek~\cite{Gonek} conjectured that for every fixed $k>0$ and $1/\log T \leq \alpha \leq 1$,
\[
 I_{-k}(\alpha,T)\asymp_k\alpha^{-k^2}.
\]
For smaller shifts, his conjecture predicts that for $0 < \alpha \leq 1/\log T$,
\[
 I_{-k}(\alpha,T)\asymp_k
 \begin{cases}
  (\log T)^{k^2}&0<k<1/2,\\[2pt]
  (\log T)^{1/4}\log\dfrac{e}{\alpha\log T}&k=1/2.
 \end{cases}
\]
Assuming the Riemann Hypothesis (RH), Gonek proved lower bounds of the conjectured
order for all $k>0$ when $1/\log T\leq\alpha\leq1$, and
for $0<k<1/2$ when $0<\alpha\leq1/\log T$.
Focusing on upper bounds for the negative moments, Bui and Florea~\cite{BF} proved non-trivial upper
bounds for certain ranges of shifts, and asymptotic formulas
when the shift is sufficiently large. More precisely, they obtained asymptotic formulas under RH for every fixed
$k,\varepsilon,C>0$ when
\[
 \alpha\geq\max\Big\{
 \frac{C(\log\log T)^{4/k+\varepsilon}}
        {(\log T)^{1/(2k)}},
 \frac{(1+\varepsilon)\log\log T}{2\log T}
 \Big\}.
\]
Their results also give
upper bounds for smaller shifts, but these do not have the
conjectured order throughout that range.

In this paper, we improve the upper bounds for small shifts.
In particular, for $0<k<1/2$, we establish Gonek's conjectured
upper bound down to shifts of order $1/\log T$. More specifically, we prove the following. 

\begin{thm}\label{negativemom}
Assume RH. For every $0<k<1/2$ and $0<c<1$, we have the following bounds:
\begin{numcases}{I_{-k}(\alpha,T)\ll}
\alpha^{-k^2}
& $\text{if } 1\ll\alpha\log T\ll\log\log T$,
\label{firstrange}\\
(\log T)^{O(1)}
& $\text{if } (\log\log\log T)^{-c}\ll\alpha\log T\ll1$.
\label{secondrange}
\end{numcases}

\end{thm}

We note that the improvement includes the range
$1 \ll \alpha\log T=o(\log\log T)$, which is not covered
by the asymptotic formula in~\cite{BF}. Combining our result
with the larger-shift results in \cite{BF} gives the
conjectured upper bound for all $\alpha=o(1)$ satisfying
$\alpha\gg 1/\log T$, for every fixed $0<k<1/2$. Moreover, combining the upper bound in this region with the corresponding lower bound obtained by Gonek \cite{Gonek} shows that when $1/\log T \ll \alpha =o(1)$ and $0<k<1/2$, we have
$$ I_{-k}(\alpha, T) \asymp \alpha^{-k^2}. $$

To the left of $1/\log T$, the upper bound in \eqref{secondrange} is almost sharp, up to some $\log T$ factors. This improves on the bound $\exp (O( \log \log T (\log \log \log T)^2 ))$ in this regime obtained in \cite{BF} (see \cite[(1.8)]{BF}).

The next theorem bounds the moment $I_{-1/2}(\alpha,T)$, again improving on \cite{BF}.
\begin{thm}\label{negative1/2}
Assume RH. For every $0 < \varepsilon<1$ and $0<c<1$, there exist 
$C=C(c)>0$ and $C_1=C_1(\varepsilon)>0$ such that the following bound holds:
\begin{numcases}{I_{-1/2}(\alpha,T)\ll}
\exp\bigg(C\log\log T
 \Big(\frac{\log\log\log T}
          {\log\log\log\log T}\Big)^2\bigg)\nonumber\\
\qquad\qquad\qquad\qquad \text{if } (\log\log\log T)^{-c}\ll\alpha\log T\ll1,
\label{firstrange'}\\
\exp\big((\log T)^{\varepsilon}\big)\nonumber\\
\qquad\qquad\qquad\qquad  \text{if } C_1 (\log\log T)^{-1}\leq \alpha\log T
   \ll(\log\log\log T)^{-c}.
\label{secondrange'}
\end{numcases}

\end{thm}
Note that by comparison, Theorem $1.2$ in \cite{BF} gives in the range \eqref{firstrange'}
$$I_{-1/2}(\alpha, T) \ll \exp \Big( \frac{(4+\varepsilon) \log T \log \log \log T}{\log \log T}\Big),$$ and in the range \eqref{secondrange'} it gives $\exp( \varepsilon \log T)$.

 Finally, when $k>1/2$, we have the following theorem. 

 \begin{thm}\label{k>1/2}
Assume RH. For every $k>1/2$, there exists $C=C(k)>0$ such that
uniformly for 
\begin{equation*}
\alpha=o(1), \qquad \alpha\geq\frac{C}{(\log T)^{1/(2k)}\log\log T},
\end{equation*}
we have
\begin{equation*}\label{eq:above-half-sharp-moment}
 I_{-k}(\alpha,T)\ll \alpha^{-k^2}.
\end{equation*}
\end{thm}
We note that one can study the corresponding negative moment problem in random matrix theory as well.
The moment calculations of Keating and Snaith~\cite{KS-moments}
give order $N^{k^2}$ for negative $2k$-th moments of
characteristic polynomials of unitary matrices when
$0<k<1/2$. 
Negative moments were also studied by
Fyodorov~\cite{Fyodorov-negative} and Strahov and
Fyodorov~\cite{SF-negative}. Berry and
Keating~\cite{BK-negative} predicted that high negative
moments are governed by rare clusters of eigenvalues.
Related calculations were carried out by Fyodorov and
Keating~\cite{FK-negative} and Forrester and
Keating~\cite{Forrester-Keating}. For unitary families, these predictions lead to transitions in the asymptotic formulas at
$k=1/2,3/2,\ldots$, with logarithmic corrections at the
transition points. For the Riemann zeta-function, they suggest corrections to
Gonek's original prediction when $\alpha\leq1/\log T$ and
$k\geq3/2$ (see also \cite[Section~1]{BF}).
For integral $k$, negative moments can also be viewed as
special cases of averages of ratios of shifted
zeta-functions. We refer to Farmer~\cite{Farmer-ratios},
Conrey, Farmer, and Zirnbauer~\cite{CFZ-ratios}, and Conrey
and Snaith~\cite{CS-ratios} for conjectures and applications
in this direction. 

\subsection{Averages of the M\"obius function}
The main application of our negative moment bounds lies in obtaining cancellation in averages of the M\"obius function. Assuming RH,
Soundararajan~\cite{Sound-mobius} proved that
\begin{equation}\label{Soundbound}
 \sum_{n\leq x}\mu(n)
 \ll \sqrt{x}\exp\!\Big(
       \sqrt{\log x}(\log\log x)^{14}\Big).
\end{equation}
Balazard and de Roton~\cite{BDR} optimized his argument,
replacing the exponent $14$ by $5/2+\varepsilon$,
which is described in~\cite{BDR} as the limit of this method. For earlier bounds on $\sum_{n\leq x}\mu(n)$ under RH,
we refer the reader to Littlewood \cite{Littlewood}, Landau \cite{Landau},
Titchmarsh \cite{Titchmarsh}, and
Maier and Montgomery \cite{MM}.
Bui and Florea~\cite{BF} used negative-moment
estimates for $\zeta(s)$ to study averages of the generalized
M\"obius functions $\mu_k$, the Dirichlet coefficients of $\zeta(s)^{-k}$. They also obtained asymptotic formulas for non-integer $k$ and
cancellation bounds for integer $k$.

Following the method in \cite{BF}, we use
Theorems~\ref{negative1/2} and \ref{k>1/2} to control the mean of $|1/\zeta(s)^k|$
close to the critical line. This removes the $\log\log x$ factor 
 in the exponent of \eqref{Soundbound}.

\begin{thm}\label{partM}
Assume RH. For every $k\in\mathbb{N}$, there exists $C=C(k)>0$ such that
\begin{align*}
&\sum_{n \leq x} \mu_k(n)\ll \sqrt{x} \exp\bigg( C\Big(\frac{\log x}{\log\log x}\Big)^{k/(k+1)}  \bigg).
\end{align*}
\end{thm}

By comparison, Bui and Florea~\cite[Theorem~1.5]{BF} established the bound
\[
 \sqrt{x}\exp\!\Big(
   (\log x)^{k/(k+1)}
   (\log\log x)^{7/(k+1)+\varepsilon}
 \Big)
\]
for every $k\in\mathbb{N}$. We remark that our arguments also work for non-integer $k$ (see Theorem \ref{partM'} for a more general statement).

\subsection{The minimum on short intervals}

We next consider the minimum on a typical interval of bounded
length. On the critical line, any interval containing a zero
has minimum zero. We therefore shift a small distance
$\alpha>0$ to the right and consider
\begin{equation*}
\label{mt}
 M_\alpha(t):=\min_{|h|\leq1}
\Big|\zeta\Big(\tfrac12+\alpha+i(t+h)\Big)\Big|.
\end{equation*}
We are interested in small shifts $\alpha \to 0$. 
The size of $M_{\alpha}(t)$
measures how small the zeta-function becomes in a typical short
interval at distance $\alpha$ from the critical line.

We determine the leading logarithmic order of this minimum
in the following range.

\begin{thm}
\label{minimum_first_order}
Assume RH. Fix $0<c<1$. Suppose that
\[
0<\alpha\ll (\log T)^{-c},
\qquad
\frac{\alpha\log T}{\log\log T}\longrightarrow\infty.
\]
For every $\varepsilon,\delta>0$ and
for all sufficiently large $T$, we have
\[
\mathbb P\big(
\alpha^{1+\delta}
\leq M_\alpha(t)
\leq\alpha^{1-\delta}
\big)\geq1-\varepsilon.
\]
\end{thm}

Note that in Theorem \ref{prop:main} we prove a lower bound result with the expected leading order term for the minimum in intervals of length $(\log T)^{\gamma}$, and moreover, we obtain a more specific $y$-dependent tail estimate. Theorem \ref{minimum_first_order} follows by combining Theorem \ref{expanded_range} and Theorem \ref{prop:minimum-upper}, and corresponds to the case $\gamma=0$ in Theorem \ref{prop:main}.

Theorem~\ref{minimum_first_order} gives the leading order
predicted by Conjecture~\ref{conj_specialcase} below. We
describe the connection with the maximum problem and the
heuristic behind this prediction.

The extreme values of the Riemann zeta-function in short intervals
have attracted considerable attention, in part because of their
connection with random matrix theory and logarithmically correlated
random fields. Fyodorov, Hiary, and Keating~\cite{FHK}, and
Fyodorov and Keating~\cite{FK} predicted the size and fluctuations
of the maximum on a typical interval of bounded length on the
critical line. Their prediction arose from modelling the zeta-function by characteristic
polynomials of random unitary matrices and applying the
freezing-transition heuristic from statistical mechanics~\cite{FHK,FK}. 
They predicted that
for every $\varepsilon>0$, there exists $C=C(\varepsilon)>0$ such that
for all sufficiently large $T$,
\[
 \P\Big(
 \Big|\max_{|h|\leq1}
       \log\Big|\zeta\Big(\tfrac12+i(t+h)\Big)\Big|
       -\log\log T+\frac34\log\log\log T\Big|
       \leq C\Big)\geq1-\varepsilon.
\]
Najnudel~\cite{Najnudel} proved
the first-order asymptotic under RH, and Arguin, Belius,
Bourgade, Radziwi\l\l, and Soundararajan~\cite{ABBRS} proved it
unconditionally. Namely, in probability, they showed that
\[
\max_{|h|\leq1}\log\Big|\zeta\Big(\tfrac12+i(t+h)\Big)\Big|
 =\big(1+o(1)\big)\log\log T.
\]
Najnudel's first-order upper bound is itself unconditional;
RH is used for the lower bound. Harper~\cite{Harper-partition} obtained an upper bound with the
correct second-order term; subsequent work of Arguin, Bourgade,
and Radziwi\l\l{}~\cite{ABR1} improved the error term and
established the predicted upper-tail bound. 
A matching lower bound was subsequently established by the same
authors~\cite{ABR2}.
For an account of the earlier developments,
see Harper's Bourbaki survey~\cite{Harper-survey}.

The Fyodorov--Hiary--Keating (FHK) Conjecture has extensions to intervals whose length is a
power of $\log T$, with different predictions according to whether
the intervals shrink or grow. For shrinking intervals, with
$-1<\gamma<0$, Fyodorov and Keating~\cite{FK} formulated a
mesoscopic version of the conjecture. Arguin and Hamdan~\cite{AH} proved the sharp
upper bound, including the second-order coefficient $3/4$.
For growing intervals, Arguin, Ouimet, and Radziwi\l\l{}~\cite{AOR}
proposed the following extension: for every fixed $\gamma>0$ and $t$
chosen uniformly from $[T,2T]$,
\[
\max_{|h|\leq(\log T)^\gamma}
\log\Big|\zeta\Big(\tfrac12+i(t+h)\Big)\Big|
=
\sqrt{1+\gamma}\log\log T
-\frac{1}{4\sqrt{1+\gamma}}\log\log\log T
+O_{\mathbb P}(1),
\]
where the error is bounded in probability. They established the
leading-order asymptotic unconditionally for $0<\gamma\leq3$;
for $\gamma>3$, their upper bound assumes RH, while their lower
bound remains unconditional. Arguin and Bailey~\cite{AB}
subsequently proved the corresponding upper bound with the
conjectured second-order term and an error bounded in probability
for $0<\gamma<3$. Note that the conjectured second-order coefficient
tends to $1/4$ as $\gamma\downarrow0$ through positive values,
whereas it equals $3/4$ for fixed-length intervals.

The random-matrix analogue of the FHK Conjecture concerns the
characteristic polynomial of an $N\times N$ Haar-distributed
unitary matrix $U_N$. Fyodorov, Hiary, and Keating~\cite{FHK}
predicted that
\[
 \max_{0\leq\theta<2\pi}
 \log\big|\det(I-e^{-i\theta}U_N)\big|
 =\log N-\frac34\log\log N+O_{\mathbb P}(1),
\]
corresponding to the zeta-function prediction under the
identification $N\asymp\log T$. Arguin, Belius, and
Bourgade~\cite{ABB-CUE} established the leading order, and
Paquette and Zeitouni~\cite{PZ-CUE} obtained the second-order
term.

To formulate an FHK-type conjecture for the minimum, we revisit the original conjecture of Fyodorov, Hiary and Keating. Their conjecture can be motivated by
approximating $\log|\zeta|$ by a sum over the primes; see
Harper~\cite{Harper-model}. For $t$ chosen uniformly from
$[T,2T]$, one models the phases $p^{-it}$ by independent random
variables $X(p)$, uniformly distributed on the unit circle.
The resulting model is
\[
 \log\Big|\zeta\Big(\tfrac12+i(t+h)\Big)\Big|
 \approx X_0(h):=\Re\sum_{p\leq T}\frac{X(p)}{p^{1/2+ih}}.
\]
Its covariance satisfies
\[
 \operatorname{Cov}(X_0(h),X_0(h'))
 =\frac12\sum_{p\leq T}\frac{\cos((h-h')\log p)}p
 =\frac12\log\frac1{|h-h'|+1/\log T}+O(1).
\]
In particular, the variance is $\frac12\log\log T+O(1)$,
and the correlations decay logarithmically for separations
between $1/\log T$ and $1$. Decomposing the sum into blocks
of comparable width in $\log\log p$ gives increments of
comparable variance. At two nearby points, the increments
from the smaller primes are strongly correlated, while those
from sufficiently large primes decorrelate.

A similar heuristic applies to the minimum of zeta to the
right of the critical line. We write
$$W_{\alpha}(t): = \log \big| \zeta (  \tfrac{1}{2}+\alpha+it) \big|^{-1}.$$
Suppose that $\alpha=o(1)$ and
$\alpha\log T\gg 1$. We model
\[
 W_{\alpha}(t+h)=\log\Big|\zeta\Big(\tfrac12+\alpha+i(t+h)\Big)\Big|^{-1}
 \approx-X_\alpha(h):=-\Re\sum_{p\leq T}
 \frac{X(p)}{p^{1/2+\alpha+ih}}.
\]
 Hence
\[
 \operatorname{Var}(X_\alpha(h))
 =\frac12\sum_{p\leq T}\frac1{p^{1+2\alpha}}
 =\frac12\log\frac1\alpha+O(1),
\]
and
\[
 \operatorname{Cov}(X_\alpha(h),X_\alpha(h'))
 =\frac12\log\frac1{\alpha+|h-h'|}+O(1).
\]
Thus the model retains the same logarithmic correlation
structure, with $\alpha$ in place of
$1/\log T$.

Since the joint distribution of the $X(p)$ is unchanged upon
replacing every $X(p)$ by $-X(p)$, the fields $X_\alpha$ and
$-X_\alpha$ have the same distribution. Applying the preceding
heuristic with $\log T$ replaced by $1/\alpha$, we are
therefore led to conjecture the following.
\begin{conj}
\label{conj_specialcase}
Suppose that
  $$
\alpha=o(1),\qquad \alpha\log T\gg1.
$$
For every $\varepsilon>0$, there exists $C=C(\varepsilon)>0$ 
such that for all sufficiently large $T$, we have
\[
\mathbb P \bigg(
\Big|
\log M_{\alpha}(t)
+\log\frac1\alpha
-\frac34\log\log\frac1\alpha
\Big|\leq C
\bigg)\geq1-\varepsilon.
\]
\end{conj}

In particular, the conjectured scale
of the minimum is $\alpha(\log\frac{1}{\alpha})^{3/4}$, up to
multiplicative fluctuations whose logarithms are bounded
in probability. (See the statement of Conjecture \ref{min_conj} for a generalization of this problem in intervals of logarithmically growing size.)

\subsection{Large deviations}
The main analytic input for the negative-moment estimates is
a precise upper bound for the probability that $-\log|\zeta|$
is large. On the positive side,
Selberg's central limit theorem  \cite{selberg} describes
$\log|\zeta(1/2+it)|$ on the scale $\sqrt{\log\log T}$ (For an extension of Selberg's CLT, see also \cite{vassaux}.)
Moments, however, depend on much rarer values of size comparable
to $\log\log T$. Soundararajan~\cite{Sound-moments} used
large-deviation estimates to obtain, under RH, positive-moment
upper bounds with exponent $k^2+\varepsilon$ on $\log T$.
Harper~\cite{Harper-moments} removed this $\varepsilon$ loss,
also under RH. Arguin and Bailey~\cite{AB} proved, unconditionally,
that for every fixed $0<\kappa<2$ and $V=\kappa\log\log T$,
\[
 \P\Big(\log\big|\zeta(\tfrac12+it)\big|>V\Big)
 \ll \frac1{\sqrt{\log\log T}}
                  \exp\Big(-\frac{V^2}{\log\log T}\Big).
\]
Their estimate recovers, for $0<\kappa<2$, the sharp positive-moment
upper bounds previously established by Heap, Radziwi\l\l{}, and
Soundararajan~\cite{HRS}. More recently,
Arguin, Bailey, and Roberts~\cite{ABR} established the same
large-deviation upper bound under RH for every fixed $\kappa>0$,
recovering Harper's sharp upper bounds for all fixed positive
moments. Their method combines the recursive  scheme
of~\cite{ABR1} with the methods of
Soundararajan~\cite{Sound-moments} and Harper~\cite{Harper-moments}. 

Lower bounds for large deviations have also been studied.
Arguin and
Bailey~\cite{AB-lower} obtained an unconditional lower bound
of the conjectured Gaussian order at
$1/2+c/\log T$, provided that $c$ is sufficiently large
in terms of $\kappa$. More recently, Arguin and
Creighton~\cite{AC} proved the corresponding unconditional
lower bound on the critical line, for every fixed $\kappa>0$
and $V\sim\kappa\log\log T$. For large-deviation results in other families of $L$-functions,
see, for example, Arguin and Creighton~\cite{AC-Dirichlet}
for Dirichlet $L$-functions, and Creighton~\cite{Creighton-twists}
for quadratic twists of elliptic curves.

Here we study the corresponding large deviations of
$-\log|\zeta(1/2+\alpha+it)|$. The shift changes the relevant
variance to $\frac12\log\frac1{\alpha}$, and the natural regime is
therefore $V=\kappa\log\frac1{\alpha}$. The main additional difficulty
is the contribution of zeros in the explicit formula. For small
shifts this contribution is significant. Our estimates preserve the Gaussian upper bound, down to $\alpha=c/\log T$ for every fixed $c>0$ and
$0<\kappa<1/2$. For larger shifts we obtain the same bound for
every fixed $\kappa>0$. The precise statement is as follows.

\begin{thm}
    \label{mainthm}
Assume RH. Fix $0<\kappa_0<\kappa_1<1/2$. Suppose that
\[
\alpha=o(1),
\qquad \alpha\log T \gg  1.
\]Then, uniformly for \(\kappa_0\leq\kappa\leq\kappa_1\) with \(V=\kappa\log\frac{1}{\alpha}\), we have
\[
\mathbb P \big(
W_\alpha(t)>V
\big)
\ll 
\frac{1}{\sqrt{\log(1/\alpha)}}
\exp\!\Big(-\frac{V^2}{\log(1/\alpha)}\Big).
\]
Moreover, for every \(\varepsilon>0\) and
\(0<\kappa_0<\kappa_1<\infty\), the same bound holds uniformly for
\(\kappa_0\leq\kappa\leq\kappa_1\), provided that
\[
\alpha\log T\gg\Big(\log\frac1\alpha\Big)^{12+\varepsilon}.
\]
\end{thm}

The proof separates larger shifts from smaller shifts, which require
a recursive argument. For larger deviation levels at small shifts, we obtain the
following weaker estimate. 

\begin{thm}\label{bigk}
Assume RH. Fix $\varepsilon,\varepsilon_1>0$, and $\kappa_1\geq1/2$. Suppose that
\[
1 \ll \alpha\log T
\ll \Big(\log\frac1\alpha\Big)^{12+\varepsilon}.
\]
Then, uniformly for $1/2\leq\kappa\leq\kappa_1$ with
$V=\kappa\log\frac{1}{\alpha}$, we have
\[
\mathbb P\big(
W_\alpha(t)>V
\big)
\ll
\exp\bigg(
-V+\Big(\frac14+\varepsilon_1\Big)\log\frac1\alpha
\bigg)
=\alpha^{\kappa-1/4-\varepsilon_1}.
\]
\end{thm}

\subsection{Further applications}
The lower bound for the minimum in Theorem \ref{minimum_first_order} (or, to be more precise, Theorem \ref{prop:main})
also gives bounds for negative moments on almost all short
intervals. The following result is the analogue of Corollary
$1.4$ in \cite{AB} and Corollary $3.2$ in \cite{ABR}.
Here $k_\gamma$ is as in \eqref{kgamma}. 

 \begin{thm}
      \label{momshortint}
      Assume RH. Fix $k>0$, $\varepsilon>0$,
      $0<c<1$, $M>0$, $\gamma \geq 0$ and let $H=(\log T)^{\gamma}$. Suppose that
\begin{equation*}
 0<\alpha\ll(\log T)^{-c}
 \qquad\text{and}\qquad
\alpha\log T\gg\Big(\log\frac1\alpha\Big)^{12+\varepsilon}.
\end{equation*}
Then, uniformly for
\begin{equation*}
 1\leq A\ll
 \alpha^{-M},
\end{equation*}
 we have 
      \begin{equation*} \label{subcritical} \frac{1}{2H} \int_{|h| \leq H} \Big|\zeta\Big(\tfrac12+\alpha+i(t+h)\Big)\Big|^{-2k} dh \ll  A  \alpha^{-k^2},\end{equation*}
      for all $t \in [T,2T]$ except on a subset of measure $\ll T/A$. 

      For $k > k_{\gamma}$, we also have
\begin{align}\label{supercritical}
&\frac1{2H}\int_{|h|\leq H}
\Big|\zeta\Big(\tfrac12+\alpha+i(t+h)\Big)\Big|^{-2k}\,dh \ll
 \frac{A^{k/k_\gamma}\alpha^{-(2kk_\gamma-k_\gamma^2)}}
      {(k-k_\gamma)\sqrt{\log(1/\alpha)}}
 \Big( \log \frac{1}{\alpha}  \Big)^{C(k-k_{\gamma})},
\end{align}
for all $t \in [T,2T]$ except on a subset of measure $\ll T/A$.
  \end{thm}
For positive moments on the critical
line in typical intervals, Arguin, Ouimet, and Radziwi\l\l{}~\cite[Theorem~1.1]{AOR}
determined the leading powers of $\log T$, with a transition
from quadratic to linear dependence on the moment parameter
at $k=\sqrt{1+\gamma}$ when $\gamma\geq0$. Their upper bounds
are unconditional for $0\leq\gamma\leq3$ and assume RH for
$\gamma>3$; the lower bounds are unconditional.
Arguin and Bailey~\cite[Corollary~1.4]{AB} obtained sharper
upper bounds for $0\leq\gamma<3$, and Arguin, Bailey, and
Roberts~\cite[Corollary~3.2]{ABR} extended these bounds to
every fixed $\gamma\geq0$ under RH. In particular, for every fixed
$k>\sqrt{1+\gamma}$ and $A>1$, they obtain
\[
 \frac1{(\log T)^\gamma}
 \int_{|h|\leq(\log T)^\gamma}
 \Big|\zeta\Big(\tfrac12+i(t+h)\Big)\Big|^{2k}\,dh
 \ll
 \frac{(\log T)^{2k\sqrt{1+\gamma}-(1+\gamma)}}
      {(\log\log T)^{k/(2\sqrt{1+\gamma})}},
\]
outside a set of measure $\ll T/A$. At the critical exponent
$k=1$ on intervals of fixed length, Harper~\cite[Corollary~1]{Harper-partition}
proved an upper bound of order $\log T/\sqrt{\log\log T}$
in probability.

We note that for our result, when $k>k_{\gamma}$, the exponent $2kk_\gamma-k_\gamma^2$ is linear in $k$ and is
smaller than $k^2$.  A key ingredient in the proof is our lower bound for the
minimum in Theorem \ref{prop:main}. 
The extra exponential factor in the bound \eqref{supercritical} comes from the error of
size $(\log\frac{1}{\alpha})^{C}$ in the logarithm of this
minimum bound (see \eqref{eq:V-form} in Theorem \ref{prop:main}). In the positive-moment problem, Arguin and
Bailey~\cite[Corollary~1.3]{AB} have a more precise bound for
the maximum, which allows them to obtain the logarithmic
factor in their moment estimate. Our bound therefore has the
analogous leading power, but is less precise in its
lower-order terms.

\subsection{Ideas of proof}
The main analytic issue is to estimate small values close
enough to the critical line that the contribution of zeros begins to
matter. Large deviations give the sharp negative-moment
bounds, while separate estimates for more extreme values
provide the bound needed for the M\"obius application. 

To obtain the large-deviation estimates, we build on the
 method of Arguin, Bailey, and
Roberts~\cite{ABR}. The main additional difficulty is that the
contribution of zeros becomes increasingly important as the
shift approaches $1/\log T$. The explicit inequalities of
Carneiro and Chandee~\cite{CC}, in the form used by
Bui and Florea~\cite{BF}, bound $-\log|\zeta|$ by a sum over
primes and an error depending on the shift. Taking more primes
reduces this error, but makes it harder to estimate the rare
large values of the sum. The proof must balance these two
requirements throughout the range of $\alpha$.

For larger shifts (i.e., larger than $(\log \log T)^{O(1)}/\log T)$, one suitably chosen prime sum suffices.
For smaller shifts (i.e., $1/\log T\ll \alpha \ll (\log \log T)^{O(1)}/\log T$), we divide the primes into successive intervals
and control how their contributions accumulate, following the
recursive approach of~\cite{ABR}. We estimate separately the contribution from points where all
partial sums remain within the prescribed barriers and from
points where one of these bounds first fails. We keep the dependence on the shift precise enough
to recover the full Gaussian upper bound, including the factor
$1/\sqrt{\log\frac1{\alpha}}$. This gives the sharp bound
in Theorem~\ref{mainthm}. In the smaller-shift regime, the same argument gives a weaker
large-deviation bound when $\kappa\geq1/2$, as stated in
Theorem~\ref{bigk}.

We then bound negative moments by integrating the deviation
probabilities against $2ke^{2kV}$. For every fixed $k>0$,
the contribution from the range where the Gaussian bound
applies is $O(\alpha^{-k^2})$. 
This calculation is the same for all $k$, with the 
Gaussian range depending on the shift. The Gaussian integral
is centered at $V=k\log\frac1{\alpha}$. Theorem~\ref{mainthm}
controls a neighbourhood of this level for $0<k<1/2$ down
to shifts of order $1/\log T$, and for $k>1/2$ in the
larger-shift range used in Theorem~\ref{k>1/2}.
In either case, the Gaussian estimate alone does not control
the full moment; we also need separate (non-Gaussian) bounds for larger deviation levels. 

For $0<k<1/2$, we first use Theorem \ref{bigk} beyond the
Gaussian range. For all $V\geq B\log\log T$, the high-moment
argument in Lemma \ref{lem:negative-moment-high-tail} bounds
the deviation probability by $O(e^{-\rho V})$, where
$2k<\rho<1$ and $B$ is sufficiently large.
Integrating this bound against $2ke^{2kV}$ shows that the
entire contribution from these levels is $o(1)$.

For $k>1/2$, the weight $e^{2kV}$ requires stronger bounds
at large deviation levels. We obtain these from high moments
of prime sums, following the method of
Soundararajan~\cite{Sound-moments} and keeping track of the
shift-dependent error in the explicit formula. In the range
of Theorem~\ref{k>1/2}, these probabilities decay fast enough
to make their integrals against $e^{2kV}$ negligible. 

The M\"obius application requires a bound on the moment
$I_{-1/2}(\alpha,T)$. Although the bound for $I_{-1/2}$ in Theorem \ref{negative1/2} is not of the
conjectured order, it is sufficient for the arithmetic
application. Following the contour argument of Bui and
Florea \cite{BF}, we use Perron's formula to express the
average of the M\"obius function in terms of $1/\zeta(s)$.
The improved bound for $I_{-1/2}$ controls the vertical portions
of the contour close to the critical line and gives
Theorem \ref{partM} in the case $k=1$. The same contour argument, using
Theorem~\ref{k>1/2}, yields the estimates for the generalized
M\"obius function in the case $k>1$.

Finally, we turn to the minimum in short intervals, for which
the two directions require different arguments. A lower bound
for the minimum must rule out values that are too small
\emph{anywhere} in the interval. We therefore combine the
explicit inequality for $\log |\zeta|^{-1}$ with moment estimates for the maximum of
a prime sum.

An upper bound for the minimum instead requires the existence
of at least one sufficiently small value. For the prime sums,
we adapt the argument of Arguin, Belius, Bourgade,
Radziwi\l\l{}, and Soundararajan~\cite{ABBRS} to produce a
large negative value. We then show that the large negative value of the prime sum
gives a correspondingly small value of $|\zeta|$ at the same point. For larger shifts $\alpha$, the inequalities of
Carneiro and Chandee~\cite{CC} bound $\log|\zeta|$ above by the
prime sum, with error terms that we can control uniformly over
the interval. For shifts of order $1/\log T$, we first find a small value
of the zeta-function at a larger shift, where the prime-sum argument
applies. We then use the maximum modulus principle to show
that a small value also occurs on the line
$\Re s=1/2+\alpha$ within the original short interval.
The negative-moment bounds provide the control on the
horizontal sides of the rectangle needed for this argument.

Finally, the upper bounds for the negative moments in almost all short intervals in Theorem \ref{momshortint} are an application of Theorem \ref{prop:main}.

\subsection{Organization of the paper}
In Section \ref{preliminary} we first gather a few preliminary results we will use throughout the paper. We prove Theorem \ref{mainthm} when the shift $\alpha$ is large enough in Section \ref{section_large} (contingent on a certain localization lemma proven in Section \ref{section_localization}), and begin the proof of Theorem \ref{mainthm} in the case of small $\alpha$ in Section \ref{section_prelim_small}. In Section \ref{section_localization} we prove some key localization results needed in the proof of Theorems \ref{mainthm} and \ref{bigk}. We finish the proofs of Theorems \ref{mainthm} and \ref{bigk} in Section \ref{section_largedev}. In Section \ref{section_negativemom} we prove the upper bounds in Theorems \ref{negativemom}, \ref{negative1/2} and \ref{k>1/2}, and we prove Theorem \ref{partM}\ in Section \ref{section_mobius}. Theorem \ref{minimum_first_order} is proved in Section \ref{section_min}, and finally, we prove Theorem \ref{momshortint} in Section \ref{section_shortintervals}.

\subsection{Notation}

We write $\P$ and $\E$ for probability and expectation with
respect to $t$ chosen uniformly from $[T,2T]$. 

Let $(X(p))_p$ be independent Steinhaus random variables uniformly
distributed on the unit circle. We extend $X$ completely multiplicatively by setting
\[
 X(1)=1,\qquad X(n)=\prod_{p^r\parallel n}X(p)^r.
\]
We denote probability and expectation for these variables by
$\P_{\rm rand}$ and $\E_{\rm rand}$. For a Dirichlet polynomial
$
 P(t)=\sum_{n\leq N}a_n n^{-it},
$
its random counterpart is
\[
 \widetilde P=\sum_{n\leq N}a_nX(n).
\]
Thus a tilde means that each $p^{-it}$ is replaced by $X(p)$;
in particular, $p^{-2it}$ is replaced by $X(p)^2$.

The notation $f\ll g$ means $|f|\leq Cg$ for some constant
$C>0$, and $f\asymp g$ means that both $f\ll g$ and $g\ll f$
hold for positive $f,g$. Throughout the paper, $\varepsilon$ denotes an arbitrarily small positive number and $C$ denotes a positive constant, both of which may change from line to line.

Throughout, $\sum_{u\in[L,U]}$ denotes a sum over
$u\in\mathbb Z\cap[L,U+1]$.

For ease of notation, we will write $ M_{\alpha}(t)=M_{\alpha,0}(t)$, where
\[
M_{\alpha,\gamma}(t)
:=
\min_{|h|\leq(\log T)^\gamma}
\Big|\zeta\Big(\tfrac12+\alpha+i(t+h)\Big)\Big|.
\]
We also write
$$W_{\alpha}(t) = \log \big| \zeta (  \tfrac{1}{2}+\alpha+it) \big|^{-1}.$$
\newline

\textbf{AI Disclosure:}
The authors used ChatGPT to assist with the proofs of Propositions \ref{comparison21} and \ref{comparison22}. They also used ChatGPT to test the choice of parameters, and with checking and proofreading the
mathematical arguments. Additionally, after finishing a draft of this paper, the authors asked Claude for a report on the paper; the report suggested a possible improvement in the $\log \log x$ term in the bound in Theorem \ref{partM} (going from $\exp\big( C(\log x)^{k/(k+1)}\big )$ to $\exp\big( C(\frac{\log x}{\log\log x})^{k/(k+1)}  \big)$). This suggestion was worked out and incorporated in the current paper. All other ideas are due to the authors,
who take full responsibility for the content of the paper.

\textbf{Acknowledgments:}
The authors would like to thank Louis-Pierre Arguin and Adam Harper for helpful comments. AF was supported by the National Science
Foundation under NSF CAREER grant DMS-2339274.

\section{Preliminary results}
\label{preliminary}
We gather a few results we will need in the proof of the main theorems. First we need the following lemma from \cite{BF}.
\begin{lem}
\label{bf_lemma}
Assume RH. For $\alpha > 0$, we have 
\begin{align*}
    \log  \Big|\zeta(\tfrac12+\alpha+it)\Big|^{-1} &\leq \Re \bigg( \sum_{n \leq x} \frac{ b_{\alpha}(n;x)}{n^{1/2+\alpha+it}} \bigg)+ \frac{ \log \frac{t}{2 \pi}}{\log x} \log \frac{1}{1-x^{-\alpha}} \\
    &\qquad+ O \Big(  \frac{(\log x)^2 \sqrt{x}}{1+(\log x) t}+ \frac{ (\log x) \log (1+ (\log x) \sqrt{t})}{\sqrt{t}}+1\Big),\nonumber 
\end{align*}
where \begin{equation*}b_{\alpha}(n;x)= -\Lambda(n) a_{\alpha}(n;x), \label{balpha}
\end{equation*} and 

$$a_{\alpha}(n;x) = \sum_{j=0}^{\infty} \Big( \frac{ j+1}{\log n+ j \log x} x^{-\alpha j}- \frac{(j+1)n^{2\alpha}}{(j+2) \log x - \log n} x^{-\alpha(j+2)} \Big).$$

\end{lem}

We will also prove the following refinement of the inequality $(2.6)$ in \cite{BF}.
\begin{lem}
\label{asympta}
    For $n \leq e^{2 \pi \Delta}$, we have
    $$ a_{\alpha}(n;e^{2 \pi \Delta}) = 
    \begin{cases}
        \frac{1}{\log n} + O(  \frac{1}{\Delta} ) & \mbox{ if } 2 \pi \Delta \alpha \geq 1,\\
        \frac{1}{\log n} + \frac{2 (2\pi \Delta-\log n)}{(2 \pi \Delta)^2} \log \frac{1}{2 \pi \Delta \alpha} + O (  \frac{1}{\Delta}) &\mbox{ if } 2 \pi \Delta \alpha < 1.
    \end{cases}$$
\end{lem}
\begin{proof}
We first consider the regime $2 \pi \Delta \alpha \geq 1$.  We write
\begin{align*}
 a_{\alpha}(n;e^{2\pi \Delta})& = \frac{1}{\log n} - \frac{n^{2 \alpha}}{4 \pi \Delta - \log n} e^{-4 \pi \Delta \alpha}\\
 &\qquad+ \sum_{j \geq 1}  \Big( \frac{j+1}{\log n+ 2 \pi j \Delta }e^{-2 \pi j \Delta \alpha} - \frac{(j+1) n^{2\alpha}}{2 \pi \Delta (j+2)  - \log n} e^{-2 \pi (j+2) \Delta \alpha} \Big)  .
\end{align*}
Let $E$ denote $a_{\alpha}(n;e^{2 \pi \Delta})-1/\log n$. Note that we have 
$$ \sum_{j \geq 1}  \frac{j+1}{\log n+ 2 \pi j \Delta }e^{-2 \pi j \Delta \alpha}  \leq \frac{1}{\pi \Delta} \sum_{j \geq 1} e^{-2 \pi j \Delta \alpha} \leq \frac{1}{\pi \Delta(e-1)}.$$
We also have that 
$$ \sum_{j \geq 1}   \frac{(j+1) n^{2\alpha}}{2 \pi \Delta (j+2)  - \log n} e^{-2 \pi (j+2) \Delta \alpha} \leq \frac{1}{2 \pi \Delta(e-1)} .$$
It follows that
$$|E| \leq \frac{1}{2 \pi \Delta} \Big(  1+ \frac{3}{ e-1}\Big) .$$

Now suppose that $2 \pi \Delta \alpha < 1.$ We write
$$a_{\alpha}(n;e^{2\pi \Delta}) = \frac{1}{\log n} +  \sum_{j=1}^{\infty} \frac{ j+1}{\log n+  2 \pi \Delta j} e^{-2 \pi \Delta \alpha j}- \sum_{j=0}^{\infty} \frac{(j+1)n^{2\alpha}}{ 2 \pi \Delta (j+2)  - \log n} e^{-2 \pi \Delta \alpha(j+2)} .$$
Note that we can replace $n^{2\alpha}$ by $1$ as follows. Since $\log n\le2\pi\Delta$, we have
$
0\ \le  n^{2\alpha}-1 \le e^{4\pi\Delta\alpha}-1,
$
and, for every $j\ge0$,
\[
\frac{j+1}{2\pi\Delta(j+2)-\log n}\ \le \frac{1}{2\pi\Delta}\,.
\]
Summing the geometric series,
\[
\sum_{j\ge0}\frac{(j+1)\,e^{-2\pi\Delta\alpha(j+2)}}{2\pi\Delta(j+2)-\log n}
\ \le  \frac1{2\pi\Delta}\sum_{j\ge0}e^{-2\pi\Delta\alpha(j+2)}
 = \frac1{2\pi\Delta}\cdot\frac{e^{-4\pi\Delta\alpha}}{1-e^{-2\pi\Delta\alpha}}\,,
\]
and hence, 
\begin{align*}
\bigl(n^{2\alpha}-1\bigr)\sum_{j\ge0}\frac{(j+1)\,e^{-2\pi\Delta\alpha(j+2)}}{2\pi\Delta(j+2)-\log n}
& \le \frac{\bigl(e^{4\pi\Delta\alpha}-1\bigr)e^{-4\pi\Delta\alpha}}{2\pi\Delta\bigl(1-e^{-2\pi\Delta\alpha}\bigr)}
  = \frac{1+e^{-2\pi\Delta\alpha}}{2\pi\Delta} \le \frac{1}{\pi\Delta}\,.
\end{align*}
Then we write
$$a_{\alpha}(n;e^{2\pi\Delta}) = \frac{1}{\log n} + \sum_{j=1}^{\infty} \frac{ j+1}{\log n+  2 \pi \Delta j} e^{-2 \pi \Delta \alpha j}- \sum_{j\ge2}\frac{(j-1)\,e^{-2\pi\Delta\alpha j}}{2\pi\Delta j-\log n}+O \Big(  \frac{1}{\Delta}\Big).$$
 The $j=1$ term of the first sum is $2e^{-2\pi\Delta\alpha}/(2\pi\Delta+\log n)\le1/(\pi\Delta)$, and for $j\ge2$,
\begin{align*}
\frac{j+1}{\log n+2\pi\Delta j}-\frac{j-1}{2\pi\Delta j-\log n}
=\frac{2j\,(2\pi\Delta-\log n)}{(2\pi\Delta j)^2-(\log n)^2} .
\end{align*}
Therefore
\[
a_\alpha\bigl(n;e^{2\pi\Delta}\bigr)
=\frac1{\log n}+2\,(2\pi\Delta-\log n)\sum_{j\ge2}\frac{j\,e^{-2\pi\Delta\alpha j}}{(2\pi\Delta j)^2-(\log n)^2}+O\Bigl(\frac1\Delta\Bigr).
\]
Since $\log n\le2\pi\Delta$ and $j\ge2$,
\[
0\le\frac{j}{(2\pi\Delta j)^2-(\log n)^2}-\frac{1}{(2\pi\Delta)^2\,j}
=\frac{(\log n)^2}{(2\pi\Delta)^2\,j\,\bigl((2\pi\Delta j)^2-(\log n)^2\bigr)}
\ll \frac{1}{\Delta^2 j^3}.
\]
Hence
\[
a_\alpha\bigl(n;e^{2\pi\Delta}\bigr)
=\frac1{\log n}+\frac{2\,(2\pi\Delta-\log n)}{(2\pi\Delta)^2}\sum_{j\ge2}\frac{e^{-2\pi\Delta\alpha j}}{j}+O\Bigl(\frac1\Delta\Bigr).
\]
Finally, 
\[
\sum_{j\ge2}\frac{e^{-2\pi\Delta\alpha j}}{j}=\log\frac{1}{1-e^{-2\pi\Delta\alpha}}-e^{-2\pi\Delta\alpha}
=\log\frac{1}{2\pi\Delta\alpha}+O(1).
\]
This finishes the proof.
\end{proof}

Let 
\begin{equation} 
\label{sx} 
S_x(t) = \Re \bigg(  \sum_{p \leq x} \frac{b_{\alpha}(p;x)}{p^{1/2+\alpha+it}}+ \sum_{p^2 \leq x} \frac{b_{\alpha}(p^2;x)}{p^{1+2\alpha+2it}}\bigg).
\end{equation}
Using Lemma \ref{asympta}, if $\alpha \log x \gg 1$ we have that 
\begin{equation}
\label{bpx}
b_{\alpha}(p;x) = -1+ O \Big( \frac{\log p}{\log x} \Big) \qquad\text{and} \qquad b_{\alpha}(p^2;x) = -\frac{1}{2} + O \Big( \frac{\log p}{\log x} \Big).
\end{equation}

For simplicity of notation, for a sequence $\beta_1 <  \ldots < \beta_K$ of parameters we will define later, we let
$$b_{\alpha}(n;j) = b_{\alpha}(n;T^{\beta_j}).$$
Let
$$S_{\ell}^{(j)}(t) = \Re \bigg( \sum_{p \leq T_{\ell}} \frac{b_{\alpha}(p;j)} {p^{1/2+\alpha+it}}  + \sum_{\substack{p \leq T_{\ell}\\ p^2 \leq T_j}} \frac{ b_{\alpha}(p^2;j)}{p^{1+2\alpha+2it}}\bigg).$$ 
Note that the definition of $S_{\ell}^{(j)}$ is slightly different from the one in \cite{ABR}; unlike in \cite{ABR}, we do not consider the prime squares for which $p^2 \leq T_{\ell}$. 

We have the following lemma which computes the variance of $S_{\ell}^{(j)}$.

\begin{lem}
\label{variances}
Let $j\ge\ell$ and suppose that $2 \leq T_\ell\le T^{1/4-\varepsilon}$ and $\alpha \geq T_j^{-1}$, with $\alpha <1$. If $\alpha\log T_j\ge1$, then we have
\[
\frac1T\int_T^{2T}S_\ell^{(j)}(t)^2\,dt=\frac12\min\Bigl\{\log\frac1\alpha,\ \log\log T_\ell\Bigr\}+O(1).
\]
If $\alpha\log T_j<1$, then we have
$$\frac1T\int_T^{2T}S_\ell^{(j)}(t)^2\,dt
=\frac12\log\log T_\ell
+O\Big(1+\log^2\frac1{\alpha\log T_j}\Big).$$

\end{lem}

\begin{proof}
The Dirichlet polynomial whose real part is $S_\ell^{(j)}$
has length at most $T_\ell^2\leq T^{1/2-\varepsilon}$.
So the mean-value theorem  for Dirichlet polynomials gives
\begin{equation}\label{eq:mvt}
\frac1T\int_T^{2T}S_\ell^{(j)}(t)^2\,dt
=\big(1+O(T^{-\varepsilon})\big)\bigg(\frac12 \sum_{p \leq T_{\ell}} \frac{b_\alpha(p;j)^2}{p^{1+2\alpha}}
+\frac12\sum_{\substack{p\le T_{\ell}\\ p^2 \leq T_j}}\frac{b_\alpha(p^2;j)^2}{p^{2+4\alpha}}\bigg).
\end{equation}

Put
\[
Z=\max\Big\{0,\log\frac1{\alpha\log T_j}\Big\}.
\]
By Lemma~2.2, for $r=1,2$ and $p^r\leq T_j$,
\[
b_\alpha(p^r;j)
=-\frac1r+O\Big(\frac{(1+Z)\log p}{\log T_j}\Big).
\]
Since $\alpha\geq T_j^{-1}$, we have $Z\ll\log T_j$.
Hence the prime-square contribution satisfies
\[
\sum_{\substack{p\leq T_\ell\\p^2\leq T_j}}
\frac{b_\alpha(p^2;j)^2}{p^{2+4\alpha}}
\ll
\sum_p\frac{1+(\log p)^2}{p^2}\ll1.
\]

Using the coefficient estimate above and
\[
\sum_{p\leq T_\ell}\frac{(\log p)^n}{p}
\ll(\log T_\ell)^n,
\]
for $n=1,2$,
we obtain
\begin{align*}
\sum_{p\leq T_\ell}
\frac{|b_\alpha(p;j)^2-1|}{p^{1+2\alpha}}
&\ll
(1+Z)\frac{\log T_\ell}{\log T_j}
+(1+Z)^2\frac{(\log T_\ell)^2}{(\log T_j)^2} \ll 1+Z^2.
\end{align*}
Finally, Mertens' theorem and partial summation give
\[
\sum_{p\leq T_\ell}\frac1{p^{1+2\alpha}}
=
\min\Big\{\log\frac1\alpha,\log\log T_\ell\Big\}
+O(1).
\]
If $\alpha\log T_j\geq1$, then $Z=0$, proving the first
statement. If $\alpha\log T_j<1$, then
$\alpha\log T_\ell<1$, so the minimum is
$\log\log T_\ell$, proving the second statement.
\end{proof}

We recall that $(X(p))_p$ are independent Steinhaus random variables uniformly
distributed on the unit circle.

\begin{lem}\label{lem_moments_random}
For any \(q\in\mathbb{N}\) and any finitely supported complex coefficients
\(\gamma_p,\delta_p\), we have
\[
 \E_{\rm rand}\bigg(\Big|\sum_p\gamma_pX(p)+\sum_p\delta_pX(p)^2\Big|^{2q}\bigg)
 \leq 2^qq!\Big(\sum_p\big(|\gamma_p|^2+|\delta_p|^2\big)\Big)^q.
\]
\end{lem}
\begin{proof}
Let 
$$\widetilde{\xi_{1}}:=\sum_{p}\gamma_pX(p)\qquad\text{and}\qquad \widetilde{\xi_{2}}:=\sum_{p}\delta_pX(p)^2.$$
By Minkowski's inequality we have
\begin{align}\label{expin}
\E_{\rm rand}\big(|\widetilde{\xi_{1}}+\widetilde{\xi_{2}}|^{2q}\big)^{1/(2q)}&\leq \E_{\rm rand}\big(|\widetilde{\xi_{1}}|^{2q}\big)^{1/(2q)}+\E_{\rm rand}\big(|\widetilde{\xi_{2}}|^{2q}\big)^{1/(2q)}\nonumber\\
&\leq \sqrt{2}\sqrt{\E_{\rm rand}\big(|\widetilde{\xi_{1}}|^{2q}\big)^{1/q}+\E_{\rm rand}\big(|\widetilde{\xi_{2}}|^{2q}\big)^{1/q}}.
\end{align}

Now
\begin{align*}
\E_{\rm rand}\big(|\widetilde{\xi_{1}}|^{2q}\big)&=\E_{\rm rand}\Bigg(\bigg|\sum_{\substack{p_1,\ldots,p_r\\ p_i\ne p_j}}\sum_{\substack{\alpha_i\geq 1\\\sum_{i=1}^{r}\alpha_i=q}}\binom{q}{\alpha_1,\ldots,\alpha_r}\prod_{i=1}^{r}\gamma_{p_i}^{\alpha_i}X(p_i)^{\alpha_i}\bigg|^2\Bigg)\\
&=\sum_{\substack{p_1,\ldots,p_r\\ p_i\ne p_j}}\sum_{\substack{\alpha_i\geq 1\\\sum_{i=1}^{r}\alpha_i=q}}\binom{q}{\alpha_1,\ldots,\alpha_r}^2\prod_{i=1}^{r}|\gamma_{p_i}|^{2\alpha_i}\\
&\leq q!\sum_{\substack{p_1,\ldots,p_r\\ p_i\ne p_j}}\sum_{\substack{\alpha_i\geq 1\\\sum_{i=1}^{r}\alpha_i=q}}\binom{q}{\alpha_1,\ldots,\alpha_r}\prod_{i=1}^{r}|\gamma_{p_i}|^{2\alpha_i}=q!\bigg(\sum_p|\gamma_p|^2\bigg)^q.
\end{align*}
A similar expression holds for $\E_{\rm rand}(|\widetilde{\xi_{2}}|^{2q})$, and so we obtain from \eqref{expin} that
\[
\E_{\rm rand}\big(|\widetilde{\xi_{1}}+\widetilde{\xi_{2}}|^{2q}\big)\leq 2^qq!\Big(\sum_p\big(|\gamma_p|^2+|\delta_p|^2\big)\Big)^q.
\]
This proves the lemma.
\end{proof}
We also have the following. (See also Lemma A.1 in \cite{ABR}.)

\begin{lem}
\label{lem_moments}
    For $j \geq 1$, and for any $q\in\mathbb{N}$ such that $2q \leq \log T/\log T_1$, we have 
    $$ \frac{1}{T} \int_T^{2T} \big(S_1^{(j)}(t)\big)^{2q} \, dt \ll  4^qq!\text{Var}(S_1^{(j)})^q.$$
    Moreover, for $1 \leq k < \ell \leq j$, and for any $q\in\mathbb{N}$ such that $2q \leq \log T/\log T_\ell$, we have
    $$ \frac{1}{T} \int_T^{2T}   \big(S_{\ell}^{(j)}(t)-S_k^{(j)}(t)\big)^{2q} \, dt \ll 4^qq! \Big( \emph{Var}(S_{\ell}^{(j)}) - \emph{Var}(S_{k}^{(j)}) \Big)^q.$$
\end{lem}
\begin{proof}

We prove the second statement. The treatment of the first one is similar. 

We denote
$$\xi_{1}(t):=\sum_{p}\gamma_pp^{-it},\qquad \xi_{2}(t):=\sum_{p}\delta_pp^{-2it},$$
and
$$\widetilde{\xi_{1}}:=\sum_{p}\gamma_pX(p),\qquad \widetilde{\xi_{2}}:=\sum_{p}\delta_pX(p)^2,$$
where 
\[
\gamma_p=\frac{b_\alpha(p;j)}{p^{1/2+\alpha}}\mathbf1_{\substack{T_k<p\leq T_\ell}}
\qquad\text{and}\qquad
\delta_p=\frac{b_\alpha(p^2;j)}{p^{1+2\alpha}}\,\mathbf 1_{\substack{T_k<p\leq T_\ell\\p^2\le T_j}}.
\]
Then
\begin{equation}\label{formdiff}
S_{\ell}^{(j)}-S_k^{(j)}=\Re\big(\xi_{1}(t)+\xi_{2}(t)\big),
\end{equation}
and, by \eqref{eq:mvt}, 
\begin{equation}\label{formdiff2}
\text{Var}(S_{\ell}^{(j)}) - \text{Var}(S_{k}^{(j)})\ll\frac12\Big(\sum_p|\gamma_p|^2+\sum_p|\delta_p|^2\Big).
\end{equation}

By expanding through the binomial formula and using the mean-value theorem for Dirichlet polynomials, we have
\begin{equation}\label{eq:reduction}
\frac{1}{T} \int_T^{2T}|\xi_1(t)+\xi_2(t)|^{2q}dt\ll\E_{\rm rand}\big(|\widetilde{\xi_{1}}+\widetilde{\xi_{2}}|^{2q}\big).
\end{equation}
Also, from Lemma \ref{lem_moments_random} we obtain that
\begin{equation}\label{eqafter2.4}
\E_{\rm rand}\big(|\widetilde{\xi_{1}}+\widetilde{\xi_{2}}|^{2q}\big)\leq 2^qq!\Big(\sum_p\big(|\gamma_p|^2+|\delta_p|^2\big)\Big)^q.
\end{equation}
Observe that by \eqref{formdiff},
\[
\frac{1}{T} \int_T^{2T}   (S_{\ell}^{(j)}-S_k^{(j)})^{2q} \, dt=\frac{1}{T} \int_T^{2T}   \Re\big(\xi_{1}(t)+\xi_{2}(t)\big)^{2q} \, dt\leq \frac{1}{T} \int_T^{2T}|\xi_1(t)+\xi_2(t)|^{2q}dt.
\]
The second statement of the lemma now follows by combining with \eqref{eq:reduction}, \eqref{eqafter2.4} and \eqref{formdiff2}.
\end{proof}

\section{Proof of Theorem \ref{mainthm} in the large $\alpha$ regime}
\label{section_large}
We will split the proof of Theorem \ref{mainthm} into two parts, depending on the size of $\alpha$. The smaller $\alpha$ is, the more difficult the problem becomes. In this section, we prove Theorem \ref{mainthm} in the large $\alpha$ regime, relying on a localization result we prove in Section \ref{section_localization}. 

We denote by $H$ the event
\begin{equation}
     H=\Big\{ \log \Big|\zeta(\tfrac12+\alpha+it)\Big|^{-1} >V\Big\}.
    \end{equation}

We first assume that 
\begin{equation} 
\label{first_regime}
\alpha \log T \gg \Big(\log \frac{1}{\alpha} \Big)^{12+\varepsilon}.
\end{equation} 
Let \begin{equation} x = \exp \Big(  \frac{2 \log (\alpha \log T)}{\alpha}\Big), \label{x_choice}
\end{equation}
and let 
\begin{equation}
\label{sigmax}
\sigma_x^2 : = \E (S_x^2) = \frac{1}{2}  \log \frac{1}{\alpha}+ O(1), \end{equation} which follows from Lemma \ref{variances}. From Lemma \ref{bf_lemma}, since the choice of $x$ in \eqref{x_choice} gives 
$$ \frac{\log \frac{t}{2\pi}}{\log x} \log \frac{1}{1-x^{-\alpha}} = o(1), $$ we have that
$$ \log \big|\zeta(\tfrac12+\alpha+it)\big|^{-1}\leq S_x(t) +C,$$ for some absolute constant $C$.
Then for any $A>1$ we have 
\begin{align} 
\label{ph_ub}
    \P (H) \leq \P \big(S_x(t) \geq V-C\big) \leq \P \big(S_x(t) \geq (A-1) \sigma_x^2\big) + \sum_{V-C \leq u \leq A \sigma_x^2} \P \big(S_x(t) \in [u-1,u]\big).
\end{align}

We first evaluate the first term in \eqref{ph_ub}. By Markov's inequality, we have that 
\begin{align*}
    \P \big(S_x(t) \geq (A-1) \sigma_x^2\big) \leq \frac{1}{((A-1)\sigma_x^2)^{2q}} \E \big(S_x(t)^{2q}\big),
\end{align*}
where $q$ is any integer such that $q \leq \frac{\alpha \log T}{4 \log(\alpha \log T)}$. 
Lemma \ref{lem_moments} and Stirling's formula give
\begin{align}
\label{pxt}
    \P \big(S_x(t) \geq (A-1) \sigma_x^2\big) \ll \exp \Big( q \log q + q \log \frac{4}{e(A-1)^2} + \frac{\log q}{2} - 2q \log \sigma_x \Big).
\end{align}
Choosing
$ q= [\sigma_x^2],$ noticing that $q \leq \frac{\alpha \log T}{4 \log (\alpha \log T)}$ (where we used the fact that $\alpha \log T \gg ( \log \frac{1}{\alpha} )^{12+\varepsilon}$), and choosing the constant $A$ large enough with respect to $\kappa$, we have that 
\begin{equation}
\label{first_bound}
    \eqref{pxt} \leq \frac{1}{\sigma_x} \exp \Big(  - \frac{V^2}{2 \sigma_x^2}\Big).
\end{equation}

We are left with bounding the second term in inequality \eqref{ph_ub}. Using Proposition \ref{comparison1}, we have that 
\begin{align*}
  \sum_{V-C \leq u \leq A \sigma_x^2}  \P \big(S_x(t) \in [u-1,u]\big) &\ll    \sum_{V-C \leq u \leq A \sigma_x^2} \frac{1}{\sigma_x} \exp \Big(  - \frac{u^2}{2\sigma_x^2}\Big)\\
  &\ll  \frac{\sigma_x}{V-C} \exp \Big( - \frac{(V-C)^2}{2 \sigma_x^2} \Big)\ll \frac{1}{\sigma_x} \exp \Big(  - \frac{V^2}{2 \sigma_x^2}\Big).
\end{align*}
Combining this and the bound \eqref{first_bound} finishes the proof of Theorem \ref{mainthm} in the regime \eqref{first_regime}.

\section{Preliminary steps --  small $\alpha$ regime}
\label{section_prelim_small}
In this section, we consider the more difficult regime when
\begin{equation} 
\label{small_regime}
1 \ll \alpha\log T \ll  \Big( \log \frac{1}{\alpha}\Big)^{12+\varepsilon}.
\end{equation}
More specifically, assume 
$$ M_0 \leq \alpha\log T \leq M_1\Big ( \log \frac{1}{\alpha}\Big)^{12+\varepsilon},$$ for some $M_0, M_1>0$.
Let 
     $T_j = T^{\beta_j},$
     where 
     \begin{equation} \label{rdef} \beta_1 = \frac{1}{(\log \log T)^{30}}, \qquad \beta_j =r^{j-1} \beta_1, 
     \end{equation} for some $r>1$ ($r$ will be chosen sufficiently close to $1$, see the beginning of Section \ref{section_largedev} for more details).  We stop when $\theta/r<\beta_K\le\theta$, 
     for some small enough $\theta$. Restrictions on the size of $\theta$ arise in the proofs of
Propositions \ref{comparison21} and \ref{comparison22} and Lemmas \ref{lem:apriori}, \ref{upper}, \ref{lower}, \ref{g1c}. We do not record
them all here, but choose $\theta$ sufficiently small to
satisfy them.

The ``good events'' $G_{\ell}$ for $1 \leq \ell \leq K$ are defined such that
\begin{equation*} 
\label{def_goodevent} G_{\ell} = \big\{ S_{\ell}^{(j)} \in [L_{\ell}, U_{\ell}], \forall j \geq \ell \big\} \cap G_{\ell-1},\end{equation*}
where
     \begin{equation} \label{new_corridor} L_{\ell} = \frac{V}{\log \log T} \log \log T_{\ell}-c_{\ell}, \qquad U_{\ell} = \frac{V}{\log \log T} \log \log T_{\ell}+c_{\ell},
     \end{equation}
with
$$c_{\ell} = \Big( \frac{1}{\beta_{\ell}}\Big)^{1/90}.$$
We use the conventions
\[
 S_0^{(j)}=0,\qquad G_0=[T,2T].
\]
For simplicity of notation, let
$$z_{\ell} = \log \frac{1}{\beta_{\ell}}.$$
Let 
$$\sigma_{\ell,j}^2= \text{Var}(S_{\ell}^{(j)}).$$
Using Lemma \ref{variances}, we have that
\begin{equation} \label{new_variance}
    \sigma_{\ell,j}^2 =\frac{t_{\ell}}{2}  + 
\begin{cases}
O(1) & \mbox{ if } \alpha \log T_j \geq 1, \\ O(z_{\ell}^2) & \mbox{ if } \alpha \log T_j <1,
    \end{cases}
 \end{equation} where 
 $$t_{\ell}: = \min \Big\{ \log   \frac{1}{\alpha}, \log \log T_{\ell} \Big\}.$$ 
 
We will frequently use the following corollary to Lemma \ref{asympta}.

\begin{cor}
\label{cor_b}
    We have
    \begin{align*} b_{\alpha}(p;j) & = -1+ O \Big( \frac{z_j \log p}{\log T_j}\Big),  \qquad p \leq T_j , \\
    b_{\alpha}(p^2;j) &= - \frac{1}{2} + O \Big( \frac{z_j \log p}{\log T_j}\Big), \qquad p^2 \leq T_j.
    \end{align*} 
\end{cor}

For $2 \leq \ell \leq K$, let
$$A_{\ell} = \big\{ |S_{\ell}^{(j)}- S_{\ell-1}^{(j)}| \leq 2 c_{\ell},  \, \, \forall \ell \leq j \leq K \big\}.$$
We have that
\begin{align}
    \P(H) & \leq \P (H \cap G_1^c)+ \sum_{\ell=2}^K \Big(\P (G_{\ell-1} \cap A_{\ell}^c)+ \P \big(G_{\ell-1} \cap A_{\ell} \cap \{ \exists j \geq \ell: S_{\ell}^{(j)}>U_{\ell}\}\Big) \nonumber \\
    &\qquad + \P \big(H \cap G_{\ell-1} \cap A_{\ell} \cap \{ \exists j \geq \ell : S_{\ell}^{(j)}<L_{\ell} \}\big) \Big) + \P(H \cap G_K). \label{ph}
\end{align}
We will bound the probabilities above in Section \ref{section_largedev}. Before doing that, in the next section, we prove some key estimates which are used to bound \eqref{ph}.

\section{Some localization results}
\label{section_localization}
In this section, we will prove three key localization bounds, which are analogues of Lemmas B.1 and B.2 in \cite{ABR}. These will be used in the proof of Theorem \ref{mainthm}. 

We first assume that 
$$ 1 \ll \alpha \log T \ll \Big(  \log \frac{1}{\alpha}\Big)^{12+\varepsilon}. $$
We denote by 
\begin{equation*}\label{cmp:events}
 G_\ell^{(j)}(u)=G_{\ell-1}\cap
                 \big\{S_\ell^{(j)}\in[u-1,u]\big\}.
\end{equation*}
For a sufficiently large constant \(C\), put
\begin{equation}\label{eq:delta-loc}
 \delta_{\rm loc}
 =C\bigg(\beta_1(\log\log T)^{11}
       +\sum_{i=2}^{K}\beta_i(1+z_i^2)c_i^{45/2}\bigg).
\end{equation}
Fix $\eta<1$. By choosing $\theta$ sufficiently small, and then taking $T$
sufficiently large, we obtain
\begin{equation*} 
\label{deltaloc}
 \delta_{\rm loc}\leq\frac{1-\eta}{2}.
\end{equation*}

\begin{prop}\label{comparison21}
Let \(1\leq\ell\leq j\leq K\).
Let \(\mathcal Q\) be a Dirichlet polynomial of length
at most \(T^\eta\), supported on integers all of whose prime factors exceed \(T_\ell\),
where \(\eta<1\) is fixed and
\begin{equation}\label{eq:length-condition}
 \eta+\delta_{\rm loc}<1.
\end{equation}
Put
\[
 L_\ell':=L_\ell-3c_\ell
   =\frac{V}{\log\log T}\log\log T_\ell-4c_\ell,
 \qquad
 U_\ell':=U_\ell+3c_\ell
   =\frac{V}{\log\log T}\log\log T_\ell+4c_\ell.
\]
Assume \(|u|\leq A\sigma_{\ell,j}^2\), for some $A >0$, and if
\(\ell\geq2\), assume also that \(u\in[L_\ell',U_\ell']\).
Then we have
\begin{equation*}
 \E\Big(|\mathcal Q|^2\mathbf1 \big(G_\ell^{(j)}(u)\big)\Big)
 \ll
 \E_{\mathrm{rand}}\big(|\widetilde{\mathcal Q}|^2\big) \,
 \frac{e^{-u^2/(2\sigma_{\ell,j}^2)}}
      {\sigma_{\ell,j}}.
\end{equation*}
\end{prop}
\begin{proof}
    Let 
\begin{equation*}
\label{yi}
    Y_i =S_{i}^{(j)}-S_{i-1}^{(j)}, \end{equation*} for $1 \leq i \leq \ell$. Note that we have dropped the superscript $j$ in the definition of $Y_i$, since $j$ is fixed throughout the proof.
Put
\[
 \Delta_1=C_1 \sigma_{1,j},\qquad\Delta_i=c_i\qquad(2\le i\le\ell),
 \]
where $C_1$ is sufficiently large, and, after decreasing
$\theta$, assume that
$$\sum_{i\le K}\Delta_i^{-1}\le\frac12\qquad \text{and}\qquad
2\sum_{i\le K}\Delta_i^{-3}\le\frac12.$$
Fix $\tau\in[0,1)$ and consider the intervals $[u_i,u_i+\Delta_i^{-1})$, where
\[
 u_i:=(n_i+\tau)(\Delta_i^{-1}+2\Delta_i^{-4}),
\]
with $n_i\in\mathbb Z$. For fixed $i$, these intervals are disjoint, and as $n_i$ varies, the gap between the interval with
left endpoint $u_i$ and the next one has length $2\Delta_i^{-4}$.

 Suppose that each $Y_i$, for $1\leq i\leq\ell$,
belongs to one of the intervals defined above, and let $u_i$ be
the left endpoint of that interval. Then
\[
0\leq Y_i-u_i<\Delta_i^{-1}.
\]
Summing these inequalities, we obtain that 
\[
0\leq S_q^{(j)}-\sum_{i=1}^{q}u_i
=\sum_{i=1}^{q}(Y_i-u_i)
<\sum_{i=1}^{q}\Delta_i^{-1}
\leq\frac12,
\]
for every $1\leq q\leq\ell$. Consequently, $S_q^{(j)}\in[L_q,U_q]$ implies
$\sum_{i\le q}u_i\in[L_q-1,U_q]$, while
$S_\ell^{(j)}\in[u-1,u]$ implies
$\sum_{i\le\ell}u_i\in[u-3/2,u]$.

Now let $\mathcal I(u)$ denote the set of tuples $\mathbf u=(u_1,\ldots,u_\ell)$ with the following properties:
\begin{align*} 
& u_i=(n_i+\tau)(\Delta_i^{-1}+2\Delta_i^{-4}),\qquad
 n_i\in\mathbb Z\qquad(1\le i\le\ell),\\
&\sum_{i\le q}u_i\in[L_q-1,U_q]\qquad(q<\ell),\\
& \sum_{i\le\ell}u_i\in\Big[u-\frac32,u\Big].
 \end{align*}
If $t\in G_{\ell-1}$, $S_\ell^{(j)}(t)\in[u-1,u]$,
and each $Y_i(t)$ belongs to one of the intervals defined above,
then the corresponding left endpoints satisfy
$
(u_1,\ldots,u_\ell)\in\mathcal I(u).
$

For each $i$, the measure of $\tau$ such that $Y_i$ lies in a gap of length $2 \Delta_i^{-4}$ is  
\[
 \frac{2\Delta_i^{-4}}
      {\Delta_i^{-1}+2\Delta_i^{-4}}
 \leq 2\Delta_i^{-3}.
\]
So we have
$$
 \operatorname{meas}\big\{
 \tau\in[0,1): Y_i \text{ lies in a gap for some }i\leq\ell
 \big\}
 \leq 2\sum_{i\leq\ell}\Delta_i^{-3}
 \leq\frac12,
$$
and, hence, for a set of phases of measure at least $1/2$,
every $Y_i$ belongs to one of the intervals in the grid.

 Multiplying by
$|\mathcal Q|^2\mathbf1(G_\ell^{(j)}(u))$, taking expectation, and then
integrating in $\tau$ gives
\[
 \int_0^1\E\!\Big(|\mathcal Q|^2\mathbf1(G_\ell^{(j)}(u))
   \mathbf1_{\{\text{every }Y_i\text{ lies in one of the displayed
                    intervals at phase }\tau\}}\Big)
 \,d\tau
 \geq\frac12\E\!\Big(|\mathcal Q|^2\mathbf1\big(G_\ell^{(j)}(u)\big)\Big).
\]
Hence, for at least one phase $\tau$, the integrand is at
least \(\frac12\E\big(|\mathcal Q|^2\mathbf1(G_\ell^{(j)}(u))\big)\), and it is, therefore, suffices to bound 
\begin{equation*}
 \sum_{\mathbf u\in\mathcal I(u)}
 \E\!\Big(|\mathcal Q|^2
   \prod_{i\le\ell}
   \mathbf 1\big(Y_i\in[u_i,u_i+\Delta_i^{-1})\big)\Big).
   \label{tobound}
\end{equation*}

Note that we must have
\begin{equation}\label{eq:centers}
 |u_1|\ll V
 \qquad \text{and}\qquad |u_i|\ll  c_i\qquad(i\ge2).
\end{equation} 

We use the polynomials in equation (B.7) of \cite{ABR} with the parameter $a=4$: for $\Delta$ large and $X\ge1$, there exists a polynomial $\mathcal{D}_{\Delta,X}$ of degree at most $100X\Delta^{12}$ with $n$-th coefficient at most $\frac{(2\pi)^n}{n!}\,2\Delta^{8(n+2)}$ such that for $|x|\le X$, we have
\begin{equation}\label{eq:sandwich}
\mathbf 1\bigl(x\in[0,\Delta^{-1}]\bigr)\bigl(1-e^{-\Delta^{2}}\bigr) \le |\mathcal{D}_{\Delta,X}(x)|^2 \le \mathbf 1\bigl(x\in[-\Delta^{-4},\Delta^{-1}+\Delta^{-4}]\bigr)+e^{-\Delta^{2}} .
\end{equation}
Take
\[
 X_1=C_2\Delta_1^{10},\qquad X_i=C_2(1+z_i^2)c_i^8\qquad(i\geq2)
\]
with \(C_2\) a large enough constant, and write
\(\mathcal D_i=\mathcal D_{\Delta_i,X_i}\). The degrees satisfy $$
 d_1\leq100C_2\Delta_1^{22}\ll (\log\log T)^{11}
\qquad \text{and}\qquad
 d_i\leq100C_2(1+z_i^2)c_i^{20}\qquad (i \geq 2).$$ 
Expanding 
the product $|\mathcal Q|^2\prod_i|\mathcal D_i(Y_i-u_i)|^2$, we get terms of the form
\[
 \Big(\mathcal Q\prod_{i\leq\ell}
       (\mathcal S_i^{(j)}-\mathcal S_{i-1}^{(j)})^{r_i}\Big)
 \overline{\Big(\mathcal Q\prod_{i\leq\ell}
       (\mathcal S_i^{(j)}-\mathcal S_{i-1}^{(j)})^{s_i}\Big)},
 \qquad r_i+s_i\leq2d_i.
\]
Each of the two Dirichlet polynomials in such a term has length at
most $T^\eta\prod_{i\le\ell}T_i^{4d_i}$, and we have
\begin{equation*}\label{eq:budget}
 T^\eta\prod_{i\leq\ell}T_i^{4d_i}
 \leq T^{\eta+C \beta_1(\log\log T)^{11}
              +C\sum_{i=2}^{\ell}\beta_i(1+z_i^2)c_i^{20}}
 \leq T^{\eta+\delta_{\rm loc}}<T,
\end{equation*}
by \eqref{eq:length-condition}. 
Applying Montgomery--Vaughan \cite[Corollary 2, (1.8)]{MV}  gives an error of size $O(T^{\eta+\delta_{\rm loc}-1})$, and
combining this with \eqref{eq:sandwich} 
gives
\begin{equation}\label{eq:B9}
\E\Bigl(|\mathcal{Q}|^2\,\mathbf 1\bigl(G_\ell^{(j)}(u)\bigr)\Bigr)
 \ll  \E_{\rm rand}\bigl(|\widetilde{\mathcal Q}|^2\bigr) \sum_{\mathbf u\in\mathcal{I}(u)}\ \prod_{i\le\ell}\E_{\rm rand}\Big(\big|\mathcal{D}_i(\widetilde Y_i-u_i)\big|^2\Big).
\end{equation}

Note that
\[
\text{Var}(\widetilde Y_1)=\sigma_{1,j}^2+o(1)\asymp\log\log T,
\qquad
 \text{Var}(\widetilde Y_i)\ll z_i^2\qquad(i\geq2).
\]
Let
\begin{equation}
\label{random_var}
\widetilde{\sigma}^2
=\text{Var}\bigl(\widetilde S_\ell^{(j)}\bigr)
=\sigma_{\ell,j}^2+o(1).
\end{equation}
We write
\[
\widetilde Y_i=\sum_{T_{i-1}<p\leq T_i}C_j(p),
\]
where $T_0=1$ and
\begin{equation}\label{cjp}
C_j(p)=\Re\Big(
 \frac{b_\alpha(p;j)X(p)}{p^{1/2+\alpha}}
 +\frac{b_\alpha(p^2;j)X(p)^2}{p^{1+2\alpha}}
       \mathbf1_{\{p^2\leq T_j\}}
\Big).
\end{equation}
Using Corollary \ref{cor_b}, we have
\begin{equation}\label{defmpvp}
m_p:=\|C_j(p)\|_\infty
\leq \frac{C}{p^{1/2+\alpha}}
\Big(1+\frac{z_j \log p}{\log T_j}\Big).
\end{equation}
Moreover,
\[
\sum_{T_{i-1}<p\leq T_i}
\E_{\rm rand}\big(C_j(p)^2\big)
=\text{Var}(\widetilde Y_i).
\]

Now Hoeffding's lemma gives
\begin{equation}\label{eq:hoeff}
\E_{\rm rand}(e^{\lambda\widetilde Y_i}) \le  \exp\Bigl(\frac{\lambda^2w_i}{2}\Bigr)
\end{equation}
for every $\lambda\in\mathbb R$, with
\[
w_i:=\sum_{T_{i-1}<p\le T_i}m_p^2.
\]
We have
\[
w_1\ll \log\log T \qquad\text{and}\qquad w_i\ll z_i^2\qquad (i\ge2).
\]
For $|\lambda|\leq3A$, Taylor's formula and the fact that
$\E_{\rm rand}(C_j(p))=0$ give
\[
\log\big(\E_{\rm rand}(e^{\lambda C_j(p)})\big)
=\frac{\lambda^2}{2}\E_{\rm rand}(C_j(p)^2)+O(m_p^3).
\]
By independence,
\begin{equation}\label{eq:mgf}
\E_{\rm rand}(e^{\lambda\widetilde Y_i})
=\exp\bigg(
 \frac{\lambda^2}{2}\text{Var}(\widetilde Y_i)
 +O \Big(\sum_{T_{i-1}<p\leq T_i}m_p^3\Big)
\bigg).
\end{equation}
From the coefficient bound \eqref{defmpvp}, we get that
\[
\sum_{T_{i-1}<p\leq T_i}m_p^3
\ll
\begin{cases}
1 & i=1,\\
z_i^3T_{i-1}^{-1/2} & i\geq2.
\end{cases}
\]

Now note that we have 
$|\mathcal{D}_i(x)|\le2\Delta_i^{16}e^{2\pi\Delta_i^{8}|x|}$. We also have that \begin{equation}
\mathbf 1\bigl(|x|>X_i\bigr)\ \le\ e^{4\pi\Delta_i^{8}(|x|-X_i)}.
\end{equation}
Hence 
\[
|\mathcal{D}_i(\widetilde{Y}_i-u_i)|^2\,\mathbf 1\bigl(|\widetilde{Y}_i-u_i|>X_i\bigr)
 \le 4\Delta_i^{32}\,e^{-4\pi\Delta_i^{8}X_i}\,e^{8\pi\Delta_i^{8}|\widetilde{Y}_i-u_i|}.
\]
Using the triangle inequality,
\begin{equation}
\E_{\rm rand}\Bigl(|\mathcal{D}_i(\widetilde{Y}_i-u_i)|^2 \mathbf 1\bigl( |\widetilde{Y}_i-u_i|>X_i \bigr)\Bigr)
 \le 4\Delta_i^{32}\,e^{-4\pi\Delta_i^{8}X_i}\,e^{8\pi\Delta_i^{8}|u_i|}\ \E_{\rm rand}(e^{8\pi\Delta_i^{8}|\widetilde{Y}_i|}).
\end{equation}
Since
\[
\E_{\rm rand}(e^{8\pi\Delta_i^{8}|\widetilde{Y}_i|}) \le \E_{\rm rand}(e^{8\pi\Delta_i^{8}\widetilde{Y}_i})+ \E_{\rm rand}(e^{-8\pi\Delta_i^{8}\widetilde{Y}_i}),
\]
and, by Hoeffding's bound \eqref{eq:hoeff}, 
\[
\E_{\rm rand}(e^{\pm8\pi\Delta_i^{8}\widetilde{Y}_i}) \le e^{32\pi^2\Delta_i^{16}w_i},
\]
it follows that
\begin{equation} \label{bigyi}
\E_{\rm rand}\Bigl(|\mathcal{D}_i(\widetilde{Y}_i-u_i)|^2 \mathbf 1\bigl( |\widetilde{Y}_i-u_i|>X_i \bigr)\Bigr)
 \le  8 \Delta_i^{32} \exp\Bigl(-4\pi\Delta_i^{8}X_i+8\pi\Delta_i^{8}|u_i|+32\pi^2\Delta_i^{16}w_i\Bigr) \le\ e^{-\Delta_i^{2}}.
\end{equation}
Combining \eqref{bigyi} with \eqref{eq:sandwich} on the range $|\widetilde Y_i-u_i|\le X_i$, we get that
\begin{equation}\label{eq:blockbd}
\E_{\rm rand}\Bigl(\bigl|\mathcal{D}_i(\widetilde Y_i-u_i)\bigr|^2\Bigr) \le  \P_{\rm rand}\Bigl(\widetilde Y_i\in[u_i-\Delta_i^{-4},\,u_i+\Delta_i^{-1}+\Delta_i^{-4}]\Bigr) + 2e^{-\Delta_i^{2}}.
\end{equation}

For simplicity of notation, let 
$$P_i(u_i) = \P_{\rm rand} \Big(  \widetilde{Y}_i \in [u_i - \Delta_i^{-4}, u_i+\Delta_i^{-1}+\Delta_i^{-4}]\Big).$$
Take
\begin{equation}
\label{lambda}
\lambda=\frac{u}{\widetilde\sigma^2}.
\end{equation}
Since $|u|\leq A\sigma_{\ell,j}^2$ and
$\widetilde\sigma^2=\sigma_{\ell,j}^2+o(1)$ by
\eqref{random_var}, we have $|\lambda|\leq3A$ for sufficiently
large $T$. Thus we may apply \eqref{eq:mgf} with this choice.

For a tuple $\mathbf u\in\mathcal{I}(u)$ we have $\sum_iu_i \in [u-2,u]$, so $1\le e^{\lambda(\sum_iu_i-u)+2|\lambda|}$, and \eqref{eq:B9}--\eqref{eq:blockbd} give
\begin{align*}
\sum_{\mathbf u\in\mathcal{I}(u)} &\prod_{i\le\ell}\Big(P_i(u_i) + 2e^{-\Delta_i^{2}} \Big)
 \le  e^{2|\lambda|-\lambda u} \sum_{\mathbf u\in\mathcal{I}(u)}\ \prod_{i\le\ell}e^{\lambda u_i}\Big( P_i(u_i) + 2e^{-\Delta_i^{2}} \Big).
\end{align*}

Let $B_i(u_i)$ denote the $i^{\text{th}}$ factor above.  Then
\begin{align*}
 \sum_{\mathbf u\in\mathcal I(u)}\prod_{i\leq\ell}B_i(u_i)
 &=
 \sum_{\substack{|u_i|\leq C c_i\\2\leq i\leq\ell}}
 \bigg(\prod_{i=2}^{\ell}B_i(u_i)\bigg)
 \sum_{\substack{u_1:\,
       (u_1,\ldots,u_\ell)\in\mathcal I(u)}}B_1(u_1).
\end{align*}
Fixing $(u_i)_{i\ge2}$, since $(u_1,\ldots, u_{\ell}) \in \mathcal{I}(u)$, we must have that $u_1$ is confined to an interval $I$ of length $O(1)$. Let $\widehat I$
be obtained by extending $I$ by $2$ at each end, so that
$|\widehat I|=|I|+4=O(1)$.

If $u_1\in I$ and
$
\widetilde Y_1\in
[u_1-\Delta_1^{-4},\,u_1+\Delta_1^{-1}+\Delta_1^{-4}],
$
then $|\widetilde Y_1-u_1|\leq2$. Thus
$\widetilde Y_1\in\widehat I$ and, since $|\lambda|\leq3A$,
$
e^{\lambda u_1}\ll e^{\lambda\widetilde Y_1}.
$
The enlarged intervals meet only at their endpoints.
Consequently, for any fixed value of $\widetilde Y_1$,
at most one of these intervals contains it. Hence
\begin{align}\label{eq:first-probability}
\sum_{u_1\in I}e^{\lambda u_1} 
\P_{\rm rand}\Bigl(\widetilde Y_1\in
[u_1-\Delta_1^{-4},\,u_1+\Delta_1^{-1}+\Delta_1^{-4}]\Bigr) &\ll  \E_{\rm rand}\Bigl(
e^{\lambda\widetilde Y_1}
\mathbf1(\widetilde Y_1\in\widehat I)\Bigr)\nonumber\\
&= \E_{\rm rand}\bigl(e^{\lambda\widetilde Y_1}\bigr)
\P_\lambda(\widetilde Y_1\in\widehat I),
\end{align}
where $\P_\lambda$ is the normalized tilted measure $$\P_\lambda(E):=\frac{\E_{\rm rand}[e^{\lambda\widetilde Y_1}\mathbf 1_E]}{\E_{\rm rand}( e^{\lambda\widetilde Y_1})}.$$ 
We denote $\E_{\P_\lambda}$ and $\operatorname{Var}_{\P_\lambda}$ the expectation and variance with respect to $\P_\lambda$.  The Berry--Esseen theorem then gives
\begin{equation}\label{eq:BE}
\sup_{x} \P_\lambda\Bigl(\widetilde Y_1\in[x,\,x+O(1)]\Bigr) \ll \frac1{\sigma_{\ell,j}} ,
\end{equation}
where we used the fact that $\operatorname{Var}_{\P_\lambda}(\widetilde Y_1)
 =\sigma_{1,j}^2+O(1)
 \asymp\sigma_{\ell,j}^2$.
For the error term at the initial block we sum  over the
whole admissible range \eqref{eq:centers}, and we have
\[
 \sum_{|u_1|\ll \log\log T}e^{\lambda u_1}2e^{-\Delta_1^2}
 \ll \Delta_1\log\log T\,
        e^{C \log\log T-\Delta_1^2}
 \leq e^{-\Delta_1^2/2}\leq\frac1{\sigma_{\ell,j}}.
\]
Since $\E_{\rm rand} (e^{\lambda \widetilde{Y}_1}) \geq 1$, using Jensen's inequality we obtain that 
\begin{equation} 
\label{firstblock} \sum_{u_1 : (u_1,\ldots,u_{\ell}) \in \mathcal{I}(u)} B_1(u_1) \ll \frac{ \E_{\rm rand} (e^{\lambda \widetilde{Y}_1})}{\sigma_{\ell,j}}.
\end{equation}

Now for $i\ge2$, recall that $|u_i|\ll  c_i$.
We have that 
\[
 \sum_{|u_i| \ll c_i}e^{\lambda u_i}2e^{-\Delta_i^2}
 \ll c_i^2e^{-c_i^2+O( c_i)}.
\]
For the term involving $P_i(u_i)$, a similar argument as in the case of $i=1$ gives 
\[
\sum_{|u_i|\ll c_i}e^{\lambda u_i}P_i(u_i)
\leq e^{O(\Delta_i^{-1}+\Delta_i^{-4})}
     \E_{\rm rand}(e^{\lambda\widetilde Y_i}).
\]
Since $\E_{\rm rand}(e^{\lambda\widetilde Y_i})\ge1$ by Jensen, combining the two bounds above leads to
\begin{align*}
&\sum_{u_i}B_i(u_i)\le e^{O(\Delta_i^{-1}+\Delta_i^{-4})}
 \E_{\rm rand}\big(e^{\lambda\widetilde Y_i}\big)
 \Big(1+O\big( c_i^2e^{-c_i^2+O(c_i)}\big)\Big).
\label{eq:Xi}
\end{align*}
For $\theta$ small enough, we have
\begin{equation*}\label{eq:errprod}
 \exp \bigg(O\Big(\sum_{i\ge2}
       (\Delta_i^{-1}+\Delta_i^{-4}) \Big)\bigg)
 \prod_{i\ge2}\Bigl(1+O\big(c_i^2e^{-c_i^2+O(c_i)}\big)\Bigr)\ll 1.
\end{equation*}
Hence, combining the above with \eqref{firstblock} and \eqref{eq:mgf}, we get that 
\begin{align*}
\sum_{\mathbf u\in\mathcal{I}(u)} \prod_{i\le\ell}  \bigg(\P_{\rm rand}\Bigl(\widetilde Y_i\in[u_i-\Delta_i^{-4},\,u_i+\Delta_i^{-1}+\Delta_i^{-4}]\Bigr)+2 e^{-\Delta_i^2} \bigg)
 &\ll \frac{e^{-\lambda u}}{\sigma_{\ell,j}}\ \prod_{i\le\ell}\E_{\rm rand}\big(e^{\lambda\widetilde Y_i}\big) e^{O(1)}\\
& \ll \frac1{\sigma_{\ell,j}}\exp\Bigl(-\lambda u+\frac{\lambda^2\widetilde\sigma^2}{2}\Bigr)
  ,
\end{align*}
using \eqref{random_var}. 
Since
\begin{equation*}
-\lambda u+\frac{\lambda^2\widetilde\sigma^2}{2}
=-\frac{u^2}{2\widetilde\sigma^2},
\end{equation*} by the choice for $\lambda$ in \eqref{lambda}, we obtain the proposition.
\end{proof}

We also need the following proposition.

\begin{prop}\label{comparison22}
Let \(2\leq\ell\leq K\) and \(\ell\leq j,m\leq K\).
Put
\[
 L_\ell':=L_\ell-3c_\ell
   =\frac{V}{\log\log T}\log\log T_\ell-4c_\ell,
 \qquad
 U_\ell':=U_\ell+3c_\ell
   =\frac{V}{\log\log T}\log\log T_\ell+4c_\ell.
\]
Assume
\[
 u,w\in[L_\ell',U_\ell'],\qquad
 |u|\leq A\sigma_{\ell,j}^2\qquad\text{and}\qquad
 |w|\leq A\sigma_{\ell,m}^2.
\]
Then there exists $C>0$ such that
\begin{align}
 &\P\Big(G_{\ell-1}\cap\{S_\ell^{(j)}\in[u-1,u]\}
                       \cap\{S_\ell^{(m)}\in[w-1,w]\}\Big)\notag\\
 &\qquad\ll \frac1{\sigma_{\ell,j}}
 \exp\Big(-\frac{(u+w)^2}{8\sigma_{\ell,j}^2}
       -C\frac{(w-u)^2}{z_\ell^2}+O(z_\ell^2)\Big).
\label{eq:two-point}
\end{align}
Also, for \(\ell=1\), \(1\leq j,m\leq K\) and \(H>0\) fixed, uniformly for 
\(|u|\leq A\sigma_{1,j}^2\),
\(|w|\leq A\sigma_{1,m}^2\), we have
\begin{align}
 &\P\Big(\{S_1^{(j)}\in[u-1,u]\}
                       \cap\{S_1^{(m)}\in[w-1,w]\}\Big)\ll \frac1{\sigma_{1,j}}
 \exp\Big(-\frac{(u+w)^2}{8\sigma_{1,j}^2}
       -H|w-u|+O(z_1^2)\Big). \label{eq:initial-two-point}
\end{align}
\end{prop}

\begin{proof}
The proof is similar to the proof of Lemma B.2 in \cite{ABR} and follows along similar lines as the proof of Proposition \ref{comparison21}.

Put
\[
 Y_i=S_i^{(j)}-S_{i-1}^{(j)}\qquad\text{and}\qquad
 Y_i'=S_i^{(m)}-S_{i-1}^{(m)}.
\]
Take
\begin{equation}\label{delta1}
 \Delta_1=C_3 \sigma_{1,j},\qquad a_1=4,
 \qquad X_1=C_3\Delta_1^{10},
\end{equation}
and, for $i\geq2$,
\begin{equation}\label{deltai}
 \Delta_i=c_i,\qquad a_i=\frac92,
 \qquad X_i={C_2z_i^2c_i^9},
\end{equation}
where $C_3$ is sufficiently large.
After decreasing $\theta$, we also assume that
\[
 \sum_{i\le K}\Delta_i^{-1}\le\frac12,\qquad
 4\sum_{i\le K}\Delta_i^{1-a_i}\le\frac12.
\]
For $\tau\in[0,1)$, we consider the intervals $[u_i,u_i+\Delta_i^{-1})$ and
$[w_i,w_i+\Delta_i^{-1})$, where
\[
 u_i=(n_i+\tau)(\Delta_i^{-1}+2\Delta_i^{-a_i}),\qquad
 w_i=(m_i+\tau)(\Delta_i^{-1}+2\Delta_i^{-a_i})
\]
with $n_i,m_i\in\mathbb Z$.    We define $\mathcal I(u)$ and $\mathcal I(w)$
as in
Proposition~\ref{comparison21}, using the grids above.
For $\mathbf u\in\mathcal I(u)$ and
$\mathbf w\in\mathcal I(w)$, when $\ell\geq2$ we have
\begin{equation*}\label{eq:b2-center-size}
|u_i|+|w_i|\ll  c_i\qquad(2\leq i\leq\ell),
\end{equation*}
and
\begin{equation}
\label{u1w1}
|u_1|+|w_1|\ll  V \ll \log\log T,
\qquad
|u_1-w_1|\ll c_1.
\end{equation}

By the averaging argument of
Proposition~\ref{comparison21},  
we may fix a phase $\tau$ for which
\begin{align*}
&\P\Bigl(
G_{\ell-1}
\cap\big\{S_\ell^{(j)}\in[u-1,u]\big\}
\cap\big\{S_\ell^{(m)}\in[w-1,w]\big\}
\Bigr)\notag\\
&\qquad\ll
\sum_{\substack{\mathbf u\in\mathcal I(u)\\
                 \mathbf w\in\mathcal I(w)}}
\P\Bigl(
Y_i\in[u_i,u_i+\Delta_i^{-1}),\
Y_i'\in[w_i,w_i+\Delta_i^{-1})
\text{ for all }i\leq\ell
\Bigr).
\label{tbc}
\end{align*}
We choose the polynomial $\mathcal{D}_i$ as in \eqref{eq:sandwich}
with the choices of parameters as in \eqref{delta1} and \eqref{deltai}.
Recall the definition \eqref{cjp} for $C_j(p)$
and define
\[
 \widetilde Y_i=\sum_{T_{i-1}<p\le T_i}C_j(p)\qquad\text{and}\qquad
 \widetilde Y_i'=\sum_{T_{i-1}<p\le T_i}C_m(p).
\]

Before passing to the random model, we check the lengths
required by the mean-value theorem. As in
Proposition~\ref{comparison21}, expanding
$
\prod_{i\leq\ell}
|\mathcal D_i(Y_i-u_i)|^2
|\mathcal D_i(Y_i'-w_i)|^2
$
gives products of a Dirichlet polynomial and the conjugate
of another. Each has length at most
\begin{equation*}\label{eq:b2-length}
 \prod_{i\leq\ell}T_i^{8d_i}
 \leq T^{C \beta_1(\log\log T)^{11}
    +C \sum_{i=2}^{\ell}\beta_i(1+z_i^2)c_i^{45/2}}
 \leq T^{\delta_{\rm loc}}<T.
\end{equation*}
The same mean-value argument as in the proof of Proposition \ref{comparison21} shows that 
\begin{align*}
 \P\Big(&G_{\ell-1}
 \cap\big\{S_\ell^{(j)}\in[u-1,u]\big\}
 \cap\big\{S_\ell^{(m)}\in[w-1,w]\big\}\Big)\notag\\
 &\quad\ll\sum_{\mathbf u\in\mathcal I(u),\,\mathbf w\in\mathcal I(w)}
 \E_{\rm rand}\bigg(\prod_{i=1}^{\ell}
 |\mathcal D_i(\widetilde Y_i-u_i)|^2
 |\mathcal D_i(\widetilde Y_i'-w_i)|^2\bigg).
 \label{eq:b2-transfer}
\end{align*}

We first remove the
ranges on which either $|\widetilde Y_i-u_i|>X_i$ or
$|\widetilde Y_i'-w_i|>X_i$.  The same exponential-moment calculation
as in \eqref{bigyi}, applied to the two variables, shows that their
total contribution is
$
 \ll e^{-\Delta_i^{a_i-2}}.
$
 Now using \eqref{eq:sandwich}, we get that
the probability on the left hand side of \eqref{eq:two-point} is
$$ \ll\sum_{\textbf{u} \in \mathcal{I}(u), \textbf{w} \in \mathcal{I}(w)} \prod_{i=1}^{\ell} \big( P_i(u_i,w_i)  + \varepsilon_i \big), $$
where, for simplicity of notation,
$$ P_i(u_i,w_i) = \P_{\rm rand} \Big( \widetilde Y_i \in [u_i-\Delta_i^{-a_i}, u_i + \Delta_i^{-1}+\Delta_i^{-a_i}], \widetilde Y_i' \in [w_i-\Delta_i^{-a_i}, w_i + \Delta_i^{-1}+\Delta_i^{-a_i}]\Big) $$
and 
$$\varepsilon_i = C e^{-\Delta_i^{a_i-2}}.$$

Define
\begin{equation*}
\label{zp}
 Z_p=\frac{C_j(p)+C_m(p)}2,\qquad F_p=C_m(p)-C_j(p),
\end{equation*}
where $C_j(p)$ is defined in equation \eqref{cjp}. Let $\mathcal Z_i,\mathcal F_i$ be their respective sums over
$T_{i-1}<p\le T_i$ and let
$$ \mathcal{Z} := \sum_{i \leq \ell}  \mathcal{Z}_i= \frac{1}{2} \big(  \widetilde{S_{\ell}}^{(j)}+\widetilde{S_{\ell}}^{(m)} \big),$$
and $$\mathcal{F} :=  \sum_{i \leq \ell} \mathcal{F}_i = \widetilde S_{\ell}^{(m)}-\widetilde S_{\ell}^{(j)}.$$
Corollary \ref{cor_b} gives
\[
 \|F_p\|_\infty\ll
 z_\ell\frac{\log p}{\log T_\ell}p^{-1/2-\alpha},
\]
and using \eqref{defmpvp} as well gives 
\begin{equation}\label{bdZF}
\sum_p\|F_p\|_\infty^2\ll z_\ell^2,\qquad
 \sum_p\|Z_p\|_\infty\|F_p\|_\infty \ll z_\ell^2,\qquad \sum_p\|Z_p\|_\infty^3 \ll1
 \end{equation}
 and
 \begin{align}
\sum_p\big(\|Z_p\|_\infty^2+\|C_j(p)\|_\infty^2\big)\|F_p\|_\infty
 +\max_p\|F_p\|_\infty
 &\ll\frac{z_{\ell}}{\log T_{\ell}}.
 \label{eq:b2-coefficient-sums-short}
\end{align}   
Hence
$$ \text{Var}(\mathcal F) \ll z_\ell^2\qquad\text{and}\qquad \sigma_{+}^{2}:=\text{Var}  (\mathcal{Z})= \sigma_{\ell,j}^2+O(z_\ell^2).$$

Let
\begin{equation*}
    \label{k}
 K(\lambda,\mu)=\log \E_{\rm rand} (e^{\lambda\mathcal Z+\mu\mathcal F}) .
\end{equation*}
For bounded $|\lambda|$, by Taylor expansion we have 
\[
\begin{aligned}
K(\lambda,0)
&=\sum_p\log \E_{\rm rand} (e^{\lambda Z_p})=\frac{\lambda^2}{2}\sum_p\E_{\rm rand} (Z_p^2)
  +O \Big(\sum_p\|Z_p\|_\infty^3\Big)=\frac{\lambda^2\sigma_+^2}{2}+O(1).
\end{aligned}
\]
Moreover, we have
\[
K(\lambda,\mu)-K(\lambda,0)
=
\sum_p
\log\frac{\E_{\rm rand} (e^{\lambda Z_p+\mu F_p})}
          {\E_{\rm rand} (e^{\lambda Z_p})}.
\]

For each $p$, Hoeffding's lemma gives
\[
\log\frac{\E_{\rm rand} (e^{\lambda Z_p+\mu F_p})}
          {\E_{\rm rand} ( e^{\lambda Z_p})}
\le
\mu\frac{\E_{\rm rand}(F_p e^{\lambda Z_p})}
          {\E_{\rm rand} (e^{\lambda Z_p})}
+\frac{\mu^2}{2}\|F_p\|_\infty^2.
\]
Since $\E_{\rm rand} ( F_p)=0$, $\E_{\rm rand}( e^{\lambda Z_p})\ge1$ and
$|e^{\lambda Z_p}-1|\ll |Z_p|$, it follows that
\[
\Big|
\frac{\E_{\rm rand}(F_p e^{\lambda Z_p})}
     {\E_{\rm rand} (e^{\lambda Z_p})}
\Big|
=
\Big|
\frac{\E_{\rm rand}(F_p(e^{\lambda Z_p}-1))}
     {\E_{\rm rand} (e^{\lambda Z_p})}
\Big|
\ll
\|Z_p\|_\infty\|F_p\|_\infty.
\]
Therefore, by \eqref{bdZF} we have
\begin{align*}
K(\lambda,\mu)-K(\lambda,0) \ll z_\ell^2(|\mu|+\mu^2).
\end{align*}
Hence, for some constants $C_1, C_2>0$,
\begin{equation}\label{eq:b2-cumulant-short}
K(\lambda,\mu)
\le
\frac{\lambda^2\sigma_+^2}{2}
+C_1 z_\ell^2(|\mu|+\mu^2)+C_2.
\end{equation}

For \(\ell\geq2\), we take
\begin{equation}
\label{lambdamu1}
 \lambda=\frac{u+w}{2\sigma_+^2},\qquad
 \mu=\eta\frac{w-u}{z_\ell^2},
\end{equation}
where \(\eta>0\) will be chosen sufficiently small. Then
\[
 |\lambda|\ll 1,\qquad |\mu|\ll \frac{c_\ell}{z_\ell^2},
\]
and, by \eqref{eq:b2-coefficient-sums-short},
\begin{equation}\label{eq:b2-mu-remainder}
 |\mu|\Big(\max_p\|F_p\|_\infty+
          \sum_p\|C_j(p)\|_\infty^2\|F_p\|_\infty\Big)=o(1).
\end{equation}

If $\ell=1$, we take instead
\begin{equation*}
\label{lambdamu2}
 \lambda=\frac{u+w}{2\sigma_+^2},\qquad
 \mu=(H+1)\sgn(w-u).
\end{equation*}
Here \(H>0\) is fixed, so \eqref{eq:b2-mu-remainder} also holds.

For an admissible pair of tuples $(\textbf{u})$ and $(\textbf{w})$, the  conditions in
$\mathcal I(u)$ and $\mathcal I(w)$ give
\begin{equation}\label{eq:b2-master-short}
 1\le
 \exp\bigg(-\lambda \frac{w+u}{2}-\mu (w-u)+C\big(|\lambda|+|\mu|\big)
 +\sum_{i\le\ell}\Big(
   \lambda\frac{u_i+w_i}{2}+\mu(w_i-u_i)\Big)\bigg).
\end{equation}
On the event in $P_i(u_i,w_i)$ we have
\[
 \Big|\frac{u_i+w_i}{2}-\mathcal Z_i\Big|
 \le\Delta_i^{-1}+\Delta_i^{-a_i}
 \qquad\text{and}\qquad
 \big|(w_i-u_i)-\mathcal F_i\big|
 \le2(\Delta_i^{-1}+\Delta_i^{-a_i}).
\]

Now we deal with the first block.  Fix all the $(u_i)_{i \geq 2}$ and $(w_i)_{i \geq 2}$. The conditions in the definition of $\mathcal{I}(u)$ and $\mathcal{I}(w)$ force $u_1$ and $w_1$ to be in intervals of
length $O(1)$. As in Proposition \ref{comparison21},
$\widetilde Y_1$ is confined to an interval $\widehat I$ of length $O(1)$, and
\begin{equation*}\label{eq:initP}
 \sum_{u_1,w_1}e^{\lambda(u_1+w_1)/2+\mu(w_1-u_1)}P_1(u_1,w_1)
 \ll
 \E_{\rm rand}\Big(e^{\lambda\mathcal Z_1+\mu\mathcal F_1}
          \mathbf1(\widetilde Y_1\in\widehat I)\Big).
\end{equation*}

Define the tilted probability $\P_{\lambda,\mu}$ by
\[
 \P_{\lambda,\mu}(E)
 :=
 \frac{
 \E_{\rm rand}(
 e^{\lambda\mathcal Z_1+
 \mu \mathcal{F}_1}
 \mathbf 1_E
)
 }{
 \E_{\rm rand} ( e^{\lambda\mathcal Z_1+
 \mu\mathcal{F}_1 })
 }.
\]
Then
\[
\E_{\rm rand}\Big(
 e^{\lambda\mathcal Z_1+
 \mu \mathcal{F}_1}
 \mathbf 1(\widetilde Y_1\in\widehat I)
\Big)
=
\E_{\rm rand} (e^{\lambda\mathcal Z_1+
 \mu \mathcal{F}_1})
\P_{\lambda,\mu}(\widetilde Y_1\in\widehat I).
\]
We write $\E_{\lambda,\mu}$ and
$\operatorname{Var}_{\lambda,\mu}$ for expectation and variance with
respect to this probability measure.
We have
\[
\sum_{p\leq T_1}\E_{\lambda,\mu}
 |C_j(p)-\E_{\lambda,\mu}C_j(p)|^3\ll1\qquad\text{and}\qquad \operatorname{Var}_{\lambda,\mu}(\widetilde Y_1)
=\sigma_{1,j}^2+O(1)\asymp\sigma_{\ell,j}^2 .
\]
The Berry--Esseen theorem therefore gives
\begin{equation*}\label{eq:BEpair}
 \P_{\lambda,\mu}(\widetilde Y_1\in\widehat I)
 \ll \frac1{\sigma_{\ell,j}}.
\end{equation*}

We now bound the error from the block $i=1$, keeping
$u_i,w_i$ with $i\geq2$ fixed.

If $\ell\geq2$, using \eqref{u1w1} we have that
\[
 \sum_{u_1,w_1}
 e^{\lambda(u_1+w_1)/2+\mu(w_1-u_1)}\varepsilon_1
 \ll  \Delta_1^2
 e^{-\Delta_1^2+O(V+Cc_1^2)}
 \le \frac{\E_{\rm rand}( e^{\lambda\mathcal Z_1+\mu\mathcal F_1})}
 {\sigma_{\ell,j}}.
\]
If \(\ell=1\), we use \(|u_1|+|w_1|\ll \log\log T\) and
\(|\mu|\leq H+1\) instead. The corresponding error is at most
\[
 \ll  \Delta_1^2
       \exp\big(-\Delta_1^2+O( \log\log T) \big)
 \leq\frac{\E_{\rm rand} (e^{\lambda\mathcal Z_1+\mu\mathcal F_1})}
                {\sigma_{\ell,j}},
\]
where we used Jensen's inequality in the last step. Hence
\begin{equation}\label{eq:b2-first-combined-short}
 \sum_{u_1,w_1}e^{\lambda(u_1+w_1)/2+\mu(w_1-u_1)}
 \bigl(P_1(u_1,w_1)+\varepsilon_1\bigr)
 \ll \frac{\E_{\rm rand} (e^{\lambda\mathcal Z_1+\mu\mathcal F_1})}
               {\sigma_{\ell,j}}.
\end{equation}

We now treat the subsequent blocks.  Recall  that $|u_i|+|w_i|\ll c_i$ for $i\ge2$.
For $i\ge2$, the enlarged rectangles meet without overlap
apart from endpoints.  Hence, as in Proposition \ref{comparison21},
\begin{align*}
 &\sum_{u_i,w_i}
 e^{\lambda(u_i+w_i)/2+\mu(w_i-u_i)}P_i(u_i,w_i) \le
 e^{O((|\lambda|+|\mu|)(\Delta_i^{-1}+\Delta_i^{-a_i}))}
 \E_{\rm rand}( e^{\lambda\mathcal Z_i+\mu\mathcal F_i}),
\end{align*}
and again as in Proposition \ref{comparison21}, we get that
\begin{align*}
 &\prod_{i=2}^{\ell}\sum_{u_i,w_i}
 e^{\lambda(u_i+w_i)/2+\mu(w_i-u_i)}
 \bigl(P_i(u_i,w_i)+C\varepsilon_i\bigr)
 \ll  \prod_{i=2}^{\ell}
 \E_{\rm rand} (e^{\lambda\mathcal Z_i+\mu\mathcal F_i}).
\end{align*}
Combining this with \eqref{eq:b2-first-combined-short} and \eqref{eq:b2-master-short} shows that
\begin{align*}
 &\P\Big(G_{\ell-1}\cap\big\{S_\ell^{(j)}\in[u-1,u]\big\}
                    \cap\{S_\ell^{(m)}\in[w-1,w]\}\Big) \nonumber\\
                    &\qquad\ll\frac1{\sigma_{\ell,j}}
 \exp\Big(-\lambda \frac{u+w}{2}-\mu (w-u)+C\big(|\lambda|+|\mu|\big)
                              +K(\lambda,\mu)\Big).
 \label{eq:b2-before-optimization}
\end{align*}

For \(\ell\geq2\), with the choices
\eqref{lambdamu1} and
using \eqref{eq:b2-cumulant-short}, the exponent above is at most
\[
 -\frac{(u+w)^2}{8\sigma_+^2}
 -\frac{\eta (w-u)^2}{z_\ell^2}
 +\frac{C_1\eta^2(w-u)^2}{z_\ell^2}
 +C \eta|w-u|+O(1).
\]
We choose \(\eta\) so that \(C_1\eta^2\leq\eta/4\). Since
\[
 C \eta|w-u|\leq\frac{\eta (w-u)^2}{4z_\ell^2}
                       +  C^2 \eta z_{\ell}^2,
\]
the exponent is at most
\[
 -\frac{(u+w)^2}{8\sigma_+^2}
 -\frac{\eta (w-u)^2}{2z_\ell^2}+O(z_\ell^2).
\]
Finally, \(\sigma_+^2=\sigma_{\ell,j}^2+O(z_\ell^2)\) and
\((u+w)^2\ll \sigma_{\ell,j}^4\), so this is
\[
 -\frac{(u+w)^2}{8\sigma_{\ell,j}^2}
 -C\frac{(w-u)^2}{z_\ell^2}+O(z_\ell^2),
\] which proves \eqref{eq:two-point}.

For \(\ell=1\), 
the choice \(\mu=(H+1)\sgn(w-u)\) gives
\(-\mu (w-u)=-(H+1)|w-u|\), and \eqref{eq:b2-cumulant-short} bounds
the remaining terms involving \(\mu\) by \(O(z_1^2)\).
This proves \eqref{eq:initial-two-point}.
\end{proof}

We now prove a simplified version of Proposition \ref{comparison21}, in the easier regime \eqref{first_regime}.

\begin{prop}\label{comparison1}
Assume \eqref{first_regime} and \(\alpha=o(1)\). Let \(x,S_x,\sigma_x\) be
as in \eqref{x_choice}, \eqref{sx}, \eqref{sigmax}, respectively.
For every \(A\geq1\), uniformly for \(|u|\leq A\sigma_x^2\), we have
\begin{equation}\label{cmp:large-result}
 \P\big(S_x\in[u-1,u]\big)
 \ll \frac1{\sigma_x}\exp\Big(-\frac{u^2}{2\sigma_x^2}\Big).
\end{equation}
\end{prop}

\begin{proof}

    We repeat the proof of Proposition \(\ref{comparison21}\) for the single sum \(S_x(t)\), with \(\mathcal Q=1\). We check below the polynomial length and variance estimates needed in this regime.

Take
\[
 \Delta=C\sigma_x,\qquad X=C \Delta^{10},
 \qquad \mathcal D=\mathcal D_{\Delta,X}
\]
from \eqref{eq:sandwich} with \(a=4\). 
Cover $[u-1,u]$ by consecutive intervals of length $\Delta^{-1}$,
allowing the last one to extend past $u$. We replace each interval $
[v,v+\Delta^{-1}]$
by $
[v-\Delta^{-4},\,v+\Delta^{-1}+\Delta^{-4}].
$

The degree of $\mathcal{D}$ is \(d\ll \sigma_x^{22}\ll (\log\frac{1}{\alpha})^{11}\).
Each factor obtained by expanding
\(|\mathcal D(S_x-v)|^2\) has length at most
\begin{equation*}\label{cmp:large-length}
 x^{2d}
 =T^{O \big((\log(1/\alpha))^{11}
                  \frac{\log(\alpha\log T)}{\alpha\log T}\big)}
 =T^{o(1)}.
\end{equation*}
Since \(\alpha\log x=2\log(\alpha\log T)\to\infty\), we have that \eqref{bpx} holds.
The coefficient bounds therefore give
\[
\sum_{p\leq x}m_p^2\ll\log\frac1\alpha,
\qquad
\sum_{p\leq x}m_p^3\ll 1,
\]
and the variance is \(\sigma_x^2+o(1)\).

The truncation estimate \eqref{bigyi} applies with
\(|v|\ll \log\frac{1}{\alpha}\). The arguments in
\eqref{eq:first-probability}, \eqref{eq:BE} then give
\[
 \P\big(S_x\in[u-1,u]\big)
 \ll \frac1{\sigma_x}
       \exp\Big(-\lambda u+
          \frac{\lambda^2}{2}\operatorname{Var}_{\rm rand}(\widetilde S_x)\Big)
       +O (\Delta e^{-\Delta^2})
\]
for bounded \(\lambda\). Choosing
\(\lambda=u/\operatorname{Var}_{\rm rand}(\widetilde S_x)\) and
 replacing the random variance by \(\sigma_x^2\)
prove \eqref{cmp:large-result}.
\end{proof}

\section{Proof of Theorems \ref{mainthm} and \ref{bigk} in the small $\alpha$ regime}
\label{section_largedev}

Throughout this section we assume \eqref{small_regime}, and let
$V=\kappa\log\frac{1}{\alpha}$ with
$0<\kappa_0\leq\kappa\leq\kappa_1<\infty$ fixed.
The implied constants in this section may depend on the
implied constants in \eqref{small_regime}.
In the range \eqref{small_regime}, we have
\[
 \log\frac1{\alpha}=\log\log T+O(\log\log\log T).
\]

 For Theorem~\ref{mainthm}, where $\kappa_1<1/2$,
we fix $2\kappa_1<\rho<1,$ while for Theorem~\ref{bigk},
we fix $0<\rho<1$ sufficiently close to $1$ such that
\begin{equation}\label{rho-choice1}
 (1-\rho)\Big(\kappa_1-\frac12\Big)
       +\frac{(1-\rho)^2}{4}<\frac{\varepsilon_1}{2}.
\end{equation}
Recall the definition of $r>1$ from \eqref{rdef} in the definition of $T^{\beta_j}$. We choose $r>1$ and two parameters $D>1$ and $a<1$ so that
\begin{equation}
\label{rho-choice2}
 \rho rD<a<1.
\end{equation}
Also recall that 
$\theta$ is sufficiently small, and  that
\begin{equation}\label{deltalocbound}
 \delta_{\rm loc}\leq\frac{1-a}{2},
\end{equation}
with $\delta_{\rm loc}$ given in equation \eqref{eq:delta-loc}.
Thus $$q\beta_m+\delta_{\rm loc}\leq\frac{1+a}{2}<1$$ whenever
$q=\lfloor a/\beta_m\rfloor$.

We will need the following auxiliary lemmas. 

\begin{lem}\label{lem:apriori}
For \(2\leq\ell\leq K\), 
\begin{equation*}
 \P(G_{\ell-1}\cap A_\ell^c)
 \ll 
 \frac{e^{-C c_\ell z_{\ell}}}{\sqrt{\log(1/\alpha)}}\exp\Big( - \frac{V^2}{\log(1/\alpha)} \Big),
\end{equation*}
for some $C>0$.
\end{lem}

\begin{proof}

As in \cite{ABR}, we have that
    \begin{equation*} 
    \label{ineqp}
    \P (  G_{\ell-1} \cap A_{\ell}^c ) \ll \sum_{j=\ell}^K \sum_{u \in [L_{\ell-1},U_{\ell-1}]} \E \bigg( \frac{(S_{\ell}^{(j)}-S_{\ell-1}^{(j)})^{2q}}{(2c_{\ell})^{2q}} \mathbf{1}\big(G_{\ell-1}^{(j)}(u)\big)\bigg).
    \end{equation*}
We choose $ q = [\delta c_{\ell}]$ with $\delta>0$ small, and note that 
condition \eqref{eq:length-condition} is satisfied. 
Using Proposition \ref{comparison21}, after writing $S_{\ell}^{(j)}-S_{\ell-1}^{(j)} = \Re \mathcal{Q}(t)$ and bounding $|\Re \mathcal Q(t)|^{2q}\le |\mathcal Q(t)|^{2q}$, we have that 
\begin{align*}
 &\E \bigg(  \frac{(S_{\ell}^{(j)}-S_{\ell-1}^{(j)})^{2q}}
 {(2c_{\ell})^{2q}} \mathbf{1}\big(G_{\ell-1}^{(j)}(u)\big) \bigg) \ll
 \frac{\E_{\rm rand} \big(| \mathcal{\widetilde  Q}|^{2q}\big)}
 {(2c_{\ell})^{2q} \sigma_{\ell-1,j}}
 \exp\Big(-\frac{u^2}{2 \sigma_{\ell-1,j}^2}+O( z_{\ell}^2)\Big).
\end{align*}
  Summing over $u \in [L_{\ell-1},U_{\ell-1}]$ and over $j \geq \ell$, and using Lemma \ref{lem_moments_random}, we obtain that 
  \begin{align*}
  \P (  G_{\ell-1} \cap A_{\ell}^c) &  \ll \frac{(K-\ell+1)c_{\ell-1}}{\sigma_{\ell-1,j}}
 \frac{(2q)!(C z_{\ell}^2 )^q}{2^q q!(2c_{\ell})^{2q}} \exp \Big( - \frac{L_{\ell-1}^2}{2 \sigma_{\ell-1,j}^2}+O(z_{\ell}^2) \Big),
  \end{align*}  
  for some $C>0$.
Using Stirling's formula and the choice \eqref{new_corridor}, we get that the above is 
$$ \ll  \frac{e^{-C c_\ell z_{\ell}}}{\sigma_{\ell-1,j} }  \exp \Big(  - \frac{V^2}{2 \sigma_{\ell-1,j}^2}\Big) \ll 
 \frac{e^{-C c_\ell z_{\ell}}}{\sqrt{\log(1/\alpha)}}\exp\Big( - \frac{V^2}{\log(1/\alpha)} \Big),$$ 
where we used the fact that in the regime \eqref{small_regime}, we have 
\begin{equation}
\label{touse} \sigma_{\ell,j}^2= \frac{1}{2}\log\frac{1}{\alpha}+O(1),
\end{equation}
by the variance estimate \eqref{new_variance}. This proves the lemma.
\end{proof}

\begin{lem}
\label{upper}
    For $2 \leq \ell \leq K$, we have 
   $$ \P \Big(G_{\ell-1} \cap A_{\ell} \cap\big \{ \exists j \geq \ell: S_{\ell}^{(j)}>U_{\ell}\big\}\Big) \ll  \frac{e^{-C c_{\ell}}}{\sqrt{\log(1/\alpha)}} \exp\Big( - \frac{V^2}{\log(1/\alpha)} \Big), $$ for some $C>0$. 
\end{lem}
\begin{proof}
  Note that on the event $G_{\ell-1} \cap A_{\ell}$ we have that $S_{\ell}^{(j)} \leq U_{\ell-1}+2c_{\ell} \leq U_{\ell}+3c_{\ell}=:U_{\ell}'$. Combining with the hypothesis, it follows that $S_{\ell}^{(j)} \in [U_{\ell},U_{\ell}'].$  So
  \begin{align}
  \label{int4}
   \P \Big(G_{\ell-1} \cap A_{\ell} \cap\big \{ \exists j \geq \ell: S_{\ell}^{(j)}>U_{\ell}\big\}\Big)   \leq  \sum_{j =\ell}^K \sum_{u \in [U_{\ell},U_{\ell}']} \P \Big(G_{\ell-1} \cap \big\{  S_{\ell}^{(j)} \in [u-1,u]\big\}\Big). 
  \end{align}
  Note that because we are in the regime \eqref{small_regime}, we have that $|u| \leq A \sigma_{\ell,j}^2$, for some $A>0$, and hence we can apply Proposition \ref{comparison21} which gives that 
  \begin{align*}
      \eqref{int4}  &\ll \sum_{j=\ell}^K \sum_{u \in [U_{\ell},U_{\ell}']} \frac{1}{\sigma_{\ell,j}} \exp\Big(- \frac{u^2}{2\sigma_{\ell,j}^2}\Big) \ll   \frac{(K-\ell+1) c_{\ell}}{\sigma_{\ell,j}} \exp\Big(- \frac{U_{\ell}^2}{2\sigma_{\ell,j}^2}\Big) \\
      & \ll  \frac{e^{-C c_{\ell}}}{\sqrt{\log(1/\alpha)}} \exp\Big( - \frac{V^2}{\log(1/\alpha)} \Big), 
  \end{align*}
  for some $C>0$, where we used equation \eqref{touse} again.
\end{proof}

\begin{lem}
\label{lower}
Suppose $0<\rho<1$,
with the parameters chosen as in \eqref{rho-choice1} and \eqref{rho-choice2}.
For $2\leq\ell\leq K$, we have
\begin{align}\label{firstbound63}
 &\P\Big(H\cap G_{\ell-1}\cap A_\ell
       \cap\big\{ \exists j \geq \ell : S_{\ell}^{(j)}<L_{\ell} \big\}\Big)
 \ll\frac{e^{-C c_\ell}}{\sqrt{\log(1/\alpha)}}
       \exp\Big(-\rho V+\frac{\rho^2}{4}\log\frac1\alpha\Big),
\end{align}
for some $C>0$. 
If, in addition, $2\kappa_1<\rho$, then
    \begin{equation} \label{improved_bd} \P \Big(H \cap G_{\ell-1} \cap A_{\ell} \cap \big\{ \exists j \geq \ell : S_{\ell}^{(j)}<L_{\ell}\big \}\Big) \ll \frac{e^{-C c_{\ell}}}{\sqrt{\log(1/\alpha)}}  \exp \Big(  - \frac{V^2}{\log(1/\alpha)}\Big),
    \end{equation}for some $C>0$.
\end{lem}
\begin{proof}
    Note that on the event $G_{\ell-1} \cap A_{\ell}$, we have that $S_{\ell}^{(j)} > L_{\ell}':=L_{\ell}-3c_{\ell}.$ Therefore, 
    \begin{align}
\label{prelim}
 &\P \Big(H \cap G_{\ell-1} \cap A_{\ell} \cap\big
 \{ \exists j \geq \ell : S_{\ell}^{(j)}<L_{\ell}\big\}\Big)\notag\\
 &\qquad\leq \sum_{j=\ell}^K\sum_{u \in [L_{\ell}',L_{\ell}]}
 \P \Big(\big\{\log |\zeta(\tfrac12+\alpha+it)|^{-1}
 -S_{\ell}^{(j)}>V-u\big\}
 \cap G_{\ell}^{(j)}(u) \cap A_\ell\Big).
\end{align}

Recall that $D>1$ is sufficiently close to $1$. 
Using Lemma \ref{bf_lemma} with $x=T_m$, we have that 
    $$ \log \big|\zeta(\tfrac12+\alpha+it)\big|^{-1} \leq S_m^{(m)} + R_m+C, $$ for some constant $C>0$, where 
    $$R_m = \frac{1}{\beta_m} \log \frac{1}{1-T_m^{-\alpha}}.$$
    Hence if 
    $$  \log \big|\zeta(\tfrac12+\alpha+it)\big|^{-1}  -S_{\ell}^{(j)}>V-u,$$
    then 
    \begin{equation}
    \label{vu} V-u < S_m^{(m)}-S_{\ell}^{(j)} +R_m+C.
    \end{equation}
    
    Note that we have 
    $$c_{\ell} < V-u <C_1 c_{\ell},$$ for some constant $C_1>0$. We first consider the case
\[
D\frac{\log(V-u)}{V-u}>\beta_K.
\]
Note that this implies that $V-u = O(1)$, and hence $c_{\ell} = O(1)$ and $z_{\ell} =O(1)$. Applying Proposition \ref{comparison21}
with $\mathcal Q=1$, we obtain that
\begin{align*}
\P\bigl(H\cap G_\ell^{(j)}(u)\cap A_\ell\bigr)
&\leq \P\bigl(G_\ell^{(j)}(u)\bigr)
\ll
\frac1{\sqrt{\log(1/\alpha)}}
\exp\Big(-\frac{V^2}{\log(1/\alpha)}\Big),
\end{align*}
where we used equation \eqref{touse} again. There are $O(1)$ relevant choices of $u$ and $j$.
Since $c_\ell=O(1)$, we may also include the factor
$e^{-Cc_\ell}$ in the upper bound above, which proves the required Gaussian bound in this case.

We now assume that
\[
D\frac{\log(V-u)}{V-u}\leq\beta_K.
\]
Choose $m$ to be the smallest index satisfying
\[
\beta_m\geq D\frac{\log(V-u)}{V-u}.
\]
Such an index now exists with $m\leq K$. Moreover,
$V-u\asymp c_\ell$ and $\beta_\ell=c_\ell^{-90}$ imply,
for $\theta$ sufficiently small, that
\[
\beta_\ell<
D\frac{\log(V-u)}{V-u}.
\]
Thus $m>\ell$, and the minimality of $m$ gives
\begin{equation}
\label{min_cond}
D\frac{\log(V-u)}{V-u}
\leq\beta_m<
rD\frac{\log(V-u)}{V-u}.
\end{equation}

From the choice of $m$, it follows that 
$$ R_m \leq \frac{V-u}{D \log(V-u)} \log \frac{1}{1-e^{-M_0 D \log(V-u)/(V-u)}} \leq \frac{V-u}{D \log(V-u)} \log \Big( 1 + \frac{V-u}{M_0 D \log(V-u)}\Big).$$
Since $V-u\geq c_\ell\geq c_K \geq \theta^{-1/90}$, the preceding
estimate gives, for $\theta$ sufficiently small,
$$ R_m+C \leq \frac{V-u}{D}.$$
Combining this with \eqref{vu} we get
$$ S_m^{(m)}-S_{\ell}^{(j)}> (V-u) \Big(  1 - \frac{1}{D}\Big).$$
Hence, by writing 
    $$S_m^{(m)}-S_{\ell}^{(j)} = S_m^{(m)} - S_{\ell}^{(m)}+S_{\ell}^{(m)}- S_{\ell}^{(j)},$$ we obtain that
    \begin{align}
        \P \big(H \cap G_{\ell}^{(j)}(u) \cap A_{\ell} \big) &\leq \P \Big(  G_{\ell}^{(j)}(u) \cap\big \{S_m^{(m)}-S_{\ell}^{(m)}> W\big \} \Big) \nonumber \\
        &\qquad+ \P \Big(  G_{\ell}^{(j)}(u) \cap A_{\ell} \cap \big\{S_{\ell}^{(m)}-S_{\ell}^{(j)}> W  \big\}  \Big),\label{second}
    \end{align}
where $W= \frac{D-1}{2D} (V-u).$

We first bound the first probability on the right hand side of \eqref{second}. Write $S_m^{(m)}-S_\ell^{(m)}=\Re A(t)$, and for $q \in \mathbb{N}$, let $\mathcal{Q} = A(t)^q$. Note that $\mathcal{Q}$ is only supported on integers whose primes are greater than $T_{\ell}$. Using the fact that $|\Re A(t)|^{2q}\le |A(t)|^{2q}$, we have 
    \begin{align}
    \label{toapply}
 \P \Big(  G_{\ell}^{(j)}(u) \cap \big\{S_m^{(m)}-S_{\ell}^{(m)}> W\big \} \Big) \leq   \E \bigg( \frac{|\mathcal{Q}|^2}{W^{2q}} \mathbf 1(G_{\ell}^{(j)}(u))\bigg) .
    \end{align}
Now choose
\begin{equation} 
\label{second_choice}
q = \Big\lfloor \frac{a}{\beta_m} \Big\rfloor.
\end{equation} Note that the length of $\mathcal{Q}$ is at most 
$$T_m^q \leq T^{ \beta_m q} <T,$$ and the condition \eqref{eq:length-condition} in Proposition \ref{comparison21} follows from \eqref{deltalocbound}. Proposition \ref{comparison21} then gives 
\begin{align} 
\label{from_comp}
    \E \bigg( \frac{|\mathcal{Q}|^2}{W^{2q}} \mathbf 1\big(G_{\ell}^{(j)}(u)\big)\bigg)  \ll \frac{1}{\sigma_{\ell,j}} \exp \Big( - \frac{u^2}{2 \sigma_{\ell,j}^2} \Big)\frac{ \E_{\text{rand}}(|\mathcal{\widetilde Q}|^2)}{W^{2q}}.
\end{align}
Using Lemma \ref{lem_moments_random} and \eqref{new_variance} gives that
\begin{align} \label{expect}
    \E_{\text{rand}}\big(| \mathcal{\widetilde Q}|^2\big) \ll q! (Cz_{\ell}^2)^q,
\end{align}
for some $C>0$. Using Stirling's formula and combining equations \eqref{toapply}, \eqref{from_comp} and \eqref{expect}, it follows that 
\begin{align*}
   &\P \Big(  G_{\ell}^{(j)}(u) \cap \big\{S_m^{(m)}-S_{\ell}^{(m)}> W \big\} \Big)\\
   &\qquad\ll \frac{1}{\sigma_{\ell,j}} \exp \Big(  - \frac{u^2}{2 \sigma_{\ell,j}^2}\Big)\exp \Big( q \log q+2q \log z_{\ell} -2q \log W +O (q)\Big) .
\end{align*}
Hence by \eqref{min_cond} and \eqref{second_choice}, we have
 \begin{align}
 \label{onebound}
    & \P \Big(  G_{\ell}^{(j)}(u) \cap \big\{S_m^{(m)}-S_{\ell}^{(m)}> W\big \} \Big)\nonumber\\
    &\qquad\ll \frac{1}{\sigma_{\ell,j}} \exp \Big(  - \frac{u^2}{2 \sigma_{\ell,j}^2}\Big) \exp \bigg(  - \frac{a(V-u)}{rD}+ O \Big( \frac{(V-u) \log \log (V-u)}{\log(V-u)} \Big) \bigg).
 \end{align}
Since $\rho < \frac{a}{rD}$, we get that 
\begin{equation}
    \label{onebound1}
    \eqref{onebound} \ll \frac{1}{\sqrt{\log(1/\alpha)}} \exp \Big(  - \frac{u^2}{2 \sigma_{\ell,j}^2}-\rho(V-u)\Big).
\end{equation}

Now we bound the second term on the right hand side of \eqref{second}. Similarly as before, since we are on the event $G_{\ell-1} \cap A_{\ell}$, it follows that 
   $S_{\ell}^{(m)} \in [L_{\ell}',U_{\ell}']$. We condition on $S_{\ell}^{(m)} \in [w-1,w]$, and then we must have $w \in [L_{\ell}',U_{\ell}']$. Since $S_{\ell}^{(m)}-S_{\ell}^{(j)}>W$, we get that $w-u>W-1$. Note that with $\theta$ small enough, we have that $W-1>\frac{W}{2}$, so after applying Proposition \ref{comparison22}, we get that 
   \begin{align*}
    &\P \Big(  G_{\ell}^{(j)}(u) \cap A_{\ell} \cap \big\{S_{\ell}^{(m)}-S_{\ell}^{(j)}> W \big \}  \Big) \\
    &\qquad\ll \frac{1}{\sqrt{\log(1/\alpha)}} \sum_{\substack{w \in [L_{\ell}',U_{\ell}'] \\ w-u> \frac{W}{2}}}  \exp\Big(
  -\frac{(u+w)^2}{8\sigma_{\ell,j}^2}
 -C\frac{(w-u)^2}{z_\ell^2} +O(z_\ell^2)
 \Big). 
   \end{align*}

Note that $w>u$, and hence
$(u+w)^2/(8 \sigma_{\ell,j}^2)> u^2/(2 \sigma_{\ell,j}^2)$. We also have that
\begin{align*}
    \frac{(w-u)^2}{z_{\ell}^2} \geq \frac{(D-1)^2}{16D^2} \frac{c_{\ell}(V-u)}{z_{\ell}^2} \geq (\rho+1) (V-u).
\end{align*}
 Summing over $w$, we get that 
 \begin{align*}
    \P \Big(  G_{\ell}^{(j)}(u) \cap A_{\ell} \cap \big\{S_{\ell}^{(m)}-S_{\ell}^{(j)}> W \big \}  \Big) \ll \frac{1}{\sqrt{\log(1/\alpha)}} \exp \Big( - \frac{u^2}{2\sigma_{\ell,j}^2} - \rho(V-u) \Big).
    \end{align*}
    Combining this with \eqref{onebound1} and \eqref{prelim} gives that
    \begin{align*}
        \P \Big(H \cap G_{\ell-1} \cap A_{\ell} \cap
 \big\{ \exists j \geq \ell : S_{\ell}^{(j)}<L_{\ell}\big\}\Big) \ll  \frac{1}{\sqrt{\log(1/\alpha)}} \sum_{j=\ell}^K \sum_{u \in [L_{\ell}',L_{\ell}]} \exp \Big( - \frac{u^2}{2\sigma_{\ell,j}^2} - \rho(V-u) \Big).
    \end{align*}
    
The preceding argument used only $\rho<a/(rD)$, so it also
applies with $\rho$ replaced by $(\rho+a/(rD))/2$. Hence
\begin{align*}\label{eq:lower-cell-reserved}
 \P(H\cap G_\ell^{(j)}(u)\cap A_\ell)
 &\ll\frac1{\sigma_{\ell,j}}
 \exp\bigg(-\frac{u^2}{2\sigma_{\ell,j}^2}-\rho(V-u)
       -\frac12\Big(\frac a{rD}-\rho\Big)(V-u)\bigg).
\end{align*}
Using the fact that
\begin{equation}
\label{fact}
- \frac{u^2}{2\sigma_{\ell,j}^2} \leq -\frac{u^2}{\log(1/\alpha)}+ C z_{\ell}^2,
\end{equation}
for some $C>0$, it follows that
\begin{align*}
 -\frac{u^2}{2\sigma_{\ell,j}^2}-\rho(V-u)
 &\leq-\rho V+\frac{\rho^2}{4}\log\frac1\alpha
       -\frac{\bigl(u-\frac\rho2\log(1/\alpha)\bigr)^2}     {\log(1/\alpha)}+Cz_\ell^2\\
 &\leq-\rho V+\frac{\rho^2}{4}\log\frac1\alpha +C z_\ell^2.
\end{align*}
There are $O(c_\ell)$ choices of $u$ and $O(z_\ell)$
choices of $j$. We can choose $c_K$ sufficiently large (or, equivalently, $\theta$ sufficiently small) such
that uniformly for $\ell\leq K$,
\[
 C z_\ell^2+\log\bigl(Cc_\ell(1+z_\ell)\bigr)
 \leq\frac14\Big(\frac a{rD}-\rho\Big)c_\ell.
\]
Hence, combining the equations above and using the fact that $V-u \geq c_{\ell}$, we get that 
\begin{align*}
    \sum_{j=\ell}^K \sum_{u \in [L_{\ell}',L_{\ell}]} 
 \P\big(H\cap G_\ell^{(j)}(u)\cap A_\ell\big)
 &\ll\frac{1}{\sqrt{\log(1/\alpha)}} \exp \Big(-\rho V + \frac{\rho^2}{4} \log \frac{1}{\alpha}- \frac{1}{4} \Big(  \frac{a}{rD}-\rho\Big) c_{\ell} \Big) ,
\end{align*}
which proves the bound \eqref{firstbound63}.

Finally, assume $2\kappa_1<\rho$.
Using \eqref{fact} and the fact that 
$u^2=V^2-2V(V-u)+(V-u)^2$ gives
  $$  - \frac{u^2}{2\sigma_{\ell,j}^2} - \rho(V-u) \leq - \frac{V^2}{\log(1/\alpha)}-(\rho-2\kappa) (V-u)+Cz_{\ell}^2. $$
  Since $\rho-2\kappa_1 >0$ and $V-u \geq c_{\ell}$, summing over $j$ and $u$ then yields
  \begin{align*}
     \P \Big(H \cap G_{\ell-1} \cap A_{\ell} \cap
\big \{ \exists j \geq \ell : S_{\ell}^{(j)}<L_{\ell}\big\}\Big) \ll  \frac{e^{-C c_{\ell}}}{\sqrt{\log(1/\alpha)}}  \exp \Big(  - \frac{V^2}{\log(1/\alpha)}\Big) ,
  \end{align*}
    for some $C>0$. This proves \eqref{improved_bd}.
\end{proof}

\begin{lem}
\label{g1c}
Suppose $0<\rho<1$,
with the parameters chosen as in \eqref{rho-choice1} and \eqref{rho-choice2}.
We have
\begin{equation}\label{eq:initial-general}
 \P(H\cap G_1^c)
 \ll \frac{e^{-C c_1}}{\sqrt{\log(1/\alpha)}} \exp\Big(-\rho V+\frac{\rho^2}{4}\log\frac1\alpha \Big),
\end{equation}
for some $C>0$.
If, in addition, $2\kappa_1<\rho$, then
\begin{equation}\label{eq:initial-result}
 \P(H\cap G_1^c)
 \ll\frac{e^{-C c_1}}{\sqrt{\log(1/\alpha)}}
           \exp\Big(-\frac{V^2}{\log(1/\alpha)}\Big)
\end{equation}
for some $C>0$.
\end{lem}
\begin{proof}
The proof is similar to the proofs of Lemmas \ref{lower} and \ref{upper}, except that we first restrict the values of $S_1^{(j)}$ to intervals $[u-1,u]$ with $|u| \leq C \log \log T$, for some large constant $C>0$. Indeed, using Markov's inequality and Lemma \ref{lem_moments}, it follows that
$$ \P \Big(  \max_{1\leq m \leq K} |S_1^{(m)}| > C \log \log T \Big) \ll \exp(-c \log \log T),$$ for some $c>0$. Note that this is smaller than the bound in the statement of the lemma. Writing
\begin{align}
 \P(H\cap G_1^c)
 &\leq
 \P\Big(\max_{1\leq m\leq K}|S_1^{(m)}|>C \log\log T\Big)
\nonumber \\
 &\qquad+\sum_{j=1}^{K}
 \P\Big(H\cap
   \big\{\max_{1\leq m\leq K}|S_1^{(m)}|\leq C \log\log T\big\}
   \cap\{S_1^{(j)}>U_1\}\Big)
  \label{3terms}\\
 &\qquad+\sum_{j=1}^{K}
 \P\Big(H\cap
   \big\{\max_{1\leq m\leq K}|S_1^{(m)}|\leq C \log\log T\big\}
   \cap\{S_1^{(j)}<L_1\}\Big),\nonumber
\end{align}
we are left with the second and the third probabilities on the right hand side of \eqref{3terms}. 

For the second term in \eqref{3terms}, the same argument as in Lemma \ref{upper} goes through yielding the bound \eqref{eq:initial-result}, and also \eqref{eq:initial-general}.

For the third term, proceeding as in Lemma \ref{lower}, on $H \cap \{S_1^{(j)} \in [u-1,u]\}$, we must have that 
$$ S_m^{(m)}-S_1^{(j)} > \frac{D-1}{D}(V-u),$$
where we similarly choose $m$ to be the smallest index with 
$$ \beta_m \geq D \frac{ \log(V-u)}{V-u}.$$ Here, 
$$ -C \log \log T \leq u \leq L_1+1,$$ and note that 
$c_1-1 \leq V-u \ll \log \log T$. As before, we write
 \begin{align}\label{2prob}
 &\P\Big(H \cap \big\{ S_1^{(j)}\in[u-1,u] \big\} \cap 
             \big\{  \max_{m \le K}|S_1^{(m)}|\le C \log\log T\big \} \Big)\nonumber\\
             &\qquad\leq
 \P\Big( \big\{ S_1^{(j)}\in[u-1,u] \big\} \cap 
          \big \{S_m^{(m)}-S_1^{(m)}>W \big\}\Big)  \\
 &\qquad\qquad+
 \P\Big(\big \{S_1^{(j)}\in[u-1,u]\big\} \cap 
          \big \{S_1^{(m)}-S_1^{(j)}>W \big\} \cap 
          \big\{ \max_{m\le K}|S_1^{(m)}|\le C \log\log T\big \} \Big), \nonumber
 \end{align}
with $W= \frac{D-1}{2D} (V-u)$. 

For the first probability in \eqref{2prob},  as in Lemma \ref{lower}, take \(q=\lfloor a/\beta_m\rfloor\) and
$$
 \mathcal Q(t)=
 \Big(
 \sum_{T_1<p\le T_m}\frac{b_\alpha(p;m)}{p^{1/2+\alpha+it}}
 +\sum_{\substack{T_1<p\\p^2\le T_m}}
       \frac{b_\alpha(p^2;m)}{p^{1+2\alpha+2it}}
 \Big)^q.
$$
Proposition \ref{comparison21}  then implies that
 \begin{align*}
 &\P\Big(\big\{S_1^{(j)}\in[u-1,u] \big\} \cap 
          \big\{ S_m^{(m)}-S_1^{(m)}>W \big\}\Big)\ll 
 \frac1{\sigma_{1,j}}\exp\Big(-\frac{u^2}{2\sigma_{1,j}^2}\Big)
       \frac{\E_{\rm rand}|\widetilde{\mathcal Q}|^2}{W^{2q}}.
 \end{align*}

By Corollary \ref{cor_b}, we have
 \begin{align*}
 &\sum_{T_1<p\le T_m}\frac{|b_\alpha(p;m)|^2}{p^{1+2\alpha}}
 +\sum_{\substack{T_1<p\\p^2\le T_m}}
       \frac{|b_\alpha(p^2;m)|^2}{p^{2+4\alpha}}
 \ll z_1^2.
 \end{align*}
Hence Lemma \ref{lem_moments_random} yields
\[
 \E_{\rm rand}|\widetilde{\mathcal Q}|^2\ll q!(Cz_1^2)^q.
\]
Proceeding as in Lemma \ref{lower} we get that 
the first probability in \eqref{2prob} is 
$$
 \ll\frac1{\sigma_{1,j}}
       \exp\Big(-\frac{u^2}{2\sigma_{1,j}^2}-\rho (V-u)-\frac12\Big(\frac a{rD}-\rho\Big)(V-u)\Big).
$$

For the second probability in \eqref{2prob}, suppose 
 $S_1^{(m)} \in [w-1,w],$ with
$$
 -C \log\log T\le w\le C\log\log T+1.
$$
We use Lemma \ref{comparison21} (in particular, the bound \eqref{eq:initial-two-point}). Note that we have
$$
 -\frac{(u+w)^2}{8\sigma_{1,j}^2}
 =-\frac{u^2}{2\sigma_{1,j}^2}
   -\frac{u(w-u)}{2\sigma_{1,j}^2}
   -\frac{(w-u)^2}{8\sigma_{1,j}^2}
 \le-\frac{u^2}{2\sigma_{1,j}^2}+\frac A2(w-u).
$$
Hence, for some $H$ large enough, we have that 
 \begin{align*}
 &\P\Big( \big\{S_1^{(j)}\in[u-1,u]\big\} \cap 
           \big\{S_1^{(m)}-S_1^{(j)}>W \big\} \cap 
         \big \{ \max_{m\le K}|S_1^{(m)}|\le C \log\log T \big\} \Big)\\
 &\qquad\ll
 \frac{\exp(-u^2/(2\sigma_{1,j}^2)+O(z_1^2))}
      {\sigma_{1,j}}
 \sum_{w-u>W/2}
       e^{-(H-A/2)(w-u)}\\
  &\qquad\ll
 \frac1{\sigma_{1,j}}
 \exp\Big(-\frac{u^2}{2\sigma_{1,j}^2}
           -\frac a{rD}(V-u)+O(z_1^2)\Big)\\
 &\qquad\ll
 \frac1{\sigma_{1,j}}
 \exp\Big(-\frac{u^2}{2\sigma_{1,j}^2}-\rho (V-u)-\frac12\Big(\frac a{rD}-\rho\Big)(V-u)\Big).
 \end{align*}
In the last step we used $z_1^2=o(c_1)$ and $V-u\ge c_1-1$. Thus, combining the bounds for the two probabilities in \eqref{2prob} gives that 
 \begin{align*}
 &\P\Big(H \cap \big\{ S_1^{(j)}\in[u-1,u] \big\} \cap 
            \big\{   \max_{m\le K}|S_1^{(m)}|\le C \log\log T\big \} \Big)\\
 &\qquad\ll
 \frac1{\sqrt{\log(1/\alpha)}}
 \exp\bigg(-\frac{u^2}{\log(1/\alpha)}
           -\rho(V-u)
           -\frac12\Big(\frac a{rD}-\rho\Big)(V-u)
           +O(z_1^2)\bigg).
 \end{align*}
Similar argument as in the end of Lemma \ref{lower} finishes the proof of the lemma. 
\end{proof}

Finally, the following lemma holds.
\begin{lem}
\label{terminal_event}
    We have that
    $$ \P (  H \cap G_K ) \ll \frac{1}{\sqrt{\log(1/\alpha)}} \exp\Big(  - \frac{V^2}{\log(1/\alpha)}\Big).$$
\end{lem}
\begin{proof}
    Note that on $G_K$, we have that $S_K^K \in [L_K,U_K]$, and $L_K=V+O(1)$, $U_K=V+O(1).$ Hence
    $$ \P (  H \cap G_K ) \leq \sum_{u \in [L_K,U_K]} \P \big(  G_K^{(K)}(u)\big) \ll \frac{1}{\sqrt{\log(1/\alpha)}} \exp\Big(  - \frac{V^2}{\log(1/\alpha)}\Big),$$ where we used Proposition \ref{comparison21} and the fact that there are $O(1)$ unit intervals in $[L_K,U_K]$.
\end{proof}

\begin{proof}[Proof of Theorems \ref{mainthm} and \ref{bigk}]
For the proof of Theorem \ref{mainthm}, we have that $2 \kappa_1 <\rho$. The proof now follows by combining \eqref{ph} and Lemmas \ref{lem:apriori}, \ref{upper}, \ref{lower} (namely equation \eqref{improved_bd}), \ref{g1c} (namely equation \eqref{eq:initial-result}) and \ref{terminal_event}.

For the proof of Theorem \ref{bigk}, we have $\rho$ as in \eqref{rho-choice1}. Combining Lemmas \ref{lem:apriori}, \ref{upper}, \ref{lower} (namely equation \eqref{firstbound63}), \ref{g1c} (namely equation \eqref{eq:initial-general}) and \ref{terminal_event}, we have that
\begin{equation}\label{eq:small-linear-rho}
 \P(H)\ll\frac1{\sqrt{\log(1/\alpha)}}
 \exp\Big(-\rho V+\frac{\rho^2}{4}\log\frac1\alpha\Big).
\end{equation}
Finally, by the choice \eqref{rho-choice1} and the identity
\[
 -\rho\kappa+\frac{\rho^2}{4}
 =-\kappa+\frac14
   +(1-\rho)\Big(\kappa-\frac12\Big)
   +\frac{(1-\rho)^2}{4},
\]
we have
$-\rho\kappa+\rho^2/4\leq-\kappa+1/4+\varepsilon_1/2$.
Substitution into \eqref{eq:small-linear-rho} yields
\[
 \P(H)\ll
 \exp\bigg(-V+\Big(\frac14+\varepsilon_1\Big)
                      \log\frac1\alpha\bigg),
\]
which finishes the proof.    
\end{proof}

\section{Negative moments of $\zeta(s)$}
\label{section_negativemom}

Before proving Theorem \ref{negativemom}, we need to prove the following lemma.

\begin{lem}\label{lem:negative-moment-high-tail}
Assume RH. Fix $0<c<1$ and $0<\rho<1$. Suppose that
\begin{equation}\label{bigrange}
 \frac1{(\log\log\log T)^c}
 \ll \alpha\log T\ll \log\log T.
\end{equation}
There exist $C_0,C_1,a>0$
such that for
\[
 V_0=C_0\log\log T
 \Big(\frac{\log\log\log T}{\log\log\log\log T}\Big)^2,
\]
we have
\begin{equation}\label{eq:negative-moment-high-tail}
 \P(W_\alpha>V)\ll 
 \begin{cases}
 e^{-\rho V},&C_1\log\log T\leq V<V_0,\\[2mm]
 \displaystyle\exp\Big(-V-a\frac{V}{\log V}\Big),&V\geq V_0.
 \end{cases}
\end{equation}
\end{lem}

\begin{proof}
We first prove the bound for $V\geq V_0$.
 First,
\cite[Lemma~2.2, equation~(2.8)]{BF} gives
\begin{equation}\label{eq:above-half-pointwise-cap}
 W_\alpha(t)\leq V_*:=
 \frac{\log T}{2\log\log T}
 \Big(\log\frac1\alpha-\log\log\log T+C\Big)
 \asymp\log T,
\end{equation}
for some $C>0$.
It suffices to consider $V_0\leq V\leq V_*$.
Take
\begin{equation*}\label{eq:endpoint-cutoff}
 J=\lfloor\log V\rfloor,\qquad
 \beta=\frac JV,\qquad x=T^{J/V}.
\end{equation*}
Using Lemma $2.1$ in \cite{BF}, we have that
\begin{align}
\label{ineq_bf}
 W_\alpha(t)
 &\leq\Re Q_x(t)
  +\frac{V}{J}\log\frac1{1-T^{-\alpha J/V}}
  +O\Big(\log\log\log T+
                     \max \Big\{0,  \log \frac1{\alpha  \beta\log T } \Big\}\Big),
\end{align}where
$$ Q_x(t) =  \sum_{p \leq x} \frac{b_{\alpha}(p;x)}{p^{1/2+\alpha+it}}.$$

From \eqref{bigrange}, we have that
$$ \frac{V}{J} \log \frac1{1-T^{-\alpha J/V}} \leq \frac{V}{J} \big( \log V - \log J + c \log \log \log \log T+ O(1) \big),$$
and hence  
\begin{align*}
 V-\frac VJ\log\frac1{1-T^{-\alpha J/V}}
 & \geq\frac VJ\big(\log J- c \log \log \log \log T - O(1)\big)
 \geq\frac{3(1-c)V\log J}{4J}.
 \label{eq:endpoint-remainder-gap}
\end{align*}
So, if $W_{\alpha}(t) >V$, then $$|Q_x(t) | >w:= \frac{(1-c)V \log J}{4J}.$$
Using Lemma \ref{asympta}, we also have that 
\begin{equation}\label{bp_again}
 |b_\alpha(p;x)|\ll
 1+J\frac{\log p}{\log x}\Big(1-\frac{\log p}{\log x}\Big).
\end{equation}

Now we divide the primes in consecutive intervals
$1,T^{1/V},T^{2/V},\ldots,T^{J/V}=x$, and write
\begin{equation}\label{eq:high-tail-prime-blocks}
 Q_x=\sum_{j=0}^{J-1}Q_j,\qquad
 Q_j(t)=\sum_{T^{j/V}<p\leq T^{(j+1)/V}}
             \frac{b_\alpha(p;x)}{p^{1/2+\alpha+it}}.
\end{equation}
For the primes in $Q_0$, using \eqref{bp_again} we have that $b_{\alpha}(p;x) \ll 1$.  For $1\leq j<J$, the primes from $Q_j$ satisfy (by
\eqref{bp_again} again) 
$b_\alpha(p;x)\ll J-j$, and hence
\begin{align}
 \sum_{p\leq T^{1/V}}\frac{|b_\alpha(p;x)|^2}{p^{1+2\alpha}}
 &\ll\log\log T,\notag\\
 \sum_{T^{j/V}<p\leq T^{(j+1)/V}}
       \frac{|b_\alpha(p;x)|^2}{p^{1+2\alpha}}
 &\ll\frac{(J-j)^2}{j+1},
 \label{eq:above-half-coefficient-masses}
\end{align}
for $1 \leq j <J$. 

Now note that 
\[
 \P(W_\alpha>V)
 \leq\P\Big(|Q_0|>\frac{w}{4}\Big)
 +\sum_{j=1}^{J-1}
       \P\Big(|Q_j|>\frac{w}{4(J-j)^2}\Big).
\]
For $0 \leq j <J$, take 
\begin{equation}\label{eq:above-half-moment-orders}
 q_j=\Big\lfloor\frac{V}{j+1}\Big\rfloor.
\end{equation}
The polynomial $Q_j^{q_j}$ has length at most
$T^{q_j(j+1)/V}\leq T$. Using the mean value theorem and \eqref{eq:above-half-coefficient-masses}, followed by Markov's inequality
and Stirling's formula, we obtain that
\begin{align}
 \P\Big(|Q_0|>\frac{w}{4}\Big) \label{q0}
 &\ll\Big(\frac{CV\log\log T}{w^2}\Big)^{\lfloor V\rfloor},\\
 \P\Big(|Q_j|>\frac{w}{4(J-j)^2}\Big)
 &\ll\Big(\frac{CV(J-j)^6}{(j+1)^2w^2}\Big)^{q_j}
 \leq\Big(\frac{CV(J-j)^8}{J^2w^2}\Big)^{q_j}.
 \label{eq:above-half-log-rates}
\end{align}

Substituting $w=(1-c) V \log J/(4J)$ in the first bound gives
\[
 \P\Big(|Q_0|>\frac{w}{4}\Big)
 \ll
 \Big(\frac{CJ^2\log\log T}{V(\log J)^2}\Big)^{\lfloor V\rfloor}
 \ll e^{-2V}.
\]

For $j\geq1$, the second bound in
\eqref{eq:above-half-log-rates} becomes
\[
 \P\Big(|Q_j|>\frac{w}{4(J-j)^2}\Big)
 \ll\Big(\frac{C(J-j)^8}{V(\log J)^2}\Big)^{q_j}.
\]
We have that
\[
 \log\frac{V(\log J)^2}{C(J-j)^8}
 \geq j+2\log\log J-\log C-9\geq j+5
\]
for sufficiently large $T$. Hence,
\begin{equation*}\label{eq:high-tail-single-block}
 \P\Big(|Q_j|>\frac{w}{4(J-j)^2}\Big)
 \ll e^{-q_j(j+5)},
\end{equation*}
and since
\[
 q_j(j+5)
 \geq V+\frac{3V}{J},
\]
it follows that 
\[
 \sum_{j=1}^{J-1}\P\Big(|Q_j|>\frac{w}{4(J-j)^2}\Big)
 \ll J e^{-V-3V/J}\leq e^{-V-2V/J}.
\]
Together with the bound for $Q_0$, this
proves the second bound in
\eqref{eq:negative-moment-high-tail}.

For $C_1\log\log T\leq V<V_0$, the argument is similar, so we will only sketch it. Fix $D$ with $1<D<1/\rho$, and take
\begin{equation*}\label{eq:negative-moment-moderate-cutoff}
 J=\lfloor D\log V\rfloor,\qquad
 \beta=\frac JV,\qquad x=T^{J/V}.
\end{equation*}
Using \eqref{bigrange}, we have that 
\[
 \frac{V}{J}\log\frac{1}{1-T^{-\alpha J/V}}
 \leq\frac VJ \big( \log V - \log J +c \log \log \log \log T +O(1) \big)
 =\frac{V}{D}+o(V).
\]
The error in \eqref{ineq_bf} is
$o(V)$, so it follows that
\[
 \{W_\alpha>V\}\subseteq\{|Q_x|>w\},
 \]
 for $w=\frac{D-1}{2D}V$.
The bound \eqref{bp_again},
the division \eqref{eq:high-tail-prime-blocks}, and 
the estimate \eqref{eq:above-half-log-rates} hold as before.
Now we have
\[
 \log\frac{J^2w^2}{CV(J-j)^8}
 \geq\frac{j+1}{D}+2\log J-C 
 \geq\frac{j+1}{D}+4, 
\]
so 
\[
 \sum_{j=1}^{J-1}\P\Big(|Q_j|>\frac{w}{4(J-j)^2}\Big)
 \ll  J\exp\bigg(-\frac{V-J}{D}
                         -4\Big(\frac VJ-1\Big)\bigg)
 \ll  e^{-V/D}.
\]
Including the contribution from $Q_0$ gives
\[
 \P(W_\alpha>V)
 \ll 
 \Big(\frac{C \log\log T}{V}\Big)^{\lfloor V\rfloor}
 +e^{-V/D}.
\]
Since $1/D>\rho$, choosing $C_1$ sufficiently large makes both
terms $O(e^{-\rho V})$. This proves the first
bound in \eqref{eq:negative-moment-high-tail}.
\end{proof}

Using similar arguments, we can also prove the following variant. We will skip the proof since it is very similar to the proof of Lemma \ref{lem:negative-moment-high-tail}.
\begin{lem}
\label{newlem}
   Assume RH. Fix $0<c<1$ and $0 < \varepsilon<1$. There exists $C>0$ such that uniformly for 
    $$ \frac{C}{  \log \log T} \leq \alpha \log T \ll \frac{1}{(\log \log \log T)^c}$$
    and $(\log T)^{\varepsilon} \leq V \leq V_{*}$ (with $V_{*}$ given in \eqref{eq:above-half-pointwise-cap}), we have
    $$ \P (W_{\alpha} >V) \ll \exp \Big( -V - \frac{2V}{\log V}\Big).$$
\end{lem}

Now suppose that $k>1/2$, and suppose that 
\begin{equation}
\label{largealpha}
\frac{C_k}{(\log T)^{1/(2k)} \log \log T} \leq \alpha \ll \frac{ (\log \log T)^{4/k+\varepsilon}}{(\log T)^{1/(2k)}}.
\end{equation}
We will prove the following lemma which holds in this regime.

\begin{lem}\label{lem:above-half-tail}
Assume RH. Fix $k>1/2$.
There exist constants $b,B,c_k,C_k>0$
such that uniformly in the range \eqref{largealpha}, we have
\begin{equation}\label{eq:above-half-tail}
 \P(W_\alpha>V)\ll 
 \begin{cases}
 e^{-(2k+2)V},
   & \mbox b \log(1/\alpha)\leq V\leq B \alpha \log T,\\
 \displaystyle\exp\Big(-\Big(2k+\frac{c_k}{\log\log T}\Big)V\Big),
   & \mbox V>B \alpha \log T.
 \end{cases}
\end{equation}
\end{lem}
\begin{proof}
    The proof follows along similar lines as Lemma \ref{lem:negative-moment-high-tail}, so we will only sketch the details. 
We first prove the bound in the second case.  The pointwise bound
\eqref{eq:above-half-pointwise-cap} still holds, so we restrict to $V\leq V_*$.
Pick
\begin{equation*}\label{eq:above-half-cutoff}
 J=\Big\lceil\log\Big( \frac{V}{\alpha \log T}\Big)-\log\log(\frac{V}{\alpha \log T})+2\Big\rceil,
 \qquad \beta=\frac JV,\qquad x=T^{J/V}.
\end{equation*}
As in Lemma \ref{lem:negative-moment-high-tail}, we have
\begin{align*}
 V-\frac VJ\log\frac1{1-T^{-J \alpha/V}}
 \geq\frac{3V}{4J}.
 \label{eq:above-half-remainder-gap}
\end{align*}
Then
\begin{equation*}\label{eq:above-half-forced-primes}
 \{W_\alpha>V\}\subseteq\{|Q_x|>w\},
\end{equation*}
for 
$$ w=\frac{V}{2J}.$$
The coefficient bound \eqref{bp_again},
the partition \eqref{eq:high-tail-prime-blocks}, and the $q_j$ in equation
 \eqref{eq:above-half-moment-orders} are unchanged.
In particular, \eqref{eq:above-half-log-rates} applies with the choice of
$w$. 
As $q_j(j+1)\geq V-J$ and $q_j\geq V/J-1$, it follows that
\[
 q_j\log\frac{V}{C(J-j)^8}
 \geq(V-J)+\Big(\frac VJ-1\Big)(\log V-J-C)
 =\Big(\frac VJ-1\Big)(\log V-C).
\]
Thus the sum over $j\geq1$ is at most
\[
 J\exp\bigg(-\Big(\frac VJ-1\Big)(\log V-C)\bigg)
 \ll\exp\Big(-\frac VJ(\log V-C)\Big).
\]

For the contribution from the small primes, \eqref{q0} still holds, and since $V>B \alpha \log T$ we have
\[
 \Big(\frac{CJ^2\log\log T}{V}\Big)^{\lfloor V\rfloor}
 \leq\exp\Big(-\frac V3\log V\Big).
\]
 Hence
\begin{equation}\label{eq:above-half-common-tail}
 \P(W_\alpha>V)
 \ll_k\exp\Big(-\frac VJ(\log V-C)\Big).
\end{equation}

We next show that
\begin{equation}\label{eq:above-half-rate-gap}
 \log V-2kJ\geq C+1.
\end{equation}
We have 
\[
 J\leq\log\frac{V}{\alpha \log T}\qquad\text{and} 
\qquad e^J\leq\frac{e^3V}{\alpha \log T\log(V/\alpha \log T)},\]
and hence
\begin{equation} \label{ubalpha}
( \alpha \log T) J e^J\ll V.
\end{equation}
Moreover, \eqref{largealpha} and
\eqref{eq:above-half-pointwise-cap} give
\[
 \log V\asymp \log\log T
\]
uniformly for
$B \alpha \log T<V\leq V_*$. If \eqref{eq:above-half-rate-gap} failed, we would have
\[
 J>\frac{\log V-C-1}{2k}\gg \log\log T.
\]
Substituting this lower bound into \eqref{ubalpha}
would give
\[
 \alpha \log T \log\log T\ll V^{1-1/(2k)}
 \ll (\log T)^{1-1/(2k)},
\]
where the last inequality uses $k>1/2$. Taking $C_k$ sufficiently large in \eqref{largealpha} gives a contradiction and hence 
 \eqref{eq:above-half-rate-gap} follows.

Finally, 
\[
 \frac{\log V-C}{J}\geq2k+\frac1J
 \geq2k+\frac{c_k}{\log\log T}.
\]
Together with \eqref{eq:above-half-common-tail}, this proves
the second bound in \eqref{eq:above-half-tail}.

For $b \log\frac{1}{\alpha}\leq V\leq B \alpha \log T$, we apply the same arguments directly to $Q_x$.
Take $\beta=A_0/V$ and $x=T^\beta$, with
$A_0$ sufficiently large. As in \eqref{ineq_bf}, since
$ \alpha \log T \beta\geq A_0/B$, the second term in \eqref{ineq_bf}
 is at most $V/8$, and its error
is $O(\log\log\log T)=o(V)$. Thus
\[
 \P(W_\alpha>V)\leq\P\Big(|Q_x|>\frac{V}{2}\Big).
\]
Moreover, Lemma \ref{asympta} gives
$|b_\alpha(p;x)|\ll1$, so
\[
 \sum_{p\leq x}\frac{|b_\alpha(p;x)|^2}{p^{1+2\alpha}}
 \ll\log\frac1\alpha.
\]
Take $q=\lfloor V/A_0\rfloor$. Since $Q_x^q$ has length at
most $x^q\leq T$, we get
\[
 \P(W_\alpha>V)
 \ll\Big(\frac{Cq\log(1/\alpha)}{V^2}\Big)^q
 \leq\Big(\frac{C\log(1/\alpha)}{A_0V}\Big)^q.
\]
Choose $b$ so that $C/(A_0b)\leq e^{-4A_0(k+1)}$.
For $V\geq b \log\frac{1}{\alpha}$ and sufficiently large $T$,
we have $q\geq V/(2A_0)$, and the last bound is
$O(e^{-(2k+2)V})$. This proves the first bound in
\eqref{eq:above-half-tail}.
\end{proof}

\begin{proof}[Proof of Theorem~\ref{negativemom}]
We have
\begin{align}
\label{negative_momallk}
 I_{-k}(\alpha,T)
 &=\frac1T\int_T^{2T}
       \Big|\zeta(\tfrac12+\alpha+it)\Big|^{-2k}\,dt=2k\int_{-\infty}^{\infty}
                   e^{2kV}\P(W_\alpha>V)\,dV.
\end{align}
First assume that we are in the range \eqref{firstrange}.
 Fix
\[
 \sqrt{k(1-k)}<b<\frac12,
 \qquad 2k<\rho<1,
\]
and let $C_1>1$ be given by Lemma~\ref{lem:negative-moment-high-tail}.
We split the integral as follows:
\begin{align*}
 I_{-k}(\alpha,T)
 &=2k\Bigg(
       \int_{-\infty}^{(k/4)\log\frac{1}{\alpha}}
      +\int_{(k/4)\log\frac{1}{\alpha}}^{b\log\frac{1}{\alpha}}
      +\int_{b\log\frac{1}{\alpha}}^{\frac12\log\frac{1}{\alpha}}+\int_{\frac12\log\frac{1}{\alpha}}^{B\log\log T}
      +\int_{B\log\log T}^{\infty}
       \Bigg)e^{2kV}\P(W_\alpha>V)\,dV\\
       &=\sum_{j=1}^{5}I_j,
 \label{eq:negative-moment-integral-split}
\end{align*}
say. The endpoints are in the displayed order for sufficiently large $T$,
since \eqref{firstrange} gives
\[
 \log\frac1\alpha
 =\log\log T+O (\log\log\log T).
\]

For the first integral, we use $\P(W_\alpha>V)\leq1$ to obtain
\[ I_1 = 
 2k\int_{-\infty}^{(k/4)\log\frac{1}{\alpha}}
                  e^{2kV}\P(W_\alpha>V)\,dV
 \leq\alpha^{-k^2/2}.
\]
For the second integral, Theorem~\ref{mainthm}  gives
\begin{align*}
 I_2 \ll \frac{\alpha^{-k^2}}{\sqrt{\log(1/\alpha)}}
      \int_{\mathbb R}
      \exp\Big(-\frac{(V-k\log(1/\alpha))^2}
                       {\log(1/\alpha)}\Big)\,dV
 \ll \alpha^{-k^2}.
 \label{eq:negative-moment-gaussian-integral}
\end{align*}
For the third integral, monotonicity and Theorem~\ref{mainthm} give
\[
 \P(W_\alpha>V)
 \leq\P\Big(W_\alpha>b\log\frac{1}{\alpha}\Big)
 \ll \frac{\alpha^{b^2}}{\sqrt{\log(1/\alpha)}}
\]
for $V\geq b\log\frac{1}{\alpha}$. Hence
\[
 I_3
 \ll \frac{\alpha^{-(k-b^2)}}{\sqrt{\log(1/\alpha)}}
 \ll \alpha^{-k^2},
\]
since $k-b^2<k^2$.

For the fourth integral, we use Theorem~\ref{bigk} with
$\varepsilon_1=(1/2-k)^2/2$ and $\kappa_1=2B$.
Since $2k<1$, we obtain that
\begin{align*}
 I_4 \ll \alpha^{-1/4-\varepsilon_1}
       \int_{\frac12\log\frac{1}{\alpha}}^{B\log\log T}
                         e^{-(1-2k)V}\,dV \ll \alpha^{-(k-1/4+\varepsilon_1)}
 =\alpha^{-k^2+(1/2-k)^2/2}\ll \alpha^{-k^2}.
\end{align*}
For the fifth integral, Lemma \ref{lem:negative-moment-high-tail}
and $\rho>2k$ (and noticing that in the second range in the Lemma the first bound dominates the second) give
\begin{align}
\label{i5}
I_5
 &\ll \int_{B\log\log T}^{\infty}e^{-(\rho-2k)V}\,dV
 \ll  e^{-(\rho-2k)B\log\log T} = o(1).
\end{align}
 Combining all the five bounds finishes the proof for the first range.

 In the second range \eqref{secondrange}, we instead split the integral into 
 \begin{align*}
     I_{-k}(\alpha, T) =2k\Bigg(
     \int_{-\infty}^{B\log\log T}
      +\int_{B\log\log T}^{\infty}
       \Bigg)e^{2kV}\P(W_\alpha>V)\,dV.
 \end{align*}
The first integral is bounded by $e^{2k B\log \log T} = (\log T)^{O(1)},$ and for the second the same bound \eqref{i5} holds. 
This finishes the proof in this case.
\end{proof}

We now prove Theorem \ref{negative1/2}.
\begin{proof}[Proof of Theorem \ref{negative1/2}]
  Let $V_0$ be as in Lemma~\ref{lem:negative-moment-high-tail},
and let $V_*$ be the
pointwise upper bound in \eqref{eq:above-half-pointwise-cap}. Let $\alpha$ be as in \eqref{firstrange'}.
We use Lemma \ref{lem:negative-moment-high-tail} to obtain that 
\begin{align*}
 I_{-1/2}(\alpha,T)
 &=\int_{-\infty}^{V_0}e^V\P(W_\alpha>V)\,dV
   +\int_{V_0}^{V_*}e^V\P(W_\alpha>V)\,dV\\
 &\leq e^{V_0}
   +O \bigg(
       (\log T)\exp\Big(-a\frac{V_0}{\log V_0}\Big)\bigg)
 =e^{V_0}+o(1).
\end{align*}

For $\alpha$ in the second range \eqref{secondrange'}, we write
\begin{align*}
     I_{-1/2}(\alpha, T) =\Bigg(
     \int_{-\infty}^{(\log T)^{\varepsilon}}
      +\int_{(\log T)^{\varepsilon}}^{V_{*}}
       \Bigg)e^{V}\P(W_\alpha>V)\,dV,
 \end{align*}
 where $V_{*}$ is as in \eqref{eq:above-half-pointwise-cap}. The first integral is bounded by $\exp ((\log T)^{\varepsilon})$, and for the second integral we use Lemma \ref{newlem}, which shows that the second integral is $o(1)$. This finishes the proof. 
  \end{proof}

  Finally, we prove Theorem \ref{k>1/2}.
  \begin{proof}[Proof of Theorem \ref{k>1/2}]
  We note that this theorem is already proven in \cite{BF} in the range
  $$ \alpha \gg \frac{ (\log \log T)^{4/k+\varepsilon}}{(\log T)^{1/(2k)}}. $$
  Hence, we assume that
  \begin{equation}
      \frac{C_k}{(\log T)^{1/(2k)} \log \log T } \leq \alpha \ll \frac{ (\log \log T)^{4/k+\varepsilon}}{(\log T)^{1/(2k)}}.
      \label{newrange}
  \end{equation}
Equation \eqref{negative_momallk} still holds, and we write 
\begin{align*}
 I_{-k}(\alpha,T)
 &=2k\Bigg(
       \int_{-\infty}^{(k/4)\log\frac{1}{\alpha}}
      +\int_{(k/4)\log\frac{1}{\alpha}}^{b\log\frac{1}{\alpha}}
+\int_{b\log\frac{1}{\alpha}}^{B \alpha \log T}
    +\int_{B \alpha \log T}^{V_*}
       \Bigg)e^{2kV}\P(W_\alpha>V)\,dV\\
       &=I_1+I_2+I_3+I_4,
 \label{eq:negative-moment-integral-split}
\end{align*}
say. 

For $I_1$,  we have
$$ I_1 \ll \int_{-\infty}^{(k/4) \log(1/\alpha)} e^{2kV} \, dV \ll \alpha^{-k^2/2}.$$
We use Theorem \ref{mainthm} for the second integral. Since from \eqref{newrange} it follows that
$$ \alpha \log T \gg \Big( \log \frac{1}{\alpha}\Big)^{12+\varepsilon},$$
we then have that $$I_2 \ll \frac{1}{\sqrt{\log(1/\alpha)}} \int_{(k/4)\log\frac{1}{\alpha}}^{b\log\frac{1}{\alpha}} \exp \Big( 2kV - \frac{V^2}{\log(1/\alpha)} \Big) \, dV \ll \alpha^{-k^2}.$$
Finally, for the third and fourth integrals, we use Lemma \ref{lem:above-half-tail} which gives that $I_3=o(1)$ and $I_4=o(1)$. Gathering the four bounds finishes the proof.
  \end{proof}

\section{Partial sums of the generalized  M\"obius function}\label{section_mobius}

Our arguments work for general $k>0$. Let $\mu_k(n)$ denote the $n^{\text{th}}$ Dirichlet coefficient of $\zeta(s)^{-k}$ and let
\begin{align*}
Z_{-k}(s)&=\frac{1}{s((s-1)\zeta(s))^k}.
\end{align*}
The function $Z_{-k}(s)$ is holomorphic in the disc $|s-1|<1$ and has a Taylor series,
\[
Z_{-k}(s)=\sum_{j=0}^{\infty}\gamma_{-k,j}(s-1)^j.
\]
The following more general theorem holds.

\begin{thm}\label{partM'}
Assume RH. For every $k>0$, there exists $C=C(k)>0$ such that
\begin{align*}
&\sum_{n \leq x} \mu_k(n)\\
&\quad  =-\frac{\sin(k\pi)}{\pi}\frac{x}{(\log x)^{k+1}}\Bigg(\sum_{j=0}^{N}\frac{(-1)^j\Gamma(k+j+1)\gamma_{-k,j}}{(\log x)^{j}}+O\bigg((N+1)^{k+1/2}\Big(\frac{N+1}{\log x}\Big)^{N+1}\bigg)\Bigg)\\
&\qquad\qquad+O\bigg(\sqrt{x} \exp\Big( C\Big(\frac{\log x}{\log\log x}\Big)^{k/(k+1)}  \Big)\bigg),
\end{align*}
for any $N\geq0$.
\end{thm}

We remark that the main term above vanishes if $k\in\mathbb{N}$, and in that case the theorem reduces to Theorem \ref{partM}.

We shall prove the case $k\in\mathbb{N}$. For $k\notin\mathbb{N}$, we need to include a truncated Hankel-type contour in the below contour deformation as in \cite[Section 9.1]{BF}. As in \cite[Section 9.1]{BF}, the main term in the theorem comes from the Hankel-type contour and only appears when $k\notin\mathbb{N}$.

\subsection{The case $k=1$}\label{sectionk1}

From \cite[(9)]{Sound-mobius} we have
\begin{equation*}\label{mobiusstart}
	\sum_{n \leq x} \mu(n)=\frac{1}{2\pi i}\int_{c-i\lfloor x\rfloor}^{c+i\lfloor x\rfloor}\frac{x^s}{\zeta(s)}\frac{ds}{s}+O(\log x)
\end{equation*}
with $c=1+\frac{1}{\log x}$.  Let
\[
F_0(x)=\log x\log\log x
\]
and let
\[
x_0=\exp \bigg(\sqrt{\frac{\log x}{\log\log x}}\bigg),\qquad \alpha_0=\frac{C}{F_0(x_0)},\qquad J=\bigg\lceil \log_2\frac{\lfloor x\rfloor}{x_0} \bigg\rceil\ll\log x,
\]
where $C=C_1(1/3)$ as in Theorem \ref{negative1/2}.

We now deform the contour as follows. We replace the line segment $c+it$, $|t|\leq \lfloor x\rfloor$, with $$\bigcup_{j=1}^{J} \big(V_j\cup H_j\big) \bigcup V_0,$$ 
where
\begin{align*}
&V_0:\qquad s=\frac{1}{2}+\alpha_0+it,\qquad |t|\leq x_0,\\
	&V_j:\qquad s=\frac{1}{2}+\frac{C}{F_0(2^{j-1}x_0)}+it,\qquad 2^{j-1}x_0\leq |t|\leq \min\{2^jx_0,\lfloor x\rfloor\},\qquad 1\leq j\leq J,\\
	&H_j:\qquad s=\sigma\pm i2^jx_0,\qquad \frac12+\frac{C}{F_0(2^{j}x_0)}\leq \sigma\leq \frac{1}{2}+\frac{C}{F_0(2^{j-1}x_0)},\qquad 1\leq j\leq J-1,\\
	& H_J:\qquad s=\sigma\pm i\lfloor x\rfloor,\qquad \frac12+\frac{C}{F_0 (2^{J-1}x_0)}\leq \sigma\leq 1+\frac{1}{\log x}.
\end{align*}
We encounter no singularity in moving the contour.

From \cite[Theorem 2]{CC} we have
\begin{equation}\label{pwb}
	|\zeta(s)|^{-1} \ll (|t|+2)^{(1/2+\varepsilon)\frac{\log\frac{1}{\alpha_0}}{\log\log(|t|+10)}}
\end{equation}
for $s\in V_0$. 
So the integral along $V_0$ is
\begin{align}\label{V0'}
	&\ll \sqrt{x}\exp\bigg(\alpha_0\log x+\Big(\frac 12+\varepsilon\Big)\frac{\log x_0}{\log\log x_0}\log\frac{1}{\alpha_0}\bigg)\ll \sqrt{x}\exp \bigg(C_0\sqrt{\frac{\log x}{\log\log x}}\bigg),
\end{align}
for some $C_0>0$

For the contribution from the vertical segments $\cup_{j=1}^{J} V_j$, we deduce from Theorem \ref{negative1/2}\footnote{If $k<1$, then we instead apply (1.8) in Theorem 1.3 of \cite{BF}.} that it is bounded by
\begin{align*}\label{estimateVj}
	&\ll \sqrt{x}\sum_{j=1}^{J}\exp\Big(\frac{\log x}{F_0(2^{j-1}x_0)}+\big(\log (2^jx_0)\big)^{1/3}\Big)\ll \sqrt{x}J\exp\Big(\alpha_0\log x+(\log x)^{1/3}\Big)\nonumber\\
    &\ll\sqrt{x}\exp \bigg(C_0\sqrt{\frac{\log x}{\log\log x}}\bigg).
\end{align*}

For $s\in H_j$, $1\leq j\leq J-1$, similar to \eqref{pwb} we have
\[
|\zeta(s)|^{-1} \ll (2^jx_0)^{1/2+\varepsilon}.
\]
So the contribution from the horizontal segments $\cup_{j=1}^{J-1} H_j$ is
\begin{align*}\label{estimateHj}
	&\ll \sqrt{x}\sum_{j=1}^{J-1}(2^jx_0)^{-1/2+\varepsilon}\exp\Big(\frac{\log x}{F_0 (2^{j-1}x_0)}\Big)\ll \sqrt{x}Jx_0^{-1/2+\varepsilon}\exp\big(\alpha_0\log x\big)\nonumber\\
    &\ll\sqrt{x}\exp \bigg(C_0\sqrt{\frac{\log x}{\log\log x}}\bigg).
\end{align*}

We are left with the integral along $H_J$, which is bounded by
\begin{equation*}
\max_{\frac12+\frac{C}{F_0(x)}\leq\sigma\leq1+\frac{1}{\log x}}\exp\Big((\sigma-1)\log x-\log|\zeta(\sigma\pm i\lfloor x\rfloor)|\Big).
\end{equation*}

For $\frac12+\frac{C}{F_0(x)}\leq\sigma\leq \frac12+o(\frac{1}{\log\log x})$, Lemma 2.2 of \cite{BF} implies that this is
\begin{align*}
	&\ll\exp\bigg((\sigma-1)\log x+\frac{\log x}{2\log\log x}\log\frac{1}{1-(\log x)^{1-2\sigma}}\nonumber\\
&\qquad\qquad\qquad\qquad\qquad+O\Big(\frac{\log x}{\log\log x}\Big)+O\Big(\frac{\log x}{(\log\log x)^2}\log\frac{1}{1-(\log x)^{1-2\sigma}}\Big)\bigg)\nonumber\\
	&\ll  \exp\bigg((\sigma-1)\log x+\frac{\log x}{2\log\log x}\log\frac{1}{(2\sigma-1)\log\log x}\nonumber\\
&\qquad\qquad\qquad\qquad\qquad+O\Big(\frac{\log x}{\log\log x}\Big)+O\Big(\frac{\log x}{(\log\log x)^2}\log\frac{1}{2\sigma-1}\Big)\bigg),
\end{align*}
which is decreasing with respect to $\sigma$, and, hence, \begin{align*}
&\ll  \exp\bigg(-\frac{\log x}{2}+\frac{C\log x}{F_0(x)}+\frac{\log x}{2\log\log x}\log\frac{F_0(x)}{2\log\log x}+O\Big(\frac{\log x}{\log\log x}\Big)\bigg)\nonumber\\
&\ll\exp\bigg(O\Big(\frac{\log x}{\log\log x}\Big)\bigg)\ll x^\varepsilon.
\end{align*}

For $1-\frac{1}{\log\log x}\leq \sigma\leq 1+\frac{1}{\log x}$, the second estimate in  Theorem 2 of \cite{CC} leads to a bound of size
\begin{align*}
	&\ll\exp\Big((\sigma-1)\log x+\log\log\log x+O(1)\Big)\ll \log\log x.
\end{align*}

Finally, for $\frac12+O(\frac{1}{\log\log x})\leq \sigma\leq 1-\frac{1}{\log\log x}$, we use the last estimate in  Theorem 2 of \cite{CC} to get the bound
\begin{align}\label{est4}
	&\ll\exp\bigg((\sigma-1)\log x+\varepsilon\log x+\frac{(\log x)^{2-2\sigma}}{(1-\sigma)\log\log x}+O\Big(\frac{(\log x)^{2-2\sigma}}{(1-\sigma)^2(\log\log x)^2}\Big)\bigg)\ll x^\varepsilon.
\end{align}
Combining the estimates \eqref{V0'}--\eqref{est4} we obtain Theorem \ref{partM} in the case $k=1$.

\subsection{The case $k\in\mathbb{N}$ and $k>1$}

The proof follows closely that of the case $k=1$. We, therefore, provide only a sketch and highlight the key differences. 

From \cite[(9.1)]{BF} we have
\begin{equation*}\label{mobiusstart2}
	\sum_{n \leq x} \mu_k(n)=\frac{1}{2\pi i}\int_{c-i\lfloor x\rfloor}^{c+i\lfloor x\rfloor}\frac{x^s}{\zeta(s)^k}\frac{ds}{s}+O\big((\log x)^k\big)
\end{equation*}
with $c=1+\frac{1}{\log x}$.  Let
\[
F_1(x)=(\log x)^{1/k}\log\log x
\]
and let
\[
x_1=\exp\bigg(\Big(\frac{\log x}{\log\log x}\Big)^{k/(k+1)}  \bigg),\qquad \alpha_1=\frac{C}{F_1(x_1)},\qquad J=\bigg\lceil \log_2\frac{\lfloor x\rfloor}{x_1} \bigg\rceil\ll\log x,
\]
with $C=C(k/2)$ as in Theorem \ref{k>1/2}.
We now replace the line segment $c+it$, $|t|\leq \lfloor x\rfloor$, with $$\bigcup_{j=1}^{J} \big(V_j'\cup H_j'\big) \bigcup V_0',$$ 
where
\begin{align*}
&V_0':\qquad s=\frac{1}{2}+\alpha_1+it,\qquad |t|\leq x_1,\\
	&V_j':\qquad s=\frac{1}{2}+\frac{C}{F_1(2^{j-1}x_1)}+it,\qquad 2^{j-1}x_1\leq |t|\leq \min\{2^jx_1,\lfloor x\rfloor\},\qquad 1\leq j\leq J,\\
	&H_j':\qquad s=\sigma\pm i2^jx_1,\qquad \frac12+\frac{C}{F_1(2^{j}x_1)}\leq \sigma\leq \frac{1}{2}+\frac{C}{F_1(2^{j-1}x_1)},\qquad 1\leq j\leq J-1,\\
	& H_J':\qquad s=\sigma\pm i\lfloor x\rfloor,\qquad \frac12+\frac{C}{F_1(2^{J-1}x_1)}\leq \sigma\leq 1+\frac{1}{\log x}.
\end{align*}
We cross no singularity in moving the contour.

For the integral along $V_0'$, we use \eqref{pwb} and get that
\begin{align}\label{1stestimate'2}
	\frac{1}{2\pi i}\int_{1/2+\alpha_1-ix_1}^{1/2+\alpha_1+ix_1}\frac{x^s}{\zeta(s)^k}\frac{ds}{s}&\ll \sqrt{x}\exp\bigg(\alpha_1\log x+\Big(\frac k2+\varepsilon\Big)\frac{\log x_1}{\log\log x_1}\log\frac{1}{\alpha_1}\bigg)\nonumber\\
    &\ll \sqrt{x}\exp\bigg( C_0\Big(\frac{\log x}{\log\log x}\Big)^{k/(k+1)}  \bigg),
\end{align}
for some $C_0>0$

For the contribution from the vertical segments $\cup_{j=1}^{J} V_j'$, we use Theorem \ref{k>1/2} and bound it by
\begin{align}\label{Vj'2}
	&\ll \sqrt{x}\sum_{j=1}^{J}\exp\Big(\frac{C\log x}{F_1(2^{j-1}x_1)}+O\big(\log\log (2^jx_1)\big)\Big)\ll \sqrt{x}J\exp\Big(\alpha_1\log x+O(\log\log x)\Big)\nonumber\\
    &\ll\sqrt{x}\exp\bigg( C_0\Big(\frac{\log x}{\log\log x}\Big)^{k/(k+1)}  \bigg).
\end{align}

The same bound holds for $\cup_{j=1}^{J-1} H_j'$, and the contribution from $H_J'$ is $O_\varepsilon(x^\varepsilon)$ as in Section \ref{sectionk1}. Together with \eqref{1stestimate'2} and \eqref{Vj'2} we obtain the bound for the case $k\in\mathbb{N}$ and $k>1$.

\section{The minimum of $\zeta(s)$ in short intervals}
\label{section_min}
Recall that
\[
M_{\alpha,\gamma}(t)
:=
\min_{|h|\leq(\log T)^\gamma}
\Big|\zeta\Big(\tfrac12+\alpha+i(t+h)\Big)\Big|.
\]
The log-correlated prime-sum model for $\log|\zeta|^{-1}$, and its
analogy with branching random walk, suggest the following
analogue of the FHK Conjecture; see the discussion in the
introduction and also \cite{AOR}.

For $\gamma \geq 0$, let
\begin{equation} 
\label{kgamma}
k_\gamma:=\sqrt{1+\frac{\gamma\log\log T}{\log(1/\alpha)}}.
\end{equation}

\begin{conj} 
\label{min_conj}
 Suppose that
$$
\alpha=o(1),\qquad \alpha\log T\gg1.
$$ Fix $\gamma\geq0$. When $\gamma>0$, assume additionally that $\alpha\leq(\log T)^{-c}$ for some fixed $0<c<1$. For every $\varepsilon>0$, there exists $C>0$ such that for all sufficiently large $T$, the following holds.
\begin{enumerate}
\item If $\gamma=0$, then
\[
\P\bigg(
\Big|
\log M_{\alpha,0}(t)+\log\frac1\alpha
-\frac34\log\log\frac1\alpha
\Big|\leq C
\bigg)\geq1-\varepsilon.
\]
\item If $\gamma>0$,
then
\[
\P\bigg(
\Big|
\log M_{\alpha,\gamma}(t)+k_\gamma\log\frac1\alpha
-\frac{1}{4k_\gamma}\log\log\frac1\alpha
\Big|\leq C
\bigg)\geq1-\varepsilon.
\]
\end{enumerate}
\end{conj}

\subsection{Lower bounds for the minimum}

We will prove the following analogue to Corollary $3.1$ in \cite{ABR}.

\begin{thm}\label{prop:main}
Assume RH. Fix $0<c<1$ and $\gamma\geq0$. Let $k_\gamma$ be as in
\eqref{kgamma}. Suppose that
\begin{equation}\label{alpha_range}
 0<\alpha\ll (\log T)^{-c},\qquad
 \frac{\alpha\log T}{\log\log T\,\log\log\log T}
 \to \infty.
\end{equation}
 There exists $C>0$
such that for all sufficiently large $T$ and every
\[
 0\leq y\ll \log\frac{1}{\alpha},
\]
we have
\begin{equation}\label{eq:V-form}
\begin{aligned}
 \mathbb P \bigg(
M_{\alpha,\gamma}(t)
 \leq e^{-y}\alpha^{k_\gamma} \Big(  \log \frac{1}{\alpha} \Big)^{-C} 
     \bigg)\leq
 C\exp\Big(-2k_\gamma y-\frac{y^2}{\log(1/\alpha)}\Big).
\end{aligned}
\end{equation}
\end{thm}

\begin{remark}
\emph{Note that the lower threshold in the theorem above is
\begin{equation*}\label{eq:y-form}
 e^{-y}\alpha^{k_\gamma} \Big(  \log \frac{1}{\alpha} \Big)^{-C}
 =e^{-y}\alpha^{\,k_\gamma+o(1)}.
\end{equation*}
Also note that when $\gamma=0$,
\begin{equation*}\label{eq:gamma-zero}
 M_{\alpha,0}(t)
 \geq e^{-y}\alpha^{1+o(1)}
\end{equation*}
outside a set of measure
\[
 \ll T\exp\Big(-2y-\frac{y^2}{\log(1/\alpha)}\Big).
\]}
\end{remark}

Before beginning the proof of Theorem \ref{prop:main}, we first need the following lemma.
\begin{lem}\label{lem:maximal}
Let $1\leq H\leq T/10$ and
\[
 P(t)=\sum_{p\leq x}a_p p^{-i t}.
\]
For any $q\in\mathbb{N}$ such that $2q\leq \log T/\log x$, we have
\begin{equation}\label{eq:maximal-moment}
 \int_T^{2T}\max_{|h|\leq H}|P(t+h)|^{2q}\,dt
 \ll T(1+qH \log x)q!\bigg(\sum_{p \leq x} |a_p|^2\bigg)^q.
\end{equation}
The same estimate holds for a polynomial of the form
$\sum_{p^2\leq x}a_p p^{-2i t}$.
\end{lem}
\begin{proof}
We shall prove \eqref{eq:maximal-moment}. The treatment of the other case is similar.

We use the Sobolev inequality \cite[(6)]{ABBRS} to get that for fixed $t$, we have 
\[
 \max_{|h|\leq H}|P(t+h)|^{2q}
 \leq\frac{|P(t-H)|^{2q}+|P(t+H)|^{2q}}2
 +q\int_{-H}^{H}|P(t+v)|^{2q-1}|P'(t+v)|\,dv.
\]
Integrating over $T\leq t\leq2T$, it follows that
\begin{align}
 \int_T^{2T} &\max_{|h|\leq H}|P(t+h)|^{2q}\,dt
 \leq\int_{T-H}^{2T+H}|P(v)|^{2q}\,dv
 +2qH\int_{T-H}^{2T+H}|P(v)|^{2q-1}|P'(v)|\,dv
 \notag\\
 &\leq\int_{T-H}^{2T+H}|P(v)|^{2q}\,dv
+2qH
 \Big(\int_{T-H}^{2T+H}|P(v)|^{2q}\,dv\Big)^{1-1/(2q)}
 \Big(\int_{T-H}^{2T+H}|P'(v)|^{2q}\,dv\Big)^{1/(2q)}.
 \label{sobolev}
\end{align}

Now we bound these two moments. By \cite[Lemma 3]{Sound-moments},
\[
 \int_{T-H}^{2T+H}|P(v)|^{2q}\,dv
 \ll Tq!\bigg(\sum_{p\leq x}|a_p|^2\bigg)^q.
\]
Since
\[
 P'(v)=-i\sum_{p\leq x}a_p(\log p)p^{-iv},
\]
the same estimate applied to $P'$ gives
\[
 \int_{T-H}^{2T+H}|P'(v)|^{2q}\,dv
 \ll Tq!\bigg(\sum_{p\leq x}|a_p|^2(\log p)^2\bigg)^q
 \leq Tq!(\log x)^{2q}\bigg(\sum_{p\leq x}|a_p|^2\bigg)^q.
\]
The two bounds above and \eqref{sobolev} yield \eqref{eq:maximal-moment}.
\end{proof}

\begin{proof}[Proof of Theorem \ref{prop:main}]

Set
\begin{equation}\label{eq:u-x}
 u=\log(\alpha\log T),
 \qquad
 x=\exp\Big(\frac{u}{\alpha}\Big).
\end{equation}
By \eqref{alpha_range} we have
\[
 u\to \infty,\qquad
 u\leq(1-c)\log\log T + O(1) ,\qquad
 c\log\log T +O(1) \leq\log\frac{1}{\alpha}\leq\log\log T.
\]
We also have that
\[
 \log x=\frac{u}{\alpha}=o(\log T),
\]
and hence $x=T^{o(1)}$.

Now Lemma \ref{bf_lemma} gives
\begin{equation}\label{eq:majorant}
 W_\alpha(t)
 \leq
 \Re\bigg(\sum_{n\leq x}
 \frac{b_\alpha(n;x)}{n^{1/2+\alpha+i t}}\bigg)
 +\frac{\log\frac{t}{2\pi}}{\log x}\log\frac1{1-x^{-\alpha}}+O(1).
\end{equation}
Since
$
 x^{-\alpha}=\frac1{\alpha\log T},
$
we have
\[
 \frac{\log\frac{t}{2\pi}}{\log x}\log\frac1{1-x^{-\alpha}}
 \ll \frac{\alpha\log T}{u} \cdot
\frac1{\alpha\log T}
 \ll \frac1u.
\]
We also separate the prime and prime-square contributions by setting
\[
 P_1(t)=\sum_{p\leq x}
 \frac{b_\alpha(p;x)}{p^{1/2+\alpha+i t}}
 \qquad\text{and}\qquad
 P_2(t)=\sum_{p^2\leq x}
 \frac{b_\alpha(p^2;x)}{p^{1+2\alpha+2i t}}.
\]
The terms supported on $p^r$, $r\geq3$, contribute $O(1)$ uniformly in $t$.
It follows that
\begin{equation*}\label{eq:basic-majorant}
 W_\alpha(t)\leq \Re P_1(t)+\Re P_2(t)+C_0
\end{equation*}
for some absolute constant $C_0>0$.

Since $\alpha\log x\to\infty$ and $x=T^{o(1)}$, the estimates
in the proof of Lemma~\ref{variances} give
\begin{align}
 B_1&:=\sum_{p\leq x}
 \frac{|b_\alpha(p;x)|^2}{p^{1+2\alpha}}
 =\log\frac{1}{\alpha}+O(1).\label{eq:B1}
\end{align}
We now take $H=(\log T)^\gamma$ and set
\begin{equation*}\label{eq:Delta}
 \Delta=C_1 \log \log \frac1{\alpha},
\end{equation*}
for $C_1$ sufficiently large. 
Let
\begin{equation*}\label{eq:U}
 U=k_\gamma \log\frac{1}{\alpha} +2\Delta+y.
\end{equation*}
By Lemma~\ref{lem:maximal} and Markov's inequality, for every $q\in\mathbb{N}$ such that
such that $2q\leq \log T/\log x$,
\begin{equation}\label{eq:P1-Markov}
\P \Big(
 \max_{|h|\leq H}|P_1(t+h)|>U \Big)
 \ll (1+qH \log x)\frac{q!B_1^q}{U^{2q}}.
\end{equation}
We choose 
\[
 q=\Big\lfloor\frac{U^2}{B_1}\Big\rfloor.
\]
By Stirling's formula,
\begin{equation}\label{eq:factorial}
 \frac{q!B_1^q}{U^{2q}}
 \ll \sqrt q\exp\Big(-\frac{U^2}{B_1}\Big).
\end{equation}
We also have that 
\begin{align}
 \log(1+qH \log x)
 &
 =k_\gamma^2 \log\frac{1}{\alpha}+O\Big(\log \log\frac{1}{\alpha}\Big),
 \label{eq:entropy}
\end{align}
where we used $u\ll \log\frac{1}{\alpha}$ and $q\asymp \log\frac{1}{\alpha}$. On the other hand,
using $B_1=\log\frac{1}{\alpha}+O(1)$,
\begin{align}
 \frac{U^2}{B_1}
 &\geq k_\gamma^2 \log\frac{1}{\alpha}+2k_\gamma y+\frac{y^2}{\log(1/\alpha)}
 +2\Delta.
 \label{eq:U-expand}
\end{align}
Combining \eqref{eq:P1-Markov}--\eqref{eq:U-expand}, and using
$\Delta\gg\log \log\frac{1}{\alpha}$, then gives
\begin{align}
\P \Big(
 \max_{|h|\leq H}|P_1(t+h)|>U \Big)
 \ll \exp\bigg(-2k_\gamma y-\frac{y^2}{\log(1/\alpha)}- 2\Delta +O \Big(  \log \log\frac{1}{\alpha}\Big)\bigg).
 \label{eq:P1-final}
\end{align}

Now using \eqref{bpx} we have that
\begin{equation} 
\label{squares} \max_{|h| \leq H} \Re P_2(t+h) \leq \frac{ \log \log \log T}{2}+ O(1) \leq \Delta.\end{equation}
Let $C\geq4$ be large enough, and define
\begin{equation}
\label{ey}
 \mathcal E_y=
 \Big\{t\in[T,2T]:
 \max_{|h|\leq H}W_\alpha(t+h) \geq k_\gamma \log\frac{1}{\alpha}+C\Delta+y\Big\}.
\end{equation}
If $t\in\mathcal E_y$, then for some $|h|\leq H$,
\[
 \Re P_1(t+h)+\Re P_2(t+h)
 >k_\gamma \log\frac{1}{\alpha}+(C-1)\Delta+y.
\]
Thus, by \eqref{squares}, we obtain that
\[
 |P_1(t+h)|>k_\gamma \log\frac{1}{\alpha} +2\Delta+y.
\]
We deduce from \eqref{eq:P1-final} that
\begin{align}
 \text{meas}(\mathcal E_y)
 &\ll T\exp\Big(-2k_\gamma y-\frac{y^2}{\log(1/\alpha)}-\Delta\Big)\ll T\exp\Big(-2k_\gamma y-\frac{y^2}{\log(1/\alpha)}\Big). \label{measey}
\end{align}
This completes the proof of the theorem.
\end{proof}

 We show below that the range \eqref{alpha_range} in Theorem \ref{prop:main} can be slightly enlarged at the expense of a weaker error term. We obtain the following leading order
bound with probability $1-o(1)$, without the explicit
$y$-dependent tail estimate.

\begin{thm}\label{expanded_range}
Assume RH. Fix $0<c<1$, $\gamma\geq0$ and $\delta>0$. Let $k_\gamma$ be as in \eqref{kgamma}.
Uniformly for
\begin{equation*}
 0<\alpha\leq(\log T)^{-c},\qquad
 \frac{\alpha\log T}{\log\log T}\longrightarrow\infty,
\end{equation*}
 we have
\[
 \P\big(M_{\alpha,\gamma}(t)
 <\alpha^{k_\gamma+\delta}\big)=o(1).
\]
\end{thm}

\begin{proof}
Let $H=(\log T)^\gamma$. We follow the proof of
Theorem \ref{prop:main} but
replace \eqref{eq:u-x} by
\begin{equation*}\label{eq:minimum-leading-cutoff}
 u=2\log\Big(\frac{\alpha\log T}{\log\log T}\Big),
 \qquad x=\exp\Big(\frac{u}{\alpha}\Big).
\end{equation*}
Then $u\to\infty$, and
\begin{equation}\label{eq:minimum-leading-length}
 \frac{\log(1/\alpha)\log x}{\log T}
 \ll\frac{\log\log T}{\alpha\log T}
       \log\Big(\frac{\alpha\log T}{\log\log T}\Big)=o(1).
\end{equation}
The remainder in \eqref{eq:majorant} is now
\[
 \frac{\log T}{\log x}\log\frac1{1-x^{-\alpha}}+O(1)
 \ll\frac{\alpha\log T}{u}e^{-u}+1
 =o\Big(\log\frac{1}{\alpha}\Big).
\]

Note that estimate \eqref{eq:B1} still holds.
For $P_1$, take
$$U=\Big(k_\gamma+\frac{\delta}{3}\Big)\log\frac{1}{\alpha}\qquad \text{and}\qquad q=\Big\lfloor \frac{U^2}{B_1}\Big\rfloor$$ in \eqref{eq:P1-Markov}.
Then $q\asymp \log\frac{1}{\alpha}$.
Thus \eqref{eq:minimum-leading-length} implies that $P_1^q$ has length
$T^{o(1)}$. Also, since $\log x=u/\alpha$ and
$1\leq u\ll \log\frac{1}{\alpha}$, equation \eqref{eq:entropy} still holds. 
Hence, \eqref{eq:P1-Markov} and \eqref{eq:factorial} give
\begin{align*}
 \P\Big(\max_{|h|\leq H}|P_1(t+h)|>U\Big)
 &\ll (1+qH\log x)\sqrt q\exp\Big(-\frac{U^2}{B_1}\Big)\\
 &\ll\exp\bigg(-\Big(\Big(k_\gamma+\frac{\delta}{3}\Big)^2-k_\gamma^2+o(1)\Big)
                         \log\frac{1}{\alpha}\bigg)=o(1).
\end{align*}

For $P_2$, equation \eqref{squares} is still true. So we get that outside the exceptional set for $P_1$, we have 
\[
 W_\alpha(t+h)
 \leq\Big(k_\gamma+\frac{\delta}{3}\Big)\log\frac{1}{\alpha}+o\Big(\log\frac{1}{\alpha}\Big)
 <(k_\gamma+\delta)\log\frac{1}{\alpha}.
\]
Taking exponentials proves the assertion.
\end{proof}

For smaller shifts, the negative-moment estimates give the following
weaker lower bound on intervals of fixed length.

\begin{prop}\label{weaker_bd}
Assume RH. Fix $\delta>0$.
Uniformly for
\begin{equation*}
 1\ll \alpha \log T\ll   \log \log T,
\end{equation*}
we have
\begin{equation*}
 \P\big(
M_{\alpha,0}(t) \geq \alpha^{5/4+\delta}\big)
 \longrightarrow1.
\end{equation*}
\end{prop}
\begin{proof}
Fix $0<k<1/2$. We use 
Theorem~1.4 of \cite{BF} and Theorem \ref{negativemom} to get that
\[
 \frac1T\int_{T/2}^{3T}
 \Big|\zeta(\tfrac12+\beta+it)\Big|^{-2k}dt
 \ll \alpha^{-k^2}
\]
for $\alpha/2\leq\beta\leq3\alpha/2$. Apply the mean-value inequality to $|\zeta|^{-2k}$ on a disk
of radius $\alpha/2$ about $1/2+\alpha+i(t+h)$.
All these disks for $|h|\leq1$ lie in the rectangle
$\alpha/2\leq\Re s-1/2\leq3\alpha/2$,
$t-2\leq\Im s\leq t+2$. Hence
\begin{align*}
\frac1T\int_T^{2T}\max_{|h|\leq1}
 \Big|\zeta\Big(\tfrac12+\alpha+i(t+h)\Big)\Big|^{-2k}dt&\ll\frac1{\alpha^2}
 \int_{\alpha/2}^{3\alpha/2}\frac1T\int_{T-2}^{2T+2}
 \Big|\zeta(\tfrac12+\beta+iv)\Big|^{-2k}dv\,d\beta
 \notag\\
 &
 \ll\alpha^{-1-k^2}.\label{eq:minimum-maximal-negative-moment}
\end{align*}
Thus,
\[
 \P\big(M_{\alpha,0}(t)
       <\alpha^{5/4+\delta}\big)
 \ll\alpha^{2k(5/4+\delta)-1-k^2}=o(1),
\]
on choosing $k$ sufficiently close to $1/2$.
This finishes the proof.
\end{proof}

\subsection{The upper bound for the minimum}
We now focus on the case $\gamma=0$, i.e., $|h| \leq 1$.
We will prove the following complementary upper bound for $M_{\alpha}$. 

\begin{thm}\label{prop:minimum-upper}
Assume RH. Fix $0<c,\delta<1$ and $M_0>0$.
Uniformly for
\begin{equation*}\label{eq:minimum-upper-range}
 \frac{M_0}{\log T}\leq\alpha\ll (\log T)^{-c},
\end{equation*}
we have
\begin{equation*}\label{eq:minimum-upper-result}
 \P\big(M_{\alpha,0}(t)
       >\alpha^{1-\delta}\big)=o(1).
\end{equation*}
\end{thm}

We first need the following lemma. We will only sketch the proof, since it is similar to the proof of Proposition $3.2$ in \cite{ABBRS}. 

\begin{lem}
    \label{ABBRS}
  Fix $0<c_1<c_2<1$ and $0<\nu<1$. Uniformly for
$$
(\log T)^{c_1}\leq Y\leq(\log T)^{c_2},
$$ we have
\begin{equation*}
 \P\bigg(\min_{|h|\leq1/2}
  \Big(\Re\sum_{p\leq e^Y}p^{-1/2-i(t+h)}\Big)
       <-(1-\nu)\log Y\bigg)=1-o(1).
\end{equation*}
\end{lem}
\begin{proof}
Let $K$ be a sufficiently large integer.
We use the intervals from \cite[(10)]{ABBRS}
and the corresponding prime sums with
\[
 J_j=\Big(\exp\big((\log T)^{j/K}\big),
             \exp\big((\log T)^{(j+1)/K}\big)\Big],
 \qquad P_j(v)=-\Re\sum_{p\in J_j}p^{-1/2-iv},
\]
for $1\leq j\leq K-3$. Proposition~3.2 in \cite{ABBRS} states that
\[
 \P\bigg(\max_{|h|\leq1/4}\min_{1\leq j\leq K-3}
      \Big( \Re\sum_{p\in J_j}p^{-\sigma_0-i(t+h)}\Big)
       >\frac{1-\nu/4}{K}\log\log T\bigg)=1-o(1),
\]
where
$
 \sigma_0=1/2+(\log T)^{3/(2K)-1}.
$
A minor modification of the argument in \cite{ABBRS} leads to the same conclusion
with coefficients $-p^{-1/2}$ instead of $p^{-\sigma_0}$. We then get that 
\begin{equation}
    \label{new_conclusion}
 \P\Big(\max_{|h|\leq1/4}\min_{1\leq j\leq K-3}P_j(t+h)
       >\frac{1-\nu/4}{K}\log\log T\Big)=1-o(1).
\end{equation}

Now note that $e^Y$ is not necessarily an endpoint of the chosen intervals. We first take
all complete blocks below $e^Y$, and then bound the contributions from the missing primes.
Let $m$ be such that 
$$ (\log T)^{(m+1)/K}\leq Y<(\log T)^{(m+2)/K}. $$
Note that 
\[
 m=\Big\lfloor\frac{K\log Y}{\log\log T}\Big\rfloor-1.
\]
and
\[
 \frac mK\log\log T\geq\log Y-\frac2K\log\log T.
\]
Thus, at the point provided by \eqref{new_conclusion}, we have 
\[
 \sum_{j=1}^mP_j(t+h)
 >\Big(1-\frac{\nu}{4}\Big)\Big(\log Y-\frac2K\log\log T\Big)
 \geq\Big(1-\frac{\nu}{2}\Big)\log Y,
\]
provided $Kc_1\geq8/\nu$, since
$\log Y\geq c_1\log\log T$.

The omitted primes are those below the first block and those
between the last complete block and $e^Y$. Write their sum as
\[
 R(v)=\sum_{p\leq\exp((\log T)^{1/K})}p^{-1/2-iv}
      +\sum_{\exp((\log T)^{(m+1)/K})<p\leq e^Y}p^{-1/2-iv}.
\]
Mertens' estimate and the definition of $m$ give
\begin{align*}
 &\sum_{p\leq\exp((\log T)^{1/K})}\frac1p
  +\sum_{\exp((\log T)^{(m+1)/K})<p\leq e^Y}\frac1p\\
 &\qquad=\frac{\log\log T}K+\log Y
                -\frac{(m+1)\log\log T}{K}+O(1)\leq\frac{2\log\log T}K+O(1)
       \leq\frac{2\log Y}{K\varepsilon}+O(1).
\end{align*}
Now we apply Lemma~\ref{lem:maximal} with $H=1$ and
$q=\lfloor\log Y\rfloor$. Note that the length condition is satisfied because $qY = o(\log T)$.
Markov's inequality and Stirling's formula therefore give
\[
 \P\Big(\max_{|h|\leq1}|R(t+h)|>\frac\nu2\log Y\Big)
 \ll(1+qY)
       \Big(\frac{C}{K}
                    +\frac{C}{\log Y}\Big)^q=o(1).
\]
Outside the two exceptional sets (the one obtained above and the one in \eqref{new_conclusion}), at the same point $|h|\leq1/4$, we have
\[
 -\Re\sum_{p\leq e^Y}p^{-1/2-i(t+h)}
 =\sum_{j=1}^mP_j(t+h)-\Re R(t+h)
 >(1-\nu)\log Y.
\]
This finishes the proof.
    \end{proof}

We are now ready to begin the proof of Theorem \ref{prop:minimum-upper}.

\begin{proof}[Proof of Theorem \ref{prop:minimum-upper}]
We first prove the assertion for $|h|\leq1/2$, when
\begin{equation}\label{eq:minimum-upper-mesoscopic}
 (\log T)^{-1+\eta}\leq\alpha\leq(\log T)^{-c},
\end{equation}
where $\eta>0$ is fixed and chosen sufficiently small that
\[
  \eta<1-c\qquad\text{and}
  \qquad (1-\eta)^2>1-\delta.
\]

Let $u,x$ be as in \eqref{eq:u-x}. The extremal minorant of
Carneiro and Chandee \cite[Lemma~5 and Section~2.1]{CC}, inserted in
the explicit formula, gives the inequality 
\begin{equation}\label{eq:minimum-CC-minorant}
 W_\alpha(v)\geq
 \Re\sum_{n\leq x}\frac{b_\alpha^-(n;x)}{n^{1/2+\alpha+iv}}
 -\frac{\log T}{\log x}\log(1+x^{-\alpha})-O(1),
 \end{equation}
with  $v\in[T-1,2T+1]$.
Here we have
\begin{align*}
 b_\alpha^-(n;x)=-\Lambda(n)\sum_{j\geq0}(-1)^j
 \Big(\frac{j+1}{\log n+j\log x}x^{-j\alpha}
 -\frac{(j+1)n^{2\alpha}}{(j+2)\log x-\log n}
                                  x^{-(j+2)\alpha}\Big).
\end{align*}
As in Lemma \ref{asympta}, we have
\begin{equation}\label{eq:minimum-minus-coefficients}
 b_\alpha^-(p^r;x)=-\frac1r+O\Big(\frac{\log p}{\log x}\Big),
 \qquad |b_\alpha^-(p^r;x)|\ll\frac1r.
\end{equation}

Fix $0<\nu<1/2$, to be chosen sufficiently small in terms of
$\delta$, and put $x_0=\exp(\alpha^{-1+\nu})$. We have
$\log x_0/\log x=\alpha^\nu/u=o(1)$, so $x_0<x$.
We write
\begin{align}
 R_1(v)
 &:=\sum_{p\leq x_0}
       \frac{b_\alpha^-(p;x)p^{-\alpha}+1}{p^{1/2+iv}}
    +\sum_{x_0<p\leq x}
       \frac{b_\alpha^-(p;x)}{p^{1/2+\alpha+iv}},
       \label{eq:minimum-R1-definition}\\
 R_2(v)
 &:=\sum_{p^2\leq x}
       \frac{b_\alpha^-(p^2;x)}{p^{1+2\alpha+2iv}}.
       \label{eq:minimum-R2-definition}
\end{align}
Then 
\[
 \sum_{p\leq x}\frac{b_\alpha^-(p;x)}{p^{1/2+\alpha+iv}}
 =-\sum_{p\leq x_0}p^{-1/2-iv}+R_1(v).
\]
The higher prime powers are bounded uniformly in $v$, and hence \eqref{eq:minimum-CC-minorant} becomes
\begin{equation*}\label{eq:minimum-remainder-decomposition}
 W_\alpha(v)\geq
 -\Re\sum_{p\leq x_0}p^{-1/2-iv}
 +\Re R_1(v)+\Re R_2(v)-C,
\end{equation*}
for some $C>0$.

We next bound $R_1$ and $R_2$ uniformly for $v=t+h$, $|h|\leq1$.
For $p\leq x_0$, \eqref{eq:minimum-minus-coefficients} and
the fact that $1-e^{-s}\leq s$ for $s\geq0$ give
\[
 |b_\alpha^-(p;x)p^{-\alpha}+1|
 \leq |b_\alpha^-(p;x)+1|p^{-\alpha}+1-p^{-\alpha}
 \ll\Big(\alpha+\frac1{\log x}\Big)\log p.
\]
It follows that
\begin{align*}
 \sum_{p\leq x_0}
 \frac{|b_\alpha^-(p;x)p^{-\alpha}+1|^2}{p} \ll\Big(\alpha+\frac1{\log x}\Big)^2
       \sum_{p\leq x_0}\frac{(\log p)^2}{p} \ll\Big(\alpha+\frac1{\log x}\Big)^2(\log x_0)^2
 \ll\alpha^{2\nu}.
\end{align*}

For the remaining primes, we use $x_0<e^{1/\alpha}<x$ and get that
the sums of the squared coefficients of
\eqref{eq:minimum-R1-definition} and
\eqref{eq:minimum-R2-definition}  satisfy
\begin{align*}
 &\sum_{p\leq x_0}
       \frac{|b_\alpha^-(p;x)p^{-\alpha}+1|^2}{p}
  +\sum_{x_0<p\leq x}
       \frac{|b_\alpha^-(p;x)|^2}{p^{1+2\alpha}}
       \ll\nu\log\frac 1{\alpha},
       \\
 &\sum_{p^2\leq x}
       \frac{|b_\alpha^-(p^2;x)|^2}{p^{2+4\alpha}}
    \ll1.
\end{align*}

Let $q=\lfloor\log\frac{1}{\alpha}\rfloor$. Then
\[
 \frac{q\log x}{\log T}
 =\frac{q\log(\alpha\log T)}{\alpha\log T}
 \ll\frac{(\log\log T)^2}{(\log T)^\eta}=o(1).
\]
Applying
Lemma~\ref{lem:maximal} with $H=1$ to 
$R_1$ and $R_2$, respectively, gives
\begin{align*}
 \E \Big(\max_{|h|\leq1}|R_1(t+h)|^{2q}\Big)
 &\ll(1+q\log x)q!\bigg(C\Big(\nu\log\frac{1}{\alpha}\Big)\bigg)^q,\\
 \E \Big( \max_{|h|\leq1}|R_2(t+h)|^{2q}\Big)
 &\ll(1+q\log x)q!C^q.
\end{align*}
Markov's inequality and Stirling again give
\begin{align}
 \P\Big(\max_{|h|\leq1}|R_1(t+h)|
            >\frac\delta8\log\frac{1}{\alpha}\Big)
 &\ll(1+q\log x)
       \bigg(\frac{C}{\delta^2}
\Big(\nu+\frac1{\log(1/\alpha)}\Big)\bigg)^q,
       \label{eq:minimum-R1-probability}\\
 \P\Big(\max_{|h|\leq1}|R_2(t+h)|
                  >\frac\delta8\log\frac{1}{\alpha}\Big)
 &\ll(1+q\log x)
       \Big(\frac{C}{\delta^2\log(1/\alpha)}\Big)^q.
       \label{eq:minimum-R2-probability}
\end{align}

Now note that
\[
 \log(1+q\log x)
 =\log\frac{1}{\alpha}+O\Big(\log\log\frac{1}{\alpha}\Big).
\]
We choose $\nu$ so that $(1-\nu)^2\geq1-\delta/4$ and
$C\nu/\delta^2\leq e^{-4}$. The term in parenthesis in
\eqref{eq:minimum-R1-probability} is then at most $e^{-3}$
for large $T$, and so the right hand side is
\[
 \ll\exp\bigg(-2\log\frac{1}{\alpha}
                  +O\Big(\log\log\frac{1}{\alpha}\Big)\bigg)=o(1).
\]
We also have the right-hand side of \eqref{eq:minimum-R2-probability} is
\[
 \ll\exp\bigg(-\log\frac{1}{\alpha}\log\log\frac{1}{\alpha}
                       +O\Big(\log\frac{1}{\alpha}\Big)\bigg)
 =o(1).
\]

We next use Lemma \ref{ABBRS} with $Y= \log x_0$ (note that $Y \log Y = o(\log T)$). It follows that outside a set of measure $o(T)$, there is a point
$|h|\leq1/2$ such that
$$
 \Re\sum_{p\leq x_0}p^{-1/2-i(t+h)}
 <-(1-\nu)\log\log x_0
 =-(1-\nu)^2\log\frac{1}{\alpha}
 \leq -\Big(1-\frac{\delta}{4}\Big)\log\frac{1}{\alpha}.
$$
Outside the union of this exceptional set and the two sets in
\eqref{eq:minimum-R1-probability}--\eqref{eq:minimum-R2-probability},
we have
\begin{align*}
 W_\alpha(t+h)>(1-\delta/2)\log\frac{1}{\alpha}-C
  >(1-\delta)\log\frac{1}{\alpha},
\end{align*}
for sufficiently large $T$. Equivalently,
$$
 \min_{|h|\leq1/2}
 \Big|\zeta\Big(\tfrac12+\alpha+i(t+h)\Big)\Big|
 =\exp\Big(-\max_{|h|\leq1/2}W_\alpha(t+h)\Big)
 <\alpha^{1-\delta}.
$$
This finishes the proof in the range \eqref{eq:minimum-upper-mesoscopic}. 

Now we assume that 
$$ \frac{M_0}{\log T} \leq \alpha \leq (\log T)^{-1+\eta}.$$
We begin with the Sobolev inequality used in the proof of
\cite[Proposition~2.1, eq.~(6)]{ABBRS}: for every
continuously differentiable complex-valued function $g$ on $[a,b]$,
\begin{equation}\label{eq:minimum-sobolev}
 \max_{a\leq v\leq b}|g(v)|^2
 \leq\frac{|g(a)|^2+|g(b)|^2}{2}
       +\int_a^b|g(v)g'(v)|\,dv.
\end{equation}
We apply this to $f(s)=\zeta(s)^{-1/4}$ on the four sides of
\[
 \mathcal R_t=\Big\{s:\frac12+\alpha\leq\Re s\leq2,
                          \ |\Im s-t|\leq1\Big\}.
\]

On either vertical side, take $g(v)=f(\sigma+iv)$ on
$[t-1,t+1]$. Averaging \eqref{eq:minimum-sobolev} over $t$
gives
\begin{align}
 &\frac1T\int_T^{2T}\max_{|h|\leq1}
             \Big|f\Big(\sigma+i(t+h)\Big)\Big|^2\,dt \ll\frac1T\int_{T-2}^{2T+2}
       \bigl(|f(\sigma+iv)|^2+
             |f(\sigma+iv)f'(\sigma+iv)|\bigr)\,dv.
 \label{eq:minimum-vertical-sobolev}
\end{align}
 On the two horizontal sides, we
apply \eqref{eq:minimum-sobolev} to
$g(\sigma)=f(\sigma+i(t\pm1))$ on $[1/2+\alpha,2]$ and we get 
\begin{equation*}\label{eq:minimum-horizontal-sobolev}
\begin{aligned}
&\frac1T\int_T^{2T}
\max_{1/2+\alpha\leq\sigma\leq2}
|f(\sigma+i(t\pm1))|^2\,dt \leq
\frac12\sum_{\sigma\in\{1/2+\alpha,\,2\}}
\frac1T\int_{T-2}^{2T+2}|f(\sigma+iv)|^2\,dv\\
&\qquad+
\int_{1/2+\alpha}^{2}
\Big(
\frac1T\int_{T-2}^{2T+2}
|f(\sigma+iv)f'(\sigma+iv)|\,dv
\Big)d\sigma.
\end{aligned}
\end{equation*}
Thus, after averaging over $t$, it remains to bound the means
of $|f|^2$ and $|ff'|$ on the vertical lines
$1/2+\alpha\leq\sigma\leq2$.

We first bound the mean of $|f|^2$.
Theorem~\ref{negativemom} with $k=1/4$ applies
when the shift is at most $(1/2+\varepsilon)\log\log T/\log T$, and
\cite[Theorem~1.4]{BF} applies to the larger shifts. We obtain that
\begin{equation}
    \label{eq:minimum-quarter-moment}
 \frac{1}{T}  \int_{T-2}^{2T+2} |f(\sigma+iv)|^2 \, dv \ll  \Big(\sigma-\frac12\Big)^{-1/16} \ll \alpha ^{-1/16}.\end{equation}

Next, we bound the mean of $|f'|^2$. For
$1/2+\alpha\leq\sigma\leq2$, take $r=\alpha/4$. Cauchy's
formula gives
\begin{align*}
 f'(\sigma+iv)&=\frac1{2\pi r}\int_0^{2\pi}
 f(\sigma+iv+re^{i\theta})e^{-i\theta}\,d\theta,
 \end{align*}
 so after using Cauchy--Schwarz, we get
 \begin{align*}
 |f'(\sigma+iv)|^2&\leq\frac1{2\pi r^2}\int_0^{2\pi}
 |f(\sigma+iv+re^{i\theta})|^2\,d\theta.
\end{align*}
Integrating over $v$ we get that
\begin{equation*}\label{eq:minimum-reciprocal-derivative}
 \frac1T\int_{T-2}^{2T+2}|f'(\sigma+iv)|^2\,dv
 \ll \alpha^{-2-1/16}.
\end{equation*}
The Cauchy--Schwarz inequality then gives
\begin{equation}\label{eq:minimum-reciprocal-product}
 \frac1T\int_{T-2}^{2T+2}
 |f(\sigma+iv)f'(\sigma+iv)|\,dv
 \ll \alpha^{-17/16}.
\end{equation}

Substituting \eqref{eq:minimum-quarter-moment} and
\eqref{eq:minimum-reciprocal-product} into
\eqref{eq:minimum-vertical-sobolev} bounds the average of
each vertical-side maximum by $O(\alpha^{-17/16})$.
We similarly get that each horizontal-side maximum has average
$O(\alpha^{-17/16})$ as well. By the maximum modulus
principle, we obtain
\begin{equation*}\label{eq:minimum-rectangle-moment}
 \frac1T\int_T^{2T}\max_{s\in\mathcal R_t}
                 |\zeta(s)|^{-1/2}\,dt
 \ll \alpha^{-17/16}.
\end{equation*}
Since $\alpha\geq M_0/\log T$, Markov's inequality yields
\begin{equation}\label{eq:minimum-rectangle-exception}
 \P\Big(\max_{s\in\mathcal R_t}|\zeta(s)|^{-1/2}
                         >(\log T)^3\Big)
 \ll (\log T)^{-31/16}=o(1).
\end{equation}

Now let $\beta= (\log T)^{-1+\eta}$. We have already shown that 
\begin{equation}\label{eq:minimum-beta-point}
 \max_{|h|\leq1/2}
\Big|\zeta\Big(\tfrac12+\beta+i(t+h)\Big)\Big|^{-1/2}
 \geq\beta^{-(1-\eta)/2}
\end{equation}
with probability $1-o(1)$. 

Outside the exceptional sets in
\eqref{eq:minimum-beta-point} and
\eqref{eq:minimum-rectangle-exception}, choose $|h|\leq1/2$
such that
\[
\Big|\zeta\Big(\tfrac12+\beta+i(t+h)\Big)\Big|^{-1}
 \geq\beta^{-(1-\eta)},
\]
and put
\[
 z=s-\Big(\frac12+\alpha+i(t+h)\Big),
 \qquad
 g(s)=\frac{\exp\bigl(32\log\log T\,z(z-(3/2-\alpha))\bigr)}
               {\zeta(s)}.
\]
This function is holomorphic on a neighbourhood of
$\mathcal R_t$. Writing $z=x+iy$, we have
\[
 \Re\Big(z\big(z-(3/2-\alpha)\big)\Big)
 =x(x-(3/2-\alpha))-y^2.
\]
On the left and right sides, $x=0$ and $x=3/2-\alpha$,
respectively, so the exponential factor has modulus at most $1$.
On the right side, moreover, $|g(s)|\leq\zeta(2)$.
On the horizontal sides $|y|\geq1/2$ and
$0\leq x\leq3/2-\alpha$, so the exponential factor is at most
$(\log T)^{-8}$. By \eqref{eq:minimum-rectangle-exception},
$|1/\zeta(s)|\leq(\log T)^6$ throughout the rectangle.
Hence $|g(s)|\leq(\log T)^{-2}$ on its horizontal sides.

At the interior point $s=1/2+\beta+i(t+h)$, however,
\[
 |g(s)|=
 e^{-32\log\log T\,(\beta-\alpha)(3/2-\beta)}
 |\zeta(s)|^{-1}
 \geq e^{-48\beta\log\log T}\beta^{-(1-\eta)}
 \longrightarrow\infty.
\]
The maximum modulus principle therefore forces a value at least
this large on the left side. Since the exponential factor has
modulus at most $1$ there, we conclude that
\begin{equation*}\label{eq:minimum-rectangle-transfer}
 \max_{|h|\leq1}
 \Big|\zeta\Big(\tfrac12+\alpha+i(t+h)\Big)\Big|^{-1}
 \geq e^{-48\beta\log\log T}\beta^{-(1-\eta)}.
\end{equation*}
Here $\beta\log\log T=o(1)$. Taking reciprocals gives
\[
 M_{\alpha,0}(t)
 \leq e^{48\beta\log\log T}\beta^{1-\eta}
 \leq2(\log T)^{-(1-\eta)^2}
 \leq\alpha^{1-\delta}
\]
for sufficiently large $T$.
The last inequality follows since $(1-\eta)^2>1-\delta$ and
$\alpha\geq M_0/\log T$. This completes the proof.
\end{proof}

\section{Moments in short intervals}
\label{section_shortintervals}
       \begin{proof}[Proof of Theorem \ref{momshortint}]
       For simplicity of notation, let
       $$Z_{k}(t) =\frac{1}{2H} \int_{|h| \leq H} \Big|\zeta\Big(\tfrac12+\alpha+i(t+h)\Big)\Big|^{-2k} dh . $$
           Let 
           $$S_t(V) = \frac{1}{2H} \text{meas} \Big\{  |h| \leq H :  W_{\alpha}(t+h)>V\Big\},$$ and note that we have
           $$Z_k(t) =2k \int_{-\infty}^{\infty} e^{2kV} S_t(V) \, dV.$$
             Since
       $$2k \int_{-\infty}^0 e^{2kV} S_t(V) \, dV \leq 1,$$ we get that
       $$Z_k(t) \leq 1+ 2k \int_0^{\infty} e^{2kV} S_t(V) \, dV.$$
      
       Now pick any $0<a <\min\{k/2,1/(4k)\}$, and let $b >a$ such that $$
 b>\max\Big\{k,\sqrt{1+\frac{2\gamma}{c_0}}+\frac{M}{2}\Big\}+2.
$$ 
       We write
       \begin{align} \label{split}
       2k \int_0^{\infty} e^{2kV} S_t(V) \, dV &= 2k \int_0^{a \log\frac{1}{\alpha}} e^{2kV} S_t(V) \, dV+ 2k\int_{a \log\frac{1}{\alpha}}^{b \log\frac{1}{\alpha}} e^{2kV} S_t(V) \, dV\nonumber\\
       &\qquad+2k\int_{b \log\frac{1}{\alpha}}^{\infty} e^{2kV} S_t(V) \, dV.  
       \end{align}
       
       Note that
       \begin{equation}
       \label{firstint}
       2k \int_0^{a \log\frac{1}{\alpha}} e^{2kV} S_t(V) \, dV \leq e^{2k a \log(1/\alpha)} \leq \alpha^{-k^2},\end{equation} where we bounded $S_t(V) \leq 1$ and used the fact that $a <k/2$. 
       
       Now using Theorem \ref{prop:main} (see equations \eqref{ey} and \eqref{measey}), outside on a set of measure $\ll T/A$, we have that 
       $$ \max_{|h| \leq (\log T)^{\gamma}} W_{\alpha}(t+h) <k_{\gamma} \log\frac{1}{\alpha} +C  \log \log\frac{1}{\alpha}+\frac{\log A}{2k_{\gamma}} , $$
       for $C>0$ sufficiently large. By the choice of $b$, it follows that
       $$ \max_{|h| \leq (\log T)^{\gamma}} W_{\alpha}(t+h) < b \log\frac{1}{\alpha},  $$ for all $t$ except a set of measure $\ll T/A$. For the good $t$, it follows that for $V>b \log\frac{1}{\alpha}$, we have
       $$ S_t(V) = 0,$$ and hence
       $$ 2k\int_{b \log\frac{1}{\alpha}}^{\infty} e^{2kV} S_t(V) \, dV =0.$$
       
       We now focus on the middle term in \eqref{split}.
       From Theorem \ref{mainthm} we have that
       \begin{align}
           \frac{1}{T} \int_T^{2T} S_t(V) \, dt  &=\frac{1}{2TH} \int_{T}^{2T} \int_{-H}^H \mathbf{1} \big(W_{\alpha}(t+h) >V\big) \, dh \, dt \nonumber\\
           &\leq \frac{1}{T} \text{meas} \Big\{  u \in [T-H,2T+H] : W_{\alpha}(u) >V \Big\}\nonumber\\
           &\ll \frac{1}{\sqrt{\log(1/\alpha)}} \exp \Big(  - \frac{V^2}{\log(1/\alpha)}\Big). \label{stv}
       \end{align}
By Fubini's theorem
and \eqref{stv}, we obtain that
\begin{align*}
\frac1T\int_T^{2T}\Big(
 2k\int_{a\log\frac{1}{\alpha}}^{b\log\frac{1}{\alpha}}
                         e^{2kV}S_t(V)\,dV\Big)dt&\ll \frac1{\sqrt{\log(1/\alpha)}}
 \int_{a\log\frac{1}{\alpha}}^{b\log\frac{1}{\alpha}}
       \exp\Big(2kV-\frac{V^2}{\log(1/\alpha)}\Big)\,dV\\
 &\leq\frac{\alpha^{-k^2}}{\sqrt{\log(1/\alpha)}}
 \int_{\mathbb R}
 \exp\Big(-\frac{(V-k\log(1/\alpha))^2}{\log(1/\alpha)}\Big)\,dV \ll\alpha^{-k^2}.
\end{align*}
By Markov's inequality we get that, outside a set of measure $\ll T/A$,
\begin{equation*}\label{eq:short-moment-middle-bound}
 2k\int_{a\log\frac{1}{\alpha}}^{b\log\frac{1}{\alpha}}
                         e^{2kV}S_t(V)\,dV
 \ll A\alpha^{-k^2}.
\end{equation*}
Combining with \eqref{firstint} we have
$$Z_k(t) \ll  A \alpha^{-k^2},$$ except for on a set of measure $\ll T/A$. 

Now suppose that $k>k_\gamma$. We use the same argument, but
integrate only up to $$U:=k_{\gamma} \log\frac{1}{\alpha} +C   \log \log\frac{1}{\alpha}+\frac{\log A}{2k_{\gamma}} .$$ Since
\[
 \exp\Big(2kV-\frac{V^2}{\log(1/\alpha)}\Big)
 \leq\alpha^{-k_\gamma^2}e^{2(k-k_\gamma)V},
\]
equation \eqref{stv}, Fubini's theorem, and Markov's inequality,
together with the fact that $2ak<1/2$, give
\begin{align*}
 Z_k(t)
 &\ll \alpha^{-1/2}
   +\frac{A\alpha^{-k_\gamma^2}}{\sqrt{\log(1/\alpha)}}
       \int_{a\log\frac{1}{\alpha}}^U e^{2(k-k_\gamma)V}\,dV \ll \alpha^{-1/2}
   +\frac{A\alpha^{-k_\gamma^2}e^{2(k-k_\gamma)U}}
          {(k-k_\gamma)\sqrt{\log(1/\alpha)}},
\end{align*}
outside
a set of measure $\ll T/A$. Note that the first term is dominated by the second.
Substituting the value of $U$ proves
\eqref{supercritical}.
      \end{proof}

\bibliographystyle{amsplain}
\bibliography{references}

\providecommand{\bysame}{\leavevmode\hbox to3em{\hrulefill}\thinspace}
\providecommand{\MR}{\relax\ifhmode\unskip\space\fi MR }
\providecommand{\MRhref}[2]{%
  \href{http://www.ams.org/mathscinet-getitem?mr=#1}{#2}
}
\providecommand{\href}[2]{#2}
\begin{thebibliography}{10}

\bibitem{AB}
L.-P. Arguin and E.~Bailey, \emph{Large deviation estimates of {Selberg}'s
  central limit theorem and applications}, International Mathematics Research
  Notices \textbf{2023} (2023), no.~23, 20574--20612. \MR{4675078}

\bibitem{AB-lower}
\bysame, \emph{Lower bounds for the large deviations of {Selberg}'s central
  limit theorem}, Mathematika \textbf{71} (2025), no.~1, e70002.

\bibitem{ABR}
L.-P. Arguin, E.~Bailey, and A.~Roberts, \emph{Conditional upper bounds for
  large deviations and moments of the {R}iemann zeta function}, Preprint,
  arXiv:2604.25579, 2026.

\bibitem{ABB-CUE}
L.-P. Arguin, D.~Belius, and P.~Bourgade, \emph{Maximum of the characteristic
  polynomial of random unitary matrices}, Communications in Mathematical
  Physics \textbf{349} (2017), no.~2, 703--751.

\bibitem{ABBRS}
L.-P. Arguin, D.~Belius, P.~Bourgade, M.~Radziwi{\l}{\l}, and K.~Soundararajan,
  \emph{Maximum of the {R}iemann zeta function on a short interval of the
  critical line}, Communications on Pure and Applied Mathematics \textbf{72}
  (2019), no.~3, 500--535. \MR{3911893}

\bibitem{ABR1}
L.-P. Arguin, P.~Bourgade, and M.~Radziwi{\l}{\l}, \emph{The
  {F}yodorov-{H}iary-{K}eating {C}onjecture. {I}}, Preprint, arXiv:2007.00988,
  2020.

\bibitem{ABR2}
\bysame, \emph{The {F}yodorov-{H}iary-{K}eating {C}onjecture. {II}}, Preprint,
  arXiv:2307.00982, 2023.

\bibitem{AC-Dirichlet}
L.-P. Arguin and N.~Creighton, \emph{Upper bounds on large deviations of
  {Dirichlet} {$L$}-functions in the {$q$}-aspect}, Journal of Number Theory
  \textbf{273} (2025), 96--158.

\bibitem{AC}
\bysame, \emph{Lower bounds for the large deviations and moments of the
  {Riemann} zeta function on the critical line}, Preprint, arXiv:2603.01711,
  2026.

\bibitem{AH}
L.-P. Arguin and J.~Hamdan, \emph{The {Fyodorov--Hiary--Keating} conjecture on
  mesoscopic intervals}, Preprint, arXiv:2405.06474, 2024.

\bibitem{AOR}
L.-P. Arguin, F.~Ouimet, and M.~Radziwi{\l}{\l}, \emph{Moments of the {Riemann}
  zeta function on short intervals of the critical line}, Annals of Probability
  \textbf{49} (2021), no.~6, 3106--3141.

\bibitem{BDR}
M.~Balazard and A.~de~Roton, \emph{Notes de lecture de l'article ``partial sums
  of the {M}\"obius function'' de {Kannan Soundararajan}}, Preprint,
  arXiv:0810.3587, 2008.

\bibitem{BK-negative}
M.~V. Berry and J.~P. Keating, \emph{Clusters of near-degenerate levels
  dominate negative moments of spectral determinants}, Journal of Physics A:
  Mathematical and General \textbf{35} (2002), no.~1, L1--L6. \MR{1891805}

\bibitem{BF}
H.~M. Bui and A.~Florea, \emph{Negative moments of the {R}iemann
  zeta-function}, Journal f{\"u}r die reine und angewandte Mathematik
  \textbf{2024} (2024), no.~806, 247--288.

\bibitem{CC}
E.~Carneiro and V.~Chandee, \emph{Bounding {$\zeta(s)$} in the critical strip},
  Journal of Number Theory \textbf{131} (2011), no.~3, 363--384.

\bibitem{CFZ-ratios}
J.~B. Conrey, D.~W. Farmer, and M.~R. Zirnbauer, \emph{Autocorrelation of
  ratios of {$L$}-functions}, Communications in Number Theory and Physics
  \textbf{2} (2008), no.~3, 593--636. \MR{2482944}

\bibitem{CS-ratios}
J.~B. Conrey and N.~C. Snaith, \emph{Applications of the {$L$}-functions
  {R}atios {C}onjectures}, Proceedings of the London Mathematical Society
  \textbf{94} (2007), no.~3, 594--646. \MR{2325314}

\bibitem{Creighton-twists}
N.~Creighton, \emph{Upper tail distributions of central {$L$}-values of
  quadratic twists of elliptic curves at the variance scale}, Preprint,
  arXiv:2507.08640, 2025.

\bibitem{Farmer-ratios}
D.~W. Farmer, \emph{Long mollifiers of the {R}iemann zeta-function},
  Mathematika \textbf{40} (1993), no.~1, 71--87. \MR{1239132}

\bibitem{Forrester-Keating}
P.~J. Forrester and J.~P. Keating, \emph{Singularity dominated strong
  fluctuations for some random matrix averages}, Communications in Mathematical
  Physics \textbf{250} (2004), no.~1, 119--131. \MR{2092032}

\bibitem{Fyodorov-negative}
Y.~V. Fyodorov, \emph{Negative moments of characteristic polynomials of random
  matrices: {I}ngham-{S}iegel integral as an alternative to
  {H}ubbard-{S}tratonovich transformation}, Nuclear Physics B \textbf{621}
  (2002), no.~3, 643--674. \MR{1877952}

\bibitem{FHK}
Y.~V. Fyodorov, G.~A. Hiary, and J.~P. Keating, \emph{Freezing transition,
  characteristic polynomials of random matrices, and the {R}iemann zeta
  function}, Physical Review Letters \textbf{108} (2012), no.~17, 170601.

\bibitem{FK-negative}
Y.~V. Fyodorov and J.~P. Keating, \emph{Negative moments of characteristic
  polynomials of random {GOE} matrices and singularity-dominated strong
  fluctuations}, Journal of Physics A: Mathematical and General \textbf{36}
  (2003), no.~14, 4035--4046. \MR{1984532}

\bibitem{FK}
\bysame, \emph{Freezing transitions and extreme values: random matrix theory,
  and disordered landscapes}, Philosophical Transactions of the Royal Society
  of London. Series A. Mathematical, Physical and Engineering Sciences
  \textbf{372} (2014), no.~2007, 20120503, 32. \MR{3151088}

\bibitem{Gonek}
S.~M. Gonek, \emph{On negative moments of the {R}iemann zeta-function},
  Mathematika \textbf{36} (1989), no.~1, 71--88. \MR{1014202}

\bibitem{Harper-model}
A.~J. Harper, \emph{A note on the maximum of the riemann zeta function, and
  log-correlated random variables}, Preprint, arXiv:1304.0677, 2013.

\bibitem{Harper-moments}
\bysame, \emph{Sharp conditional bounds for moments of the {Riemann} zeta
  function}, Preprint, arXiv:1305.4618, 2013.

\bibitem{Harper-partition}
\bysame, \emph{On the partition function of the {Riemann} zeta function, and
  the {F}yodorov--{H}iary--{K}eating conjecture}, Preprint, arXiv:1906.05783,
  2019.

\bibitem{Harper-survey}
\bysame, \emph{The {Riemann} zeta function in short intervals [after
  {Najnudel}, and {Arguin}, {Belius}, {Bourgade}, {Radziwi{\l}{\l}} and
  {Soundararajan}]}, S{\'e}minaire Bourbaki. Vol. 2018/2019. Expos{\'e}s
  1151--1165, Ast{\'e}risque, no. 422, Soci{\'e}t{\'e} Math{\'e}matique de
  France, 2020, Exp. No. 1161, pp.~391--414. \MR{4224641}

\bibitem{HRS}
W.~Heap, M.~Radziwi{\l}{\l}, and K.~Soundararajan, \emph{Sharp upper bounds for
  fractional moments of the {R}iemann zeta function}, The Quarterly Journal of
  Mathematics \textbf{70} (2019), no.~4, 1387--1396. \MR{4045106}

\bibitem{KS-moments}
J.~P. Keating and N.~C. Snaith, \emph{Random matrix theory and
  {$\zeta(1/2+it)$}}, Communications in Mathematical Physics \textbf{214}
  (2000), no.~1, 57--89. \MR{1794265}

\bibitem{Landau}
E.~Landau, \emph{{\"U}ber die {M\"obiussche} funktion}, Rendiconti del Circolo
  Matematico di Palermo \textbf{48} (1924), 277--280.

\bibitem{Littlewood}
J.~E. Littlewood, \emph{Quelques cons{\'e}quences de l'hypoth{\`e}se que la
  fonction {$\zeta(s)$} de {Riemann} n'a pas de z{\'e}ros dans le demi-plan
  {$\Re(s)>1/2$}}, Comptes rendus de l'Acad\'emie des Sciences \textbf{154}
  (1912), 263--266.

\bibitem{MM}
H.~Maier and H.~L. Montgomery, \emph{The sum of the {M\"obius} function},
  Bulletin of the London Mathematical Society \textbf{41} (2009), no.~2,
  213--226.

\bibitem{MV}
Hugh~L. Montgomery and Robert~C. Vaughan, \emph{{Hilbert}'s inequality},
  Journal of the London Mathematical Society \textbf{8} (1974), no.~1, 73--82.

\bibitem{Najnudel}
J.~Najnudel, \emph{On the extreme values of the {R}iemann zeta function on
  random intervals of the critical line}, Probability Theory and Related Fields
  \textbf{172} (2018), no.~1-2, 387--452. \MR{3851835}

\bibitem{PZ-CUE}
E.~Paquette and O.~Zeitouni, \emph{The maximum of the {CUE} field},
  International Mathematics Research Notices \textbf{2018} (2018), no.~16,
  5028--5119.

\bibitem{selberg}
A.~Selberg, \emph{Contributions to the theory of the {R}iemann zeta-function},
  Archiv for Mathematik og Naturvidenskab \textbf{48} (1946), no.~5, 89--155.
  \MR{20594}

\bibitem{Sound-moments}
K.~Soundararajan, \emph{Moments of the {R}iemann zeta function}, Annals of
  Mathematics \textbf{170} (2009), no.~2, 981--993. \MR{2552116}

\bibitem{Sound-mobius}
\bysame, \emph{Partial sums of the {M}\"obius function}, Journal f{\"u}r die
  reine und angewandte Mathematik \textbf{2009} (2009), no.~631, 141--152.
  \MR{2542220}

\bibitem{SF-negative}
E.~Strahov and Y.~V. Fyodorov, \emph{Universal results for correlations of
  characteristic polynomials: {R}iemann-{H}ilbert approach}, Communications in
  Mathematical Physics \textbf{241} (2003), no.~2-3, 343--382. \MR{2013803}

\bibitem{Titchmarsh}
E.~C. Titchmarsh, \emph{A consequence of the {Riemann} hypothesis}, Journal of
  the London Mathematical Society \textbf{2} (1927), no.~4, 247--254.

\bibitem{vassaux}
Louis Vassaux, \emph{{Brownian} behaviour of the {Riemann} zeta function around
  the critical line}, arXiv:2505.07352, 2025.

\end{thebibliography}

\end{document}